\documentclass[11pt,reqno]{article}

\usepackage{pb-diagram}

\usepackage{pb-diagram}
\usepackage{hyperref}

\usepackage{amsfonts}
\usepackage{amscd}
\usepackage{color}
\usepackage{hyperref}
\usepackage{amsmath}
\usepackage{float}
\usepackage{amsthm}
\usepackage{amssymb}
\usepackage{mathrsfs}
\usepackage{amstext}
\usepackage{graphicx}
\usepackage{tikz}
\usepackage{verbatim}
\newcommand{\reallywidehat}[1]{%
	\savestack{\tmpbox}{\stretchto{%
			\scaleto{%
				\scalerel*[\widthof{\ensuremath{#1}}]{\kern-.6pt\bigwedge\kern-.6pt}%
				{\rule[-\textheight/2]{1ex}{\textheight}}
			}{\textheight}%
		}{0.5ex}}%
	\stackon[1pt]{#1}{\tmpbox}%
}

\newcommand{\boldj}{{\ensuremath{\boldsymbol{j}}}}

\newcommand{\boldr}{{\ensuremath{\boldsymbol{r}}}}
\newcommand{\boldx}{{\ensuremath{\boldsymbol{x}}}}

\newcommand{\boldF}{{\ensuremath{\boldsymbol{F}}}}

\numberwithin{equation}{section}
\theoremstyle{plain}
\newtheorem{satz}{Theorem}[section]
\newtheorem{defi}[satz]{Definition}
\newtheorem{cor}[satz]{Corollary}
\newtheorem{lem}[satz]{Lemma}
\newtheorem{prop}[satz]{Proposition}
\newtheorem{rem}[satz]{Remark}

\newcommand{\zb}[1]{\ensuremath{\boldsymbol{#1}}}

\newcommand{\re}{\ensuremath{\mathbb{R}}}\newcommand{\N}{\ensuremath{\mathbb{N}}}
\newcommand{\zz}{\ensuremath{\mathbb{Z}}}\newcommand{\C}{\ensuremath{\mathbb{C}}}
\newcommand{\T}{\ensuremath{\mathbb{T}^d}}\newcommand{\tor}{\ensuremath{\mathbb{T}}}

\newcommand{\Z}{{\ensuremath{\zz}^d}}

\newcommand{\R}{\ensuremath{{\re}^d}}

\newcommand{\id}{\ensuremath{\mbox{id}}}
\newcommand{\svrpw}{\ensuremath{S^{\zb r}_{p}W(\T)}}

\newcommand{\srptfu}{\ensuremath{F^{r}_{p,\theta}(\tor)}}
\newcommand{\srptbu}{\ensuremath{B^{r}_{p,\theta}(\tor)}}
\newcommand{\sabpqf}{\ensuremath{\srptf}}
\newcommand{\srptf}{\ensuremath{S^{\zb r}_{p,\theta}F(\T)}}

\newcommand{\srptb}{\ensuremath{S^{\zb r}_{p,\theta}B(\T)}}
\newcommand{\srpib}{\ensuremath{S^{\zb r}_{p,\infty}B(\T)}}
\newcommand{\srpw}{\ensuremath{S^{\zb r}_{p}W(\T)}}

\newcommand{\srpqb}{\ensuremath{S^{r}_{p,q}B}(\T)}

\newcommand{\rank}{{\rm rank \, }}
\newcommand{\supp}{{\rm supp \, }}

\newcommand{\sinc}{{\rm sinc}\,}

\newcommand{\bproof}{\begin{proof}}
\newcommand{\eproof}{\end{proof}}

\newlength{\fixboxwidth}
\newcommand{\be}{\begin{equation}}
\newcommand{\ee}{\end{equation}}
\newcommand{\beq}{\begin{eqnarray}}
\newcommand{\beqq}{\begin{eqnarray*}}
\newcommand{\eeq}{\end{eqnarray}}
\newcommand{\eeqq}{\end{eqnarray*}}

\def\Id{\mbox{Id}}

\begin{document}
\title{Optimal sampling recovery of mixed order Sobolev embeddings via discrete Littlewood-Paley
	type characterizations}

\author{Glenn Byrenheid$^a$\footnote{Corresponding author. Email: byrenheid.glenn@gmail.com}, Tino Ullrich$^a$ \\\\
	$^a$Hausdorff-Center for Mathematics/Institute for Numerical Simulation\\Endenicher Allee 62, 53115 Bonn, Germany\\
}

\date{\today}

\maketitle

\begin{abstract} In this paper we consider the $L_q$-approximation of multivariate periodic 
functions $f$ with $L_p$-bounded mixed derivative (difference). The (possibly non-linear) reconstruction algorithm is supposed to recover the
	function from function values, sampled on a discrete set of $n$ sampling nodes. The general performance is measured in terms of 
	(non-)linear sampling widths $\varrho_n$. We conduct a systematic analysis 
	of Smolyak type interpolation algorithms in the framework of Besov-Lizorkin-Triebel spaces of dominating mixed 
	smoothness based on specifically tailored discrete Littlewood-Paley type characterizations. 
	As a consequence, we provide sharp upper bounds for the asymptotic order of the (non-)linear sampling widths in
	various situations and close some gaps in the existing literature. 
	For example, in case $2\leq p<q<\infty$ and $r>1/p$ the linear sampling widths 
	$\varrho_n^{\text{lin}}(S^r_pW(\tor^d),L_q(\tor^d))$ and 
	$\varrho^{\text{lin}}_n(S^r_{p,\infty}B(\tor^d),L_q(\tor^d))$
	show the asymptotic behavior of the corresponding Gelfand $n$-widths, whereas
	in case $1 < p  < q \leq 2$ and $r>1/p$ the linear sampling widths match 
 	the corresponding linear widths. In the mentioned cases linear Smolyak interpolation based on 
 	univariate classical trigonometric interpolation turns out to be optimal.
\end{abstract}
Besov-Triebel-Lizorkin 
spaces of mixed smoothness, Smolyak algorithm, Gelfand $n$-widths, linear widths 


\section{Introduction}

This paper is concerned with the problem of optimal sampling recovery in classes of
multivariate functions. We consider the approximation of $d$-variate functions $f$ from classes with $L_p$-bounded mixed derivative (difference) in
$L_q$. We aim for the exact asymptotic order of the sampling widths which measure the minimal worst-case error for the (linear) sampling recovery problem with
$n$ points. To be more precise, we measure the performance of an optimal sampling algorithm with the {\em linear sampling widths}
\begin{equation}\label{samplingnumber}
	\varrho_n^{\text{lin}}(\boldF,Y) := \inf\limits_{X_n}\inf\limits_{\Psi_n}\sup\limits_{\|f|\boldF\| \leq 1}
	\Big\| f-\sum\limits_{i=1}^n f(\boldx^i)\psi_i(\cdot)\Big\|_Y\quad,\quad n\in \N\,,
\end{equation}
where the sampling nodes $X_n:=\{\boldx^i\}_{i=1}^n
\subset \tor^d$ and associated (continuous) functions $\Psi_n:=\{\psi_i\}_{i=1}^n$ determine a linear sampling recovery algorithm which 
is fixed in advance for a class $\mathbf{F}$ of multivariate functions  on the $d$-torus $\T$. Here the error is measured in $Y = L_q$. Let us emphasize that 
in \eqref{samplingnumber} we restrict to {\em linear} recovery algorithms, whereas we admit general recovery algorithms $\varphi:\C^n \to L_q$ in the definition of the 
(non-linear) {\em sampling widths}
\begin{equation}\label{sampling}
	\varrho_n(\mathbf{F},Y):=\inf_{\varphi,X_n}\sup_{\|f|\mathbf{F}\|\leq 1}\|f-\varphi(X_n(f))\|_Y,
\end{equation}
which is also denoted as the worst-case error for {\em standard information}, see \cite[Sect.\ 4.1]{NoWo08}. We are particularly interested in optimal point sets $X_n$ and corresponding recovery algorithms which yield the correct asymptotic order of \eqref{samplingnumber} and 
\eqref{sampling}. 

The interest in this topic
goes back to 1963 and started with Smolyak \cite{Smolyak63} who considered uniform approximation of 
multivariate functions with
mixed smoothness on the basis of function values. He used an influential construction which is nowadays known as
Smolyak's algorithm
\begin{equation}\label{f0}
	T_mf := \sum\limits_{\substack{\boldj \in \N_0^d\\|\boldj|_1\leq m}} (L_{j_1}-L_{j_1-1})\otimes...\otimes
	(L_{j_d}-L_{j_d-1})f\quad,\quad m\in
	\N\,,
\end{equation}
where the $(L_{j})_{j\in \N_0}$ represent univariate approximation operators (put $L_{-1}:=0$). For more historical comments see \cite[Sect.\ 5]{DTU16}.
When applied to a univariate sampling
(interpolation) scheme $(I_j)_j$ this construction yields a powerful sampling (interpolation) algorithm for the multivariate case taking 
points from a so-called {\it sparse grid}.

\subsection{New matching bounds for (non-)linear sampling recovery}

In this paper we investigate the optimal sampling recovery problem for the 
embedding 
\begin{equation}\label{emb1} 
	\Id:S^{\boldr}_{p,\theta}F(\tor^d) \to L_q(\tor^d)\,,
\end{equation}
where $0 < p< q \leq \infty$, $0<\theta\leq \infty$ and $\boldr>1/p$. Without loss of generality we assume 
\begin{equation}\label{eq:smoothnessvector}
	r = r_1=\ldots=r_{\mu} <r_{\mu+1}\leq \ldots\leq r_{d}<\infty\quad,\quad \mu\leq d.
\end{equation}
The main goal of this paper is to present a systematic framework towards new upper bounds for sampling recovery 
on Smolyak grids in Sobolev-Triebel-Lizorkin spaces $S^{\boldr}_{p,\theta}F(\tor^d)$ with dominating mixed smoothness. 
One of the main results in this paper is the sharp rate of convergence
\begin{equation}\label{f04_1}
	\varrho_n^{\text{lin}}(S^{\boldr}_{p,\theta}F(\tor^d), L_q(\tor^d)) \asymp \Big(\frac{(\log
		n)^{\mu-1}}{n}\Big)^{r-1/p+1/q}\quad,\quad n\in
	\N\,,
\end{equation}
whenever $1<p<q\leq 2$, $1\leq \theta \leq \infty$ or $2\leq p<q <\infty$, $2\leq \theta \leq \infty$ and $r>1/p$, see Corollary \ref{cor:sharp} below. 
Our main contribution is the constructive upper bound which holds true whenever 
$0<p<q<\infty$, $0<\theta \leq \infty$ and $r>1/p$. This is complemented by (see Theorem \ref{satz:linftysamplingmu})
\begin{equation}\label{f04_1b}
	\varrho_n^{\text{lin}}(S^{\boldr}_{p,\theta}F(\tor^d), L_{\infty}(\tor^d)) \lesssim \Big(\frac{(\log
		n)^{\mu-1}}{n}\Big)^{r-1/p}(\log n)^{(\mu-1)(1-1/p)_+}\quad,\quad n\in
	\N\,.
\end{equation}

The upper bounds are realized by
an explicit family of interpolation operators $T^L_m$ using $n \asymp 2^m m^{\mu-1}$ 
function values on a (anisotropic) Smolyak grid, where the parameter $L\in \N$ refers to the polynomial 
decay of the univariate fundamental interpolant ($L=1$ Dirichlet kernel, $L=2$ de la 
Vall{\'e}e Poussin type kernels, $L>2$ higher order kernels). It turned out that, for the sampling recovery problem \eqref{emb1} 
and the upper bounds in \eqref{f04_1}, \eqref{f04_1b}, \eqref{f04}, \eqref{f04_2}, the condition $L>1/q$ is sufficient, 
which means that Smolyak's algorithm \eqref{f0} applied to the classical 
trigonometric interpolation (based on the Dirichlet kernel \eqref{Dirichlet}) does the job. For $\theta = p = 2$ in \eqref{f04_1b}
this has been already observed in \cite[Rem.\ 6.12]{BDSU15}.

Let us emphasize the important special case ($\theta =2$), where it holds the identification $S^{\boldr}_{p,\theta}F(\tor^d) = S^{\boldr}_pW(\tor^d)$ 
with the space of functions with bounded mixed derivative. As a corollary from \eqref{f04_1}
we obtain the new sharp rate of convergence
\begin{equation}\label{f04}
	\varrho_n^{\text{lin}}(\srpw, L_q(\tor^d)) \asymp \Big(\frac{(\log
		n)^{\mu-1}}{n}\Big)^{r-1/p+1/q}\quad,\quad n\in
	\N\,,
\end{equation}
in case $1<p<q\leq2$ or $2\leq p<q<\infty$ and $\boldr>1/p$ which was unknown before. The upper bound is achieved with 
sparse grid interpolation based on classical univariate trigonometric interpolation. In particular, this improves on the bounds stated by Triebel in 
\cite[Thm.\ 4.15, Cor.\ 4.16]{Tr10} in case $r=1$. The parameter domain where 
\eqref{f04} holds is shown in the left diagram, where the parameters $\alpha$ and $\beta$ refer to the following rate of convergence 
$$
\varrho_n(\mathbf{F},L_q) \asymp \Big(\frac{(\log n)^{\mu-1}}{n}\Big)^{\alpha}(\log n)^{(\mu-1)\beta}. 
$$
The precise statements can be found in Sections \ref{sec:sharplin}, \ref{sec:sharpgelf}. 
\begin{figure}[H]
	\begin{minipage}{0.48\textwidth}
		\begin{tikzpicture}[scale=2.7]
		\filldraw[lightgray] (0.3,2.1) -- (1.1,2.1) -- (1.1,2.3) -- (0.3,2.3);
		
		\draw (-0.1,0.0) -- (0,0 );
		\draw[->] (2,0) -- (2.1,0.0) node[right] {$\frac{1}{p}$};
		\draw[->] (0,2) -- (0.0,2.1) node[above] {$\frac{1}{q}$};
		
		\draw[dashed](0,0) -- (0.0,2);
		\draw (0,-0.1) -- (0,0);
		\draw (1.0,0.03) -- (1.0,-0.03) node [below] {$\frac{1}{2}$};
		\draw (0.03,1) -- (-0.03,1) node [left] {$\frac{1}{2}$};
		
		\node at (1.6,2.2) {$\varrho_n^{\text{lin}}(S^{\zb r}_pW,L_q)$};
		\node at (0.7,2.2) {$\varrho_n(S^{\zb r}_pW,L_q)$};

		\draw[dashed] (0,2) -- (2,2);
		\draw (1,1) -- (2,1);

		\draw (1,0) -- (1,1);
		\draw (1,1) -- (2,1);
		\draw[dashed] (2,2) -- (2,0);
		
		\filldraw[lightgray] (0,0) -- (1,1) -- (1,0);
		\filldraw[lightgray] (0,0) -- (1,1) -- (1,0);
		\node at (.6,0.3) {\tiny  $\beta=0$};
		\node at (.6,0.15) {\tiny  $\alpha=r-\frac{1}{p}+\frac{1}{q}$};
		\node at (1.62,1.35) {\tiny  $\beta=0$};
		\node at (1.62,1.2) {\tiny  $\alpha=r-\frac{1}{p}+\frac{1}{q}$};
		
		\draw[dashed] (0,0) -- (2,2);
		\draw[dashed] (0,0) -- (2,0);
		
		
		\node at (0.65,1.45) {\huge ?};
		

		\draw (2,0.03) -- (2,-0.03) node [below] {$1$};
		\draw (0.03,2) -- (-0.03,2) node [left] {$1$};
		\node at (1.5,0.5) { $\lambda_n=o(\varrho_n)$};
		\end{tikzpicture}
	\end{minipage}
	\begin{minipage}{0.48\textwidth}
		\begin{tikzpicture}[scale=2.7]
		\draw[->] (0,2) -- (0.0,2.1) node[above] {$\frac{1}{q}$};
		\filldraw[lightgray] (0.15,2.1) -- (1.05,2.1) -- (1.05,2.3) -- (0.15,2.3);
		\draw[dashed](0,0) -- (0.0,2);
		\draw (0,-0.1) -- (0,0);
		\draw (-0.1,0.0) -- (0,0.0); 
		\draw[dashed] (0,0.0) -- (2,0.0);
		\draw[->] (2,0.0) -- (2.1,0.0) node[right] {$\frac{1}{p}$};
		\draw[->] (0.0,-.1) -- (0.0,2.1) node[above] {$\frac{1}{q}$};
		
		\draw (1.0,0.03) -- (1.0,-0.03) node [below] {$\frac{1}{2}$};
		\draw (0.03,1) -- (-0.03,1) node [left] {$\frac{1}{2}$};
		
		\node at (1.6,2.2) {$\varrho_n^{\text{lin}}(S^{\zb r}_{p,\infty}B,L_q)$};
		\node at (0.6,2.2) {$\varrho_n(S^{\zb r}_{p,\infty}B,L_q)$};
		
		\draw[dashed] (0,2) -- (2,2);
		\draw (1,1) -- (2,1);

		\draw (1,0) -- (1,1);
		\draw (1,1) -- (2,1);
		\draw[dashed] (2,2) -- (2,0);
		
		\filldraw[lightgray] (0,0) -- (1,1) -- (1,0);
		\node at (.6,0.3) {\tiny  $\beta=\frac{1}{q}$};
		\node at (.6,0.15) {\tiny  $\alpha=r-\frac{1}{p}+\frac{1}{q}$};
		\node at (1.62,1.35) {\tiny  $\beta=\frac{1}{q}$};
		\node at (1.62,1.2) {\tiny  $\alpha=r-\frac{1}{p}+\frac{1}{q}$};
		
		\draw[dashed] (0,0) -- (2,2);

		
		\node at (0.65,1.45) {\huge ?};
		

		\draw (2,0.03) -- (2,-0.03) node [below] {$1$};
		\draw (0.03,2) -- (-0.03,2) node [left] {$1$};
		\node at (1.5,0.5) {$\lambda_n=o(\varrho_n)$};
		\end{tikzpicture}
	\end{minipage}
	\caption{Linear and non-linear sampling widths.}
	\label{fig1}
\end{figure}
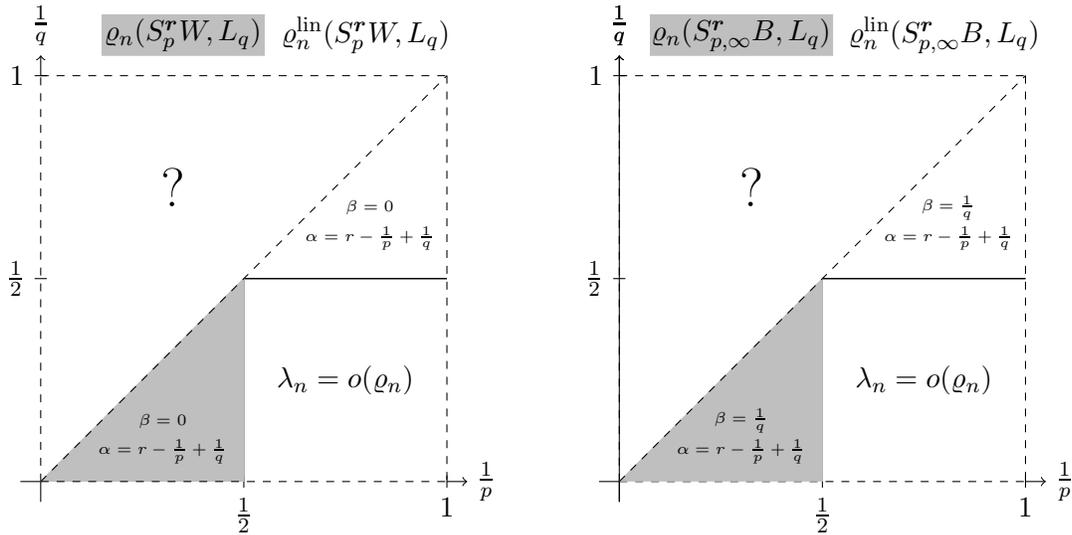
We mainly contributed to the upper bounds in the left figure. 
Most of the results illustrated in the right figure for H\"older-Nikolskij spaces $S^{\boldr}_{p,\infty}B(\tor^d)$ of mixed smoothness 
are well-known. Note that Open Problem 5.3 in \cite{DTU16} refers to the lower triangle in the right Figure \ref{fig1}.
A new approach of Malykhin and Ryutin \cite{MaRu16} settled this question for linear sampling recovery, cf. Corollary \ref{cor:sharpboundnikolskij}. We observed that their method also bounds the Gelfand widths from below, 
cf. Theorem \ref{satz:gelfandnikolskij} and Corollary \ref{cor:nonlinsampeqnonlinappr}. This yields 
the optimal order also for the non-linear sampling widths \eqref{sampling}, which is illustrated by the shaded lower triangles in Figure \ref{fig1}. 
The matching bound in the right upper triangle  for $\srpib$ were obtained by Dinh D\~ung \cite{Di91, Di92}. 
The necessary benchmark results on linear widths were obtained by Galeev \cite{Ga87, Ga96}, Romanyuk \cite{Ro01, Ro08}, and the recent paper by Malykhin and Ryutin \cite{MaRu16}. 
Note, that all sharp upper bounds can be realized by Smolyak type operators \eqref{f0}, i.e. via linear interpolation on sparse grids based on univariate Dirichlet interpolation, 
\eqref{I_j}, \eqref{Dirichlet}.
What concerns Besov spaces with bounded mixed difference $S^{\boldr}_{p,\theta}B(\tor^d)$ it is known that 
\begin{equation}\label{f04_2}
	\varrho_n^{\text{lin}}(S^{\boldr}_{p,\theta}B(\tor^d), L_q(\tor^d)) \asymp \Big(\frac{(\log
		n)^{\mu-1}}{n}\Big)^{r-1/p+1/q}(\log n)^{(\mu-1)(1/q-1/\theta)_+}\quad,\quad n\in
	\N\,,
\end{equation}
if $1<p<q\leq 2$, $1\leq \theta <\infty$ and $\boldr>1/p$, see \cite[Thm.\ 4.47, 5.15]{DTU16} and the references therein. 
With our method we can show the upper bound in case $0<p<q \leq \infty$, $0<\theta\leq \infty$ and $r>1/p$, see Theorem \ref{satz:lqsampanisonikolskij}, 
with interpolation operators providing $L>1/q$.
Comparing to \eqref{f04_1} there is an extra $\log$-term in \eqref{f04_2} in case of ``large'' $\theta>q$. There are still many open cases in this framework which actually 
lack the suitable lower bounds.

Let us refer to 
the works by Temlyakov \cite{Tem93complexity,Tem93} and the more recent papers Sickel, Ullrich \cite{Si06,SiUl07,Ul08}, Dinh D\~ung \cite{Di11,Di15}, \cite{BDSU15}, as well as
\cite{DTU16} and the references therein for upper bounds in case $p\geq q$ and the question-marked region. We emphasize that our technique allows to reproduce all those results, including the upper bound in \cite{Tem93complexity}, within a few lines of proof.

In Open Problem 18 in \cite[Sect.\ 4.2.4]{NoWo08} the authors conjecture the equivalence $\varrho_n^{\text{lin}}\asymp \varrho_n$ for all parameters $1< p,q<\infty$ 
in case of isotropic Sobolev spaces $W^{r}_p(\Omega)$ on bounded Lipschitz domains $\Omega$, see also Novak, Triebel \cite{NoTr05} and Heinrich \cite[Thms.\ 5.2, 5.3]{He09}. In the present paper
we consider mixed smoothness periodic Sobolev embeddings. In our case, the conjecture is true if $2\leq p< q <\infty$ for both Sobolev and H\"older-Nikolskij spaces, see the shaded regions in the 
diagrams above. In all other cases it is not known. Our results also support the above conjecture in the mixed smoothness setting. 
A similar statement as in \cite[Rem. 4.18]{NoWo08}, namely the equivalence $\lambda_n\asymp \varrho_n^{\text{lin}}$ if $p<q$ are on the same side of $2$ and 
$\lambda_n = o(\varrho_n^{\text{lin}})$ if $p<2<q$ is also true in our case.


\subsection{Discrete Littlewood-Paley type characterizations}

In Definition \ref{def:srptf} below we introduce Besov-Lizorkin-Triebel spaces of mixed smoothness 
via Fourier analytic building blocks $\delta_{\zb j}[f]$
generated by a dyadic decomposition of unity. 
In this paper we aim for function space characterizations where we replace the building blocks $\delta_{\zb j}[f]$
by the blocks
\begin{equation}\label{f2}
	q_{\boldj}[f] = (I_{j_1}-I_{j_1-1})\otimes...\otimes (I_{j_d}-I_{j_d-1})f\quad,\quad \boldj \in \N_0^d\,,
\end{equation}
used in the classical Smolyak algorithm (see \eqref{f0} above). Here the operators $(I_j)_j$ are univariate interpolation operators
\be\label{I_j}
I^L_j[f]=\sum_{u=0}^{2^j-1}f\Big(\frac{2\pi u}{2^j}\Big)K_{\pi,j}^L\Big(\cdot-\frac{2\pi u}{2^j}\Big)\,.
\ee
In a way we replace the usual convolution by a discrete one such that the building blocks
$q_{\boldj}^L[f]$ are constructed out of $\asymp 2^{|\boldj|_1}$ function values.
The parameter $L\in \N$ refers to the decay of the fundamental interpolant $K_{\pi,j}^L$, which 
represents a suitable trigonometric polynomial of degree $2^j$ and will be explicitly constructed 
in Section \ref{chapter:kernel}. In case $L=1$ we have the classical univariate nested 
trigonometric interpolation, where $K^1_{\pi,j} := 2^{-j}\mathcal{D}^1_{j}$ with $\mathcal{D}^1_0 \equiv 1$ and 
\be\label{Dirichlet}
\mathcal{D}^1_{j}(x) := \mathcal{D}_{2^{j-1}}(x)-e^{i2^{j-1}x} = e^{-i(2^{j-1}-1)x}\frac{e^{i2^jx} - 1}{e^{ix} - 1}\quad,\quad x\in \tor\,,
\ee
for $j\in \N$. The parameter $L=2$ refers to de la Vall{\'e}e Poussin type operators and $L>2$ to higher order kernels.  

We will prove the following
characterization for Sobolev spaces of mixed smoothness if $\boldr > \max\{1/p,1/2\}$ and $1<p<\infty$  
\begin{equation}\label{f02}
	\|f|\srpw\| \asymp \Big\|\Big(\sum\limits_{\boldj \in \N_0^d} 2^{2\boldj\cdot
		\boldr}|q_{\boldj}^L[f](\cdot)|^2\Big)^{1/2}\Big\|_p\,,
\end{equation}
where we may use $L \geq 1$, i.e. Dirichlet type characterizations are admitted. This result provides a powerful tool to
deal with Sobolev embeddings $S^r_pW(\tor^d)$ in $L_q(\tor^d)$\,. Analyzing Smolyak's algorithm \eqref{f0} in this
context has been a technical issue in the past. With \eqref{f02} and its counterpart for Triebel-Lizorkin spaces \eqref{f4} 
it becomes a straight-forward computation. 

For Triebel-Lizorkin spaces we obtain the representation (see Theorem \ref{maincharF})
\begin{equation}\label{f4}
	\|f|S^{\boldr}_{p,\theta}F(\tor^d)\| \asymp \Big\|\Big(\sum\limits_{\boldj \in \N_0^d} 2^{\boldr \cdot \boldj
		\theta}|q_{\boldj}^L[f](\cdot)|^{\theta}\Big)^{1/\theta}\Big\|_p
\end{equation}
in case $0<p<\infty$, $0<\theta\leq \infty$, $\boldr>\max\{1/p,1/\theta\}$ and $L>\max\{1/p,1/\theta\}$ (except in the 
case $\theta = \infty$ where $L \geq 2$)\,. 
Note, that we encounter the well-known (and infamous) condition
$\boldr, L >\max\{1/p,1/\theta\}$ (see also \eqref{f02} for $\theta = 2$), which is relevant if $p>\theta$. However, this condition is most likely optimal for the
respective sampling characterization. Note, that when replacing the classical smooth dyadic decomposition of unity (see Def.\ \ref{def:varphi}) 
in the definition of the spaces (see Def.\ \ref{def:srptf}) by a non-smooth variant like de la Vall{\'e}e Poussin means, we would encounter the 
same condition on $L$, which may not be improved as the recent findings in \cite{ChrSe07,SeUl15_1,SeUl15_2} indicate. In addition, 
Note, that in case of quasi-Banach spaces, where $\min\{p,\theta\} <1$, we need to use sampling kernels \eqref{I_j} of higher order $L$ as the 
condition $L>\max\{1/p,1/\theta\}$ indicates. Surprisingly, the de la Vall{\'e}e Poussin type kernels work well for the characterization \eqref{f4} if  $1/2<p,\theta<\infty$. 

Interestingly, the bounds \eqref{f04_1} as well as \eqref{f04} are valid for all $r>1/p$, in particular for the case of
``small smoothness'', i.e., $1/p<r\leq 1/\theta$ if $p>\theta$. Note, that in this range the characterization \eqref{f02} does not help. That is why 
we develop a complete theory of sampling representations also for more general Triebel-Lizorkin spaces in this paper, which allows to analyze the larger 
space $S^r_{p,\theta^{\ast}}F(\T)$ instead with $\theta^{\ast} > p$, which works since 
$\theta$ plays no role in \eqref{f04_1}. Approximation operators realizing the upper bounds in \eqref{f04_1}, \eqref{f04} and \eqref{f04_2} 
only depend on the integrability $q$ in the target space $L_q(\T)$, that is $L>1/q$ will be sufficient. Hence, we observe a
certain universality property of the Smolyak interpolation operators based on classical 
trigonometric interpolation which work well also in the
quasi-Banach situation whenever the integrability in the target space satisfies $q>1$, cf. Theorem \ref{satz:lqsampaniso} and 
\ref{satz:linftysamplingmu}. 

Finally, let us mention that in  
\cite{ByUl17} we provide similar characterizations for non-periodic Triebel-Lizorkin spaces using time local building blocks, i.e. the 
Faber-Schauder system. Recently, Dinh D\~ung \cite{Di16} extended this idea to prove B-spline quasi-interpolant characterizations in Sobolev spaces $\svrpw$.

{\bf Outline.} The paper is organized as follows. Section \ref{chapter:kernel} deals with the construction of
univariate interpolation operators and corresponding kernels necessary for \eqref{f2}. 
In Section \ref{chapter:functionspaces} we define and
discuss several characterizations of Besov-Triebel-Lizorkin spaces. In Section 4 we provide tools for estimating the norms of superpositions
of dyadic trigonometric polynomials in Besov-Triebel-Lizorkin spaces. Section 5 gives a proof for the sampling
representations, see \eqref{f4}, which are used in Section 6 to analyze Smolyak's
algorithm in this context. In the last  section we compare these results  with linear and Gelfand $n$-widths. Finally,
in the appendix we recall some basic facts and known results on maximal inequalities, Fourier transform, Fourier series,
exponential sums and $n$-widths.

{\bf Notation.} As usual $\N$ denotes the natural numbers, $\N_0:=\N\cup\{0\}$, $\zz$ denotes the integers, 
$\re$ the real numbers, and $\C$ the complex numbers. The letter $d$ is always
reserved for the underlying dimension in $\R, \Z$ etc. With $\T\;$  we denote the torus represented by the interval $[-\pi,\pi]^d$, where opposite points are identified. Elements
$\zb x,\zb y, \zb r \in \R$ are always typesetted in bold face.  We denote
with $\zb x\cdot \zb y$ the usual Euclidean inner
product in $\R$. For
$a\in \re$ we denote $a_+ := \max\{a,0\}$. 
For $0<p\leq \infty$ and $\zb x\in \R$ we denote $|\zb x|_p := (\sum_{i=1}^d
|x_i|^p)^{1/p}$ with the usual modification in the case $p=\infty$. By
$\zb x = (x_1,\ldots,x_d)>0$ we mean that each coordinate is positive. For $\zb j\in\N_0^d$ we use the notation $\zb 2^{\zb j}=(2^{j_1},\ldots,2^{j_d})$, $ 2^{\zb j}=2^{j_1}\cdot\ldots\cdot 2^{j_d}.$
If $X$ and $Y$ are two (quasi-)normed spaces, the (quasi-)norm
of an element $x$ in $X$ will be denoted by $\|x|X\|$. If $T:X\to Y$ is a continuous operator we write $T\in
\mathcal{L}(X,Y)$. The symbol $X \hookrightarrow Y$ indicates that the
identity operator from $X$ to $Y$ is continuous. For two sequences $(a_n)_{n=1}^{\infty},(b_n)_{n=1}^{\infty}\subset \re$  we
write $a_n \lesssim b_n$ if there exists a constant $c>0$ such that $a_n \leq
c\,b_n$ for all $n$. We will write $a_n \asymp b_n$ if $a_n \lesssim b_n$ and
$b_n \lesssim a_n$ and use the Landau symbol $(a_n)_n=o((b_n)_n) :\Longleftrightarrow \lim_{n\to\infty}a_n/b_n=0.$ In addition, we use the following notation $[d]:=\{1,\ldots,d\}$, $\Z(e):=\{\zb k\in\Z:k_i=0:\quad
i\notin e\}$, $\N^d_0(e):=\{\zb k\in\N_0^d:k_i=0:\quad i\notin e\}$ where $e\subset[d]$,
$\sigma_{p}:=\max\{0,\frac{1}{p}-1\}$, $\sigma_{p,\theta}:=\max\{0,\frac{1}{p}-1,\frac{1}{\theta}-1\}.$
For $\zb \ell,\zb j\in \Z$, $\zb s\in\R$ we use the notation $\zb \ell>a :\Longleftrightarrow \ell_i>a \mbox{ for all } i\in[d]$ and
$\zb \ell>\zb j :\Longleftrightarrow \ell_i>j_i \mbox{ for all } i\in[d].$ For $\Omega\subset\R$  the set of all
bounded and continuous functions $f:\Omega\to \C$ is denoted by $C(\Omega)$.  We denote by $L_p(\Omega)$,
$0<p\leq \infty$, the space of all measurable functions $f:\Omega\to \C$ where $\|f\|_p:=(\int_{\Omega} |f(\zb x)|^pd\zb x)^{1/p}$
is finite (with the usual modification if $p=\infty$).
For $f\in L_1(\R)$ and $\zb x,\zb \xi\in \R$ we define the Fourier transform and its inverse by
\be\mathcal{F}f(\zb \xi):=\frac{1}{(2\pi)^{\frac{d}{2}}}\int_{\R}f(\zb x)e^{-i\zb \xi \cdot \zb x}d\zb x\quad \mbox{and}\quad \mathcal{F}^{-1}f(\zb x):=\frac{1}{(2\pi)^{\frac{d}{2}}}\int_{\R}f(\zb x)e^{i\zb \xi \cdot \zb x}d\zb \xi\label{eq:fouriertransform}.\ee
For $f\in L_1(\T)$ the $\zb k$-th Fourier coefficient is defined by $\hat{f}(\zb k):=1/(2\pi)^d\int_{\T}f(\zb x)e^{-i\zb k\cdot \zb x}d\zb x.$ 
The convolution is always defined as $f\ast g(\zb x):=\int f(\zb y)g(\zb x- \zb y)\,d\zb y$.
\section{Frequency-limited fundamental interpolants} \label{chapter:kernel}
\subsection{Univariate fundamental interpolants}
In this section we construct univariate sampling operators of type \eqref{I_j} 
based on bandlimited kernels $K:\re \to \C$ with suitable decay. Here $K^L_{\pi,j}$ denotes the $2\pi$-periodization of $K^L(2^j(\cdot))$
which we will call fundamental interpolant. 
The following construction allows to
arrange any prescribed polynomial decay (of order $L$) of the kernel $K$, which is crucial for our analysis. In addition the operator
$I_{j}^L$ is supposed to reproduce trigonometric polynomials of a degree related to $\asymp 2^j$.
The sampling kernels we study are constructed from a finite product of dilated $\sinc$ functions.
As a starting point we define for $L\in\N,$
$$K^L(x):=\prod_{\ell=1}^{L} \sinc(2^{-\ell}x), \quad x\in\re,$$
with
$$\sinc(x):=\begin{cases}
\frac{\sin(x)}{x}&:\quad x\neq 0,\\
1&:\quad otherwise.
\end{cases}$$
The next step is a $2\pi$-periodization of dyadic dilations of $K^L(x)$ given by 
\be K_{\pi,j}^L(x):=\sum_{k=-\infty}^{\infty}K^L(2^j(x+2\pi k))\label{def:samplingkernel}.\ee
In case $L=1$ the summation in \eqref{def:samplingkernel} does not make sense. Instead we use a direct definition
\be\nonumber
K^1_{\pi,j}(x):=2^{-j}\mathcal{D}^1_j(x):=2^{-j}\sum_{k=-2^{j-1}}^{2^{j-1}-1}e^{ikx}.
\ee
This kernel represents an exception and requires some extra attention in the whole paper. It is a convenient modification of the classical Dirichlet kernel that provides a nested set of zeros as $j$ increases, cf. \cite[(2.6)]{DTU16}.
\noindent For $j\in \N_0$ we define the interpolation operator
$$I_{j}^L[f](x):=\sum_{u=-2^{j-1}}^{2^{j-1}-1}f\Big(\frac{2\pi u}{2^j}\Big)K^L_{\pi,j}\Big(x-\frac{2\pi u}{2^j}\Big),$$
where in case $j=0$ we put $I_{0}^L[f](x):=f(0)K^L_{\pi,0}(x).$
The kernel defined in \eqref{def:samplingkernel} consists of a sum with infinitely many summands. For practical reasons such a definition is not useful. For every fixed $L\in\N$ we can compute an explicit representation of the kernel. Beginning from the definition we obtain the following identity
\beqq
K_{\pi,j}^L(x)=\sum_{k=-\infty}^{\infty}\prod_{\ell=1}^{L}\frac{\sin(2^{j-\ell}(x+2\pi k))}{ 2^{j-\ell} (x+2\pi k)}.
\eeqq
Obviously, in case $x \mod \pi = 0$ we obtain
$$K^L_{\pi,j}(0)=K^L(0)=1.$$
In case $0<|x|<\pi$ an elementary calculation shows
\beq
K_{\pi,j}^L(x)&=&\sum_{k=-\infty}^{\infty}\prod_{\ell=1}^{L}\frac{\sin(2^{j-\ell}x)\cos(2^{\ell+j+1}\pi k)+\sin(2^{\ell+j+1}\pi k)\cos(2^{j-\ell}x)}{ 2^{j-\ell} (x+2\pi k)}\nonumber\\
&=&\sum_{k=-\infty}^{\infty}\prod_{\ell=1}^{L}\frac{\sin(2^{j-\ell}x)}{ 2^{j-\ell} (x+2\pi k)}\nonumber
=\frac{\sin(2^{j-1}x)\ldots\sin (2^{j-L}x)}{2^{jL}2^{-\frac{(L+1)L}{2}}}\sum_{k=-\infty}^{\infty}\frac{1}{(x+2\pi k)^{L}}.\label{eq:notclosedformula}
\eeq
Using the so-called Herglotz-trick (cf. \cite{AiZi01}) we find
\be \frac{1}{2}\cot\Big(\frac{x}{2}\Big)=\sum_{k=-\infty}^{\infty}\frac{1}{x+2\pi k}.\ee
Taking $ L-1$ derivatives yields
$$ \Big[\frac{1}{2}\cot\Big(\frac{\cdot}{2}\Big)\Big]^{(L-1)}(x)=(-1)^{L-1} (L-1)! \sum_{k=-\infty}^{\infty}\frac{1}{(x+2\pi k)^L}.$$
Computing $\Big[\frac{1}{2}\cot\Big(\frac{\cdot}{2}\Big)\Big]^{(L-1)}$ and inserting this identity in \eqref{eq:notclosedformula} gives us a closed representation of the kernel $K^L_{\pi,j}(x)$. For $L=2$ and $L=3$ we obtain the explicit representations
\be K^2_{\pi,j}(x)=\begin{cases}
	\frac{2\sin(2^{j-1}x)\sin(2^{j-2}x)}{2^{2j}\sin^2(\frac{x}{2})}&:\quad x\mod 2\pi \neq 0,\\
	1&:\quad otherwise,
\end{cases}\label{eq:k2pj}\ee
and
$$K^3_{\pi,j}(x)=\begin{cases}
8\frac{\sin(2^{j-1}x)\sin(2^{j-2}x)\sin(2^{j-3}x)\cos(\frac{x}{2})}{2^{3j}\sin^3(\frac{x}{2})}&:\quad x\mod 2\pi \neq 0,\\
1&:\quad otherwise.
\end{cases}$$
\begin{rem}\label{rem:ksupp}
	$K^L$, $L>1$, consists of products of dilated $\sinc$ functions. The convolution property of the Fourier transform yields
	\be K^L(x)=\prod_{\ell=1}^L \sinc(2^{-\ell}x)=\sqrt{2\pi}\mathcal{F}\Big[\chi_{[-2^{-1},2^{-1}]}\ast \ldots\ast 2^{L-1}\chi_{[-2^{-L},2^{-L}]}(\cdot)\Big](x)\label{eq:kmfouriertrans}\ee
	Altogether $\mathcal{F}K^L$ is a locally supported $L-2$ times continuously differentiable function fulfilling
	\beqq
	\mathcal{F}K^L(\xi)=\begin{cases}
		\sqrt{2\pi}&:|\xi|\leq \frac{1}{2^L},\\
		0&: |\xi|\geq 1-\frac{1}{2^L}.\\
	\end{cases}
	\eeqq
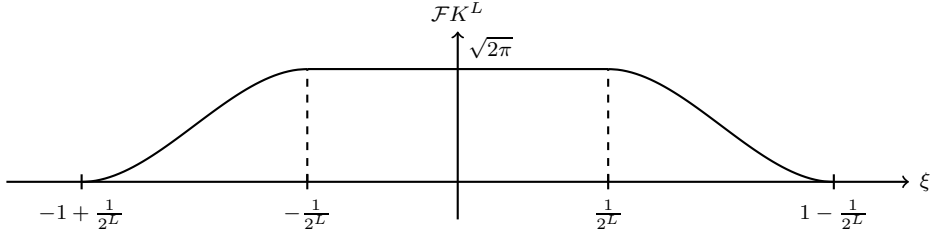
	\begin{figure}[H]
		\centering
		\begin{tikzpicture}
		\begin{scope}[thick,font=\scriptsize]
		\draw[->] (-6,0) -- (6,0) node[right] {$\xi$};
		\tikzset{swapaxes/.style = {rotate=90,yscale=-1}}
		\draw[->] (0,-0.5) -- (0,2) node[above] {$\mathcal{F}K^L$};
		\draw (-2,1.5) -- (2, 1.5);
		\node at (0,1.5) [above right] {$\sqrt{2\pi}$};
		\draw[dashed] (-2,0) -- (-2,1.5);
		\draw[dashed] (2,0) -- (2,1.5);
		\draw (-2,0.1) -- (-2,-0.1)
		node[below] {$-\frac{1}{2^L}$};
		\draw (2,0.1) -- (2,-0.1) node[below] {$\frac{1}{2^L}$};
		\draw (-5,0.1) -- (-5,-0.1) node[below] {$-1+\frac{1}{2^L}$};
		\draw (5,0.1) -- (5,-0.1) node[below] {$1-\frac{1}{2^L}$};
		\draw[scale=1,domain=-5:-2,smooth,variable=\x] plot ({\x},{-0.1111*\x*\x*\x-1.16667*\x*\x-3.3333*\x-1.38889});
		\draw[scale=1,domain=2:5,smooth,variable=\x] plot ({\x},{0.1111*\x*\x*\x-1.16667*\x*\x+3.3333*\x-1.38889});
		\end{scope}
		
		\end{tikzpicture}
	\end{figure}
\end{rem}
\begin{lem}\label{lem:fourierformula}
	Let $L \geq 1$, $j\in\N_0$ and $f\in C(\tor)$.
	\begin{enumerate}
		\item Then for $\ell\in\zz$  	\be\widehat{I_{j}^L[f]}(\ell)=\frac{1}{\sqrt{2\pi}}\mathcal{F}K^L\Big(\frac{\ell}{2^j}\Big)\sum_{u=-2^{j-1}}^{2^{j-1}-1}f\Big(\frac{2\pi u}{2^j}\Big)e^{-i\frac{2\pi u}{2^j}\ell}\label{eq:fseries}\ee
		holds true. 
		\item If additionally $\sum_{\ell\in\zz}|\widehat{f}(\ell)|<\infty$ is fulfilled. Then
		$$\widehat{I_{j}^L[f]}(\ell)=\frac{1}{\sqrt{2\pi}}\mathcal{F}K^L\Big(\frac{\ell}{2^j}\Big)\sum_{k\in\zz}\widehat{f}(\ell+2^jk)$$
		holds.
	\end{enumerate}
\end{lem}
\begin{proof}
	We compute the $\ell$-th Fourier coefficient of $f$ and obtain by the translation property the following identity
	\beqq\widehat{I_{j}^Lf}(\ell)&=&\sum_{u=-2^{j-1}}^{2^{j-1}-1}f\Big(\frac{2\pi u}{2^j}\Big)\widehat{K^{L}_{\pi,j}\Big(\cdot-\frac{2\pi u}{2^j}\Big)}(\ell)
	=\widehat{K^{L}_{\pi,j}}(\ell)\sum_{u=-2^{j-1}}^{2^{j-1}-1}f\Big(\frac{2\pi u}{2^j}\Big)e^{-i\frac{2\pi u}{2^j}\ell}.\eeqq
	Lemma \ref{lem:poissonsum} together with the dilation property of the Fourier transform yields
	\beqq
	\widehat{I_{j}^Lf}(\ell) &=&\frac{1}{\sqrt{2\pi}2^j}\mathcal{F}K^L\Big(\frac{\ell}{2^j}\Big)\sum_{u=-2^{j-1}}^{2^{j-1}-1}f\Big(\frac{2\pi u}{2^j}\Big)e^{-i\frac{2\pi u}{2^j}\ell}.\eeqq
	If the Fourier coefficients are absolutely summable we get
	\beqq
	\widehat{I_{j}^Lf}(\ell)	 &=&\frac{1}{\sqrt{2\pi}2^j}\mathcal{F}K^L\Big(\frac{\ell}{2^j}\Big)\sum_{u=-2^{j-1}}^{2^{j-1}-1}\Bigg(\sum_{k\in \zz}\widehat{f}(k)e^{ik\frac{2\pi u}{2^j}}\Bigg)e^{-i\frac{2\pi u}{2^j}\ell}.\eeqq
	Interchanging the order of summation yields
	\beqq
	\widehat{I_{j}^Lf}(\ell)	 &=&\frac{1}{\sqrt{2\pi}2^j}\mathcal{F}K^L\Big(\frac{\ell}{2^j}\Big)\sum_{k\in \zz}\widehat{f}(k)\sum_{u=-2^{j-1}}^{2^{j-1}-1}e^{i\frac{2\pi u}{2^j}(k-\ell)}.\eeqq
	The formula for geometric partial sums tells us
	$$\sum_{u=-2^{j-1}}^{2^{j-1}-1}e^{i\frac{2\pi u}{2^j}(k-\ell)}=\begin{cases}
	2^j&:\quad k-\ell \mod{2^j} =0\\
	0  &:\quad \mbox{otherwise}.
	\end{cases}$$
	Finally, we obtain
	\beqq
	\widehat{I_{j}^Lf}(\ell)	 &=&\frac{1}{\sqrt{2\pi}}\mathcal{F}K^L\Big(\frac{\ell}{2^j}\Big)\sum_{k\in \zz}\widehat{f}(\ell+2^{j}k).\eeqq
\end{proof}

\begin{defi}
	We define for 
	$j,L\in \N_0$ the dyadic blocks
	\be\mathcal{P}^L_j:=\Big\{k\in\zz: |k|\leq\frac{1}{2^L}2^j\Big\}.\label{unidyadicblock}\ee
	Additionally, we denote the set of trigonometric polynomials with frequencies in $\mathcal{P}^L_j$ by
	$$\mathcal{T}^L_j:=\mbox{span}\{e^{ikx}:k\in\mathcal{P}^L_j\}.$$
\end{defi}
\begin{cor}\label{lem:trigreprod}	Let $L\in\N$ and $f\in C(\tor)$.
	\begin{enumerate}
		\item 	Then it holds $I_{j}^L[f]\in \mathcal{T}^{0}_j$.
		\item If additionally  $f\in\mathcal{T}^L_j$ then
		$ I_{j}^L[f]=f$.
	\end{enumerate}

\end{cor}
\begin{proof}
	Assertion (i) is an easy consequence of \eqref{eq:fseries} together with the support properties 
	of $K^L$. For assertion (ii) we may use 
	\beqq\widehat{I_{j}^L[f]}(\ell)=\frac{1}{\sqrt{2\pi}}\mathcal{F}K^L\Big(\frac{\ell}{2^j}\Big)
	\sum\limits_{k\in \zz}\widehat{f}(\ell+2^jk)\eeqq
	which equals $\widehat{f}(\ell)$ for all $\ell$ if $f\in \mathcal{T}^L_j$.
\end{proof}

The next lemma provides the reason for calling $K^{L}_{\pi,j}$ a fundamental interpolant for the equidistant grid $\mathcal{G}^1_j:=\{\frac{-2\pi 2^{j-1}}{2^j},\ldots, \frac{2\pi (2^{j-1}-1)}{2^j}\}$.
\begin{lem}\label{lem:interpolation}
	Let $f\in C(\tor)$ and  $L\geq 1$. Then
	$$f\Big(\frac{2\pi u}{2^j}\Big)= I^L_j f\Big(\frac{2\pi u}{2^j}\Big)\quad,\quad 
	u\in \{-2^{j-1},\ldots,{2^{j-1}-1}\}\,.$$
\end{lem}
\begin{proof}
	Obviously, it is sufficient to proof
	
	$$K^L_{\pi,j}\Big(\frac{2\pi u}{2^j}\Big)=\delta_{0,u}.$$
	In case $L=1$ this is a trivial consequence of \eqref{Dirichlet}. 
	In case $L>1$ this we have according to our definition for $u\in \{-2^{j-1},\ldots,{2^{j-1}-1}\}$
	$$
	K^L_{\pi,j}\Big(\frac{2\pi u}{2^j}\Big)=\sum_{k=-\infty}^{\infty}K^L(2\pi (u+2^{j} k)) = \delta_{0,u}\,.
	$$
\end{proof}

\begin{lem}\label{lem:coredecreasing}Let $j\in \N_0$ and  $L>1$. Then there are constants $C,C^*>0$ (independent of $x$ and $j$) such that
	$$|K^L_{\pi,j}(x)|\leq C \min\Big\{\frac{1}{|2^jx|^L},1\Big\}\leq C^* \frac{1}{(1+2^j|x|)^L}$$
	holds for all $x\in [-\pi,\pi]$.
\end{lem}
\begin{proof}
	The second inequality of the chain is trivial. We prove the first one.
	Starting for $x\in[-\pi,\pi]$ the estimate with
	\beq
	|K^L_{\pi,j}(x)|&=&\Big|\sum_{k=-\infty}^{\infty}K^L(2^j(x+2\pi k))\Big|\nonumber
	\leq|K^L(2^jx)|+\sum_{|k|>0} |K^L(2^j(x+2\pi k))|\nonumber\\
	&\lesssim& \prod_{\ell=1}^L|\sinc(2^{j-\ell}x)|+\sum_{|k|>0} \frac{1}{2^{jL}|x+2\pi k|^L}
	\label{eq:bound1}.\eeq
	Clearly, the first summand is uniformly bounded.   Estimating the second summand in \eqref{eq:bound1} we use the fact that $|x|\leq \pi$ implies $|2\pi k+x|\geq |\pi k|$ for every integer $k\in \mathbb{Z}$ and obtain
	\beqq
	\sum_{|k|>0} \frac{1}{2^{jL}|x+2\pi k|^L}\leq \frac{1}{2^{jL}\pi^L} \sum_{|k|>0} \frac{1}{|k|^L},
	\eeqq
	which is known to be finite for $L\geq 2$. Using $|x|\leq\pi$ yields
	\beqq
	\sum_{|k|>0} \frac{1}{2^{jL}|x+2\pi k|^L}\lesssim  \frac{1}{2^{jL}|x|^L}.
	\eeqq 
	Considering again the first summand in \eqref{eq:bound1} gives
	\beqq
	\prod_{\ell=1}^L|\sinc(2^{j-\ell}x)|&\leq&\frac{1}{2^{(L+1)\frac{L}{2}}}\frac{1}{2^{jL}|x|^L},
	\eeqq
	which concludes the proof.
\end{proof}
\subsection{Multivariate interpolation}\label{sec:22}
Based on the univariate interpolation scheme from the previous subsection we are now able to define the building blocks used for the Smolyak algorithm, cf. \eqref{f0},
\be q^L_{\zb j}[f](\zb x):=\Big(\bigotimes_{i=1}^d\eta^L_{j_i}\Big)[f](\zb x)\label{eq:qkf}
\quad \mbox{with} \quad
\eta^L_{j_i}:=\begin{cases}
	I^L_{j_i}-I^L_{j_i-1}&:\quad j_i>0,\\
	I^L_{0} &:\quad j_i=0.
\end{cases}\ee
We may write $q^L_{\zb j}[f]$ as follows
\be\label{eq:qjasijinterpret}
q^L_{\zb j}[f] = \sum_{\zb b\in \{-1,0\}^d}\varepsilon_{\zb b}I^L_{{\zb j+\zb b}}[f]
\ee
with suitable signs $\varepsilon_{\zb b}$. The definition of the operators $I^L_{{\zb j+\zb b}}[f](\zb x)$ requires some more notation. 
\beqq \zb  x_{\zb u}^{\zb j}=\Big(x^{j_1}_{u_1},\ldots,x^{j_d}_{u_d}\Big)\quad,\quad \zb u \in \Z\,,\eeqq
where $x_u^j = 2\pi u/2^j$ for $u\in \zz$\,. For $\zb x \in \R$ let further
\be  A_{\zb j}(\zb x):=A_{j_1}(x_1)\times ... \times A_{j_d}(x_d)\label{Aj(x)} \ee
with $A_{j}(x) = \{u\in \zz~:~x_u^j \in [x-\pi, x+\pi)\}$ and put $A_{\zb j}:=A_{\zb j}(\zb 0)$. 
We further let 
$$K^L_{\pi^d, \zb j}:= \prod_{i=1}^d K^L_{\pi,j_i}(x_i)$$
and define the tensorized interpolation operator by \beqq I^L_{{ \zb j}}[f]=\sum_{\zb u\in A_{{\zb j}}}f(\zb x^{{\zb j}}_{\zb u})K^L_{\pi^d,\zb {\zb j}}(\zb x-\zb x^{{\zb j}}_u)\,.  \eeqq
\begin{lem}\label{lem:trigreprod_hyp} Let $\Delta\subset \N_0^d$ be a solid finite set meaning
	that 
	$\zb j \in \Delta$ and $\zb k \leq \zb j$ implies $\zb k \in \Delta$. Then
	$\sum_{\zb j \in \Delta}q^L_{\zb j}[f]$ reproduces trigonometric polynomials with frequencies in 
	$$
	\mathcal{H}^{L}_{\Delta}:=\bigcup_{\zb j\in \Delta}\mathcal{P}^{L}_{\zb j}\,,
	$$
	
\end{lem}
\begin{proof}
	We refer to \cite[Lem.\ 6.1]{BDSU15}.
\end{proof}
\begin{lem}\label{lem:triginterpol} Let $\Delta\subset \N_0^d$ be a solid finite set (i.e. $\zb k \leq \zb j$ and $\zb j \in \Delta$ 
	implies $\zb k \in \Delta$). Then
	$T^L_{\Delta}f:=\sum_{\zb j \in \Delta}q^L_{\zb j}[f]$ interpolates $f$ on the grid
	\be G^{d}_{\Delta}:=\bigcup_{\zb j\in \Delta}\{\zb x^{\zb j}_{\zb u}~:~\zb u \in A_{\zb j}\;\}\quad,\quad\ee
	that means
	$$f(\zb x)=T^L_{\Delta}f(\zb x)$$
	for all $\zb x\in G^{d}_{\Delta}.$
	\begin{proof}
		The interpolation property of the univariate operator $I^L_j$ in Lemma \ref{lem:interpolation} immediately 
		gives an interpolation property of the multivariate sampling operator $I^L_{(m_1,\ldots,m_d)}$ on 
		a ``full grid'' $\mathcal{G}^d_{\{\zb j \leq \zb m\}}$. Choosing $\zb m$ such that $\Delta \subset \{\zb j \leq \zb m\}$ and 
		arguing similar as in Lemma \cite[Lem.\ 4.3]{SiUl11} gives the result.
	\end{proof}
	
\end{lem}

\begin{defi}
	For $\zb j\in\N_0^d$ and $L\in\N$ we tensorize the dyadic blocks defined in \eqref{unidyadicblock} by
	\beqq\mathcal{P}^L_{\zb j}:= \mathcal{P}^L_{j_1} \cdot \ldots \cdot \mathcal{P}^L_{j_d},\eeqq
	and define the set of trigonometric polynomials with frequencies in $\mathcal{P}^L_{\zb j}$ by
	$$\mathcal{T}^L_{\zb j}:=\mbox{span } \{e^{i\zb k \cdot\zb x}:\zb k\in \mathcal{P}^L_{\zb j}\}.$$ 
\end{defi}
\begin{prop}\label{lem:qkfreprod}
	Let $L\in\N$ and $f\in\mathcal{T}^L_{\zb \ell}$ then
	$q^L_{\zb j}[f]\neq 0$
	implies $\zb \ell \geq \zb j $. 
\end{prop}
\begin{proof}
	The proof follows immediately from the definition of $q^L_{\zb j}[f]$ in \eqref{eq:qkf} and the univariate reproduction property in Corollary \ref{lem:trigreprod}.
\end{proof}
\section{Besov-Triebel-Lizorkin spaces of mixed smoothness}\label{chapter:functionspaces}
In this section we give the classical Fourier analytical definition of periodic Triebel-Lizorkin spaces $\srptf$ and Besov spaces $\srptb$ with dominating mixed smoothness. We start introducing vector-valued Lebesgue spaces.  
\begin{defi}We define for $0<p,\theta\leq \infty$ the spaces $L_p(\ell_{\theta}(\N^d_0))$ and $\ell_{\theta}(L_p(\T))$ as the space of all sequences of functions $(f_{\zb j})_{\zb j\in\N_0^d}\subset L_p(\T)$ 
	with finite (quasi)-norm
	$$\|f_{\zb j}|L_p(\tor^d,\ell_{\theta}(\N^d_0))\|:=\begin{cases}
	\Big\|\Big(\sum_{\zb j\in\N_0^d}|f_{\zb j}|^{\theta}\Big)^{\frac{1}{\theta}}\Big\|_p&:\quad 0<\theta<\infty,\\
	\Big\|\sup_{\zb j\in\N_0^d}|f_{\zb j}|\Big\|_p&:\quad \theta=\infty,
	\end{cases}$$
	and
	$$\|f_{\zb j}|\ell_{\theta}(L_p(\T))\|:=\begin{cases}
	\Big(\sum_{\zb j\in\N_0^d}\|f_{\zb j}\|_p^{\theta}\Big)^{\frac{1}{\theta}}&:\quad 0<\theta<\infty,\\
	\sup_{\zb j\in\N_0^d}\|f_{\zb j}\|_p&:\quad \theta=\infty,
	\end{cases}$$
	respectively.
\end{defi}
For $0<p,\theta\leq \infty$ the (quasi-)norms $\|\cdot|L_p(\tor^d,\ell_{\theta}(\N^d_0))\|$ and $\|\cdot|\ell_{\theta}(L_p(\T))\|$ fulfill a $\mu$-triangle inequality with $\mu=\min\{p,\theta,1\}.$
\begin{defi}\label{def:varphi}
	A system $\varphi=(\varphi_j)_{j=0}^{\infty}\subset C_0^{\infty}(\re)$ belongs to the class $\Phi(\re)$ if and only if
	\begin{enumerate}
		\item It exists $A>0$ such that $\supp \varphi_0 \subset [-A,A].$
		\item There are constants $0<B<C$, such that $\supp \varphi_j\subset \{\xi\in\re:\;B2^j\leq |\xi|\leq C2^j\}$.
		\item For all $r\in\N_0$ holds
		$$\sup_{\xi\in\re,j\in\N_0}2^{jr}|D^{r}\varphi_j(\xi)|\leq c_{r}<\infty\mbox{ and}$$
		\item $$\sum_{j=0}^{\infty}\varphi_j(\xi)=1.$$ 
	\end{enumerate}
\end{defi}
\noindent Because of (iv) in Definition \ref{def:varphi} we obtain the following decomposition of $f\in D'(\T)$. Here, $D'(\T)$ denotes the space of distributions on $\T$. For more details on periodic distributions we refer to \cite[Section 3.2]{ST87}. Let
\be \delta_{\zb j}[f](\zb x):=\sum_{\zb k\in\Z}\varphi_{j_1}(k_1)\cdot\ldots\cdot\varphi_{j_d}(k_d)\hat{f}(\zb k)e^{i \zb k\cdot \zb x} \label{eq:fourierdecomp}\ee
be a trigonometric polynomial. Then it holds
\be\label{decunity}
f=\sum_{\zb j\in \N_0^d}\delta_{\zb j}[f]
\ee
with convergence in $D'(\T)$. We introduce the function spaces $\srptf$ and $\srptb$ using these Fourier-analytic building blocks.
\begin{defi}\label{def:srptf}
	Let $\varphi=\{\varphi_j(x)\}_{j=0}^{\infty}\in \Phi(\re)$, and $\zb r\in\R$. Let further
	\begin{enumerate}
		\item 	$0<p<\infty$ and $0<\theta\leq \infty$. Then $$S^{\zb r}_{p,\theta}F(\T):=\Big\{f\in D'(\T):\|f|\srptf\|<\infty\Big\},
		$$
		where $\|f|\srptf\|:=\|2^{\zb r \cdot \zb j}\delta_{\zb j}[f]|L_p(\ell_{\theta}(\N^d_0))\|$.
		\item $0<p,\theta\leq\infty$. Then
		$$S^{\zb r}_{p,\theta}B(\T):=\Big\{f\in D'(\T):\|f|\srpqb\|<\infty\Big\},$$
		where $\|f|\srptb\|:=\|2^{\zb r \cdot \zb j}\delta_{\zb j}[f]|\ell_{\theta}(L_p(\N^d_0)).$
	\end{enumerate}
\end{defi}
\begin{rem}\label{rem:propfspace}
	\begin{enumerate}
		\item 	Different resolutions of unity $\varphi,\psi \in \Phi(\re)$ used in \eqref{eq:fourierdecomp} generate equivalent norms in $\srptf$ and $\srptb$, respectively.
		\item  When replacing the smooth dyadic decomposition of unity by a less regular window function
		(like de la Vall{\'e}e Poussin means) we encounter the same condition on the decay $L$ of the 
		convolution kernel as we observe for the sampling representations below. It is open whether this is optimal. However, taking
		\cite{ChrSe07,SeUl15_1,SeUl15_2} into account we strongly conjecture the optimality.
		\item 	It was proved in \cite[Thms.\ 3.8.1, 3.8.2]{TDiff06} that $B$-spaces defined via modulus 
		of continuity (bounded mixed difference, see for instance \cite{Di91,Di92,Tem93}) and 
		the Fourier-analytical approach in the present paper coincide in case $p\geq 1$ and $\boldr>0$. 
		For $p<1$ one does not even know the equivalence if $\boldr>\sigma_p$. In Theorem \cite[Thm.\ 3.8.3]{TDiff06} 
		we need $\boldr>1/p$, see also \cite[Thm.\ 2.3.4/2, Rem.\ 2.3.4/2]{ST87}. In case 
		$0<\boldr\leq \sigma_p$ both approaches may yield different spaces.\\
	\end{enumerate}
\end{rem}
In case $d=1$ the concepts of dominating mixed smoothness and isotropic smoothness coincide. We use the notation
$$F^r_{p,\theta}(\tor):=S^r_{p,\theta}F(\tor) \quad \mbox{and}\quad B^r_{p,\theta}(\tor):=S^r_{p,\theta}B(\tor). $$ 
In case $\theta=2$ with $1<p<\infty $ the space $\srptf$ coincides with the Sobolev space of dominating mixed smoothness $\srpw$ including $L_p(\T)$ if $\zb r=0$. $\srpw$ is classically normed by
\beqq \|f|\srpw\|':= \Big\|\sum_{\zb k\in\Z}\widehat{f}(\zb k) \prod_{i=1}^d(1+|k_i|^2)^{\frac{r_i}{2}}e^{i\zb k \cdot \zb x}\Big\|_p.\label{eq:classsobnorm}\eeqq
We state the following embedding results without proof. For a reference see \cite{ST87,Vy06} and \cite{HaVyb09}. For a complete history of the non-trivial embedding in Lemma \ref{lem:jf} we refer to \cite[Remark 3.8]{DTU16}.
\begin{lem} \label{lem:embeddings}
	\begin{enumerate}
		\item Let $0<p\leq\infty$ ($F$-case: $p<\infty$), $0<\theta\leq \infty$, $\zb r >\sigma_{p}$. Then
		$$ S^{\zb r}_{p,\theta}F(\T)\hookrightarrow L_{\max\{p,1\}}(\T)\quad\mbox{and}\quad
		S^{\zb r}_{p,\theta}B(\T)\hookrightarrow L_{\max\{p,1\}}(\T)\,,$$
		which means $S^{\zb r}_{p,\theta}F(\T)$ and $S^{\zb r}_{p,\theta}B(\T)$ consist of regular distributions that allow an interpretation as functions.
		\item Let $0<p\leq\infty$ ($F$-case: $p<\infty$), $0<\theta\leq \infty$, $\zb r >\frac{1}{p}$. Then
		$$ S^{\zb r}_{p,\theta}F(\T)\hookrightarrow C(\T)
		\quad\mbox{and}\quad
		S^{\zb r}_{p,\theta}B(\T)\hookrightarrow C(\T)\,,$$
		which means that we find in every equivalence class of $S^{\zb r}_{p,\theta}F(\T)$ and $S^{\zb r}_{p,\theta}B(\T)$ a unique continuous representative making discrete point evaluations possible.
		\item Let $0<q<p\leq \infty$ (F-case: $p<\infty$), $0<\theta\leq\infty$ and $\zb r\in\R$. Then
		$$\srptf \hookrightarrow S^{\zb r}_{q,\theta}F(\T)
		\quad\mbox{and}\quad
		\srptb \hookrightarrow S^{\zb r}_{q,\theta}B(\T).$$
		\item Let $0<p\leq \infty$ (F-case: $p<\infty$), $0<\theta_1<\theta_2\leq\infty$ and $\zb r\in\R$. Then
		$$S^{\zb r}_{p,\theta_1}F(\T) \hookrightarrow S^{\zb r}_{p,\theta_2}F(\T)
		\quad\mbox{and}\quad
		S^{\zb r}_{p,\theta_1}B(\T) \hookrightarrow S^{\zb r}_{p,\theta_2}B(\T).$$
		\item 
		Let $0<p<\infty$, $0<\theta\leq \infty$ and $\zb r\in\R$. Then
		
		$$ S^{\zb r}_{p,\min\{p,\theta\}}B(\T)\hookrightarrow	S^{\zb r}_{p,\theta}F(\T) \hookrightarrow 	S^{\zb r}_{p,\max\{p,\theta\}}B(\T).$$
		\item 	 Let $0<p\leq\infty$ ($F$: $p<\infty$), $0<\theta,\nu\leq \infty$ and $\zb r_1,\zb r_2\in\R$ with $\zb r_1>\zb r_2$ 
		Then
		$$ S^{\zb r_1}_{p,\theta}F(\T)\hookrightarrow	S^{\zb r_2}_{p,\nu}F(\T)\quad\mbox{and}\quad S^{\zb r_1}_{p,\theta}B(\T)\hookrightarrow	S^{\zb r_2}_{p,\nu}B(\T).$$
		\item 	Let $0<p<q<\infty$, $0<\theta,\nu\leq \infty$ and $\zb r_1,\zb r_2\in\R$ with $\zb r_1>\zb r_2$ fulfilling
		$$\zb r_1-\frac{1}{p}=\zb r_2-\frac{1}{q}.$$
		Then
		$$ S^{\zb r_1}_{p,\theta}F(\T)\hookrightarrow	S^{\zb r_2}_{q,\nu}F(\T)\quad\mbox{and}\quad S^{\zb r_1}_{p,\theta}B(\T)\hookrightarrow	S^{\zb r_2}_{q,\theta}B(\T).$$
	\end{enumerate}
\end{lem}
Observe that, in contrast to the diagonal Besov embedding in Lemma \ref{lem:embeddings}, (vi), the fine index $\theta$ and $\nu$ play no role for the $F$-case.
\begin{lem}[Jawerth-Franke embedding]\label{lem:jf}
	Let $0<p<q\leq\infty$, $0<\theta\leq \infty$, $\zb r_1,\zb r_2\in\R$ such that
	$$\zb r_1-\frac{1}{p}=\zb r_2-\frac{1}{q}$$
	is fulfilled.
	\begin{enumerate}
		
		\item 	Then	$$S^{\zb r_1}_{p,\theta}F(\T)\hookrightarrow S^{\zb r_2}_{q,p}B(\T).$$ 
		\item If additionally $q<\infty$ then
		$$S^{\zb r_1}_{p,q}B(\T)\hookrightarrow S^{\zb r_2}_{q,\theta}F(\T).$$ 
	\end{enumerate}
\end{lem}
\begin{defi}
	For a univariate function $f:\tor\to\C$ we introduce the following difference operator. Let $h\in\re$ and $m\in\N$. Then for $x\in\T$
	\be \Delta_h^mf(x):=\sum_{k=0}^m (-1)^{m-k}\binom{m}{k}f(x+kh).\label{eq:diffi}\ee
\end{defi}
\begin{defi} For a multivariate function $f:\T\to \C$ and
	$m\in\N$, $h\in\re$, $i\in[d]$ we denote with $\Delta^{m,i}_hf(\zb x)$ the operator from \eqref{eq:diffi} applied to the $i$-th direction.
	This allows us to define for
	$e\subset[d]$, $\zb m\in \N^d$ and $\zb h\in \R$ the operator
	$$\Delta_{\zb h}^{\zb m,e}f(\zb x):=\begin{cases}
	\Big(\prod_{i\in e}\Delta^{m_i,i}_{h_i}\Big) f(\zb x),&:\quad|e|>0,\\
	f(\zb x)&:\quad \mbox{otherwise}.
	\end{cases} $$
	
\end{defi}

\begin{satz}\label{satz:differences}
	Let $0<p<\infty$, $0<\theta\leq \infty$, $\zb r\in\R$ and $\zb m\in\N^d$ such that $\sigma_{p,\theta} <\zb r<\zb m$ is fulfilled. Then 
	$$\|f|\srptf\|\asymp \sum_{e\subset [d]}\|f|\srptf\|_{e,\zb m}$$
	holds with
	$$\|f|\srptf\|_{e,\zb m}:=
	\Bigg\|\Big[\sum_{\zb j\in\N_0^d(e)}2^{\theta\zb r \cdot \zb j}\Big(\Big(\prod_{i\in e}2^{j_i}\Big)\int_{\substack{|h_i|\leq\tiny 2^{-j_i}\\i\in[d]}}|\Delta^{\zb m,e}_{\zb h}f(\cdot)|d\zb h\Big)^{\theta}\Big]^{\frac{1}{\theta}}\Bigg\|_p$$
	and the usual modification in case $\theta=\infty$.
\end{satz}
\begin{proof}
	We refer to \cite[Theorem 3.7]{NUU15}. There the case for constant smoothness vector $\zb r=(r,\ldots,r)$ has been considered. The necessary modifications are straight forward. 
\end{proof}
\begin{satz}\label{satz:bdifferences}
	Let $0<p,\theta\leq \infty$, $\zb r\in\R$ and $\zb m\in\N^d$ such that $\sigma_{p} <\zb r<\zb m$ is fulfilled. Then 
	$$\|f|\srptb\|\asymp \sum_{e\subset [d]}\|f|\srptb\|_{e,\zb m}$$
	holds with
	$$\|f|\srptb\|_{e,\zb m}:=
	\Bigg[\sum_{\zb j\in\N_0^d(e)}2^{\theta\zb r \cdot \zb j}\Big\|\Big(\prod_{i\in e}2^{j_i}\Big)\int_{\substack{|h_i|\leq\tiny 2^{-j_i}\\i\in [d]}}|\Delta^{\zb m,e}_{\zb h}f(\cdot)|d\zb h\Big\|_p^{\theta}\Bigg]^{\frac{1}{\theta}} $$
	and the usual modification in case $\theta=\infty$.
\end{satz}
\begin{proof}
	We refer to \cite[Theorem 3.7.1 and Remark 3.7.1]{TDiff06}. There the outer sum is an integral. By discretizing this into dyadic parts one obtains the form stated above.   
\end{proof}
\section{Sums of trigonometric polynomials}\label{superposition}
In this section we want to estimate the norm of a superposition of trigonometric polynomials $$f=\sum_{\zb j\in\N_0^d}f_{\zb j}$$ 
where $f_{\zb j}$ are trigonometric polynomials of degree $\asymp \zb 2^{\zb j}$. In contrast to the usual Littlewood-Paley 
building blocks $\delta_{\zb j}[f]$ which are `almost` orthogonal, we only need to restrict the degree of the polynomial in the sequel. 

As a main tool we introduce the following componentwise variant of the Hardy-Littlewood maximal operator, see \cite[(1.14),(1.15)]{Vy06}, \cite[(10)]{TDiff06} 
and the references therein. 
\begin{defi}\label{dirHL}
	Let $i\in [d]$ and $f\in L_1^{loc}(\T)$ then we define the Hardy-Littlewood maximal operator in the $i$-th direction as
	\begin{equation}\label{HLdir}
		M_if(\zb x):=\sup_{t>0}\frac{1}{2t}\int_{-t}^{t}|f(x_1,\ldots,x_{i-1},x_i+y,x_{i+1},\ldots,x_d)|dy.
	\end{equation}
	
\end{defi}
There is a corresponding variant of the Fefferman-Stein theorem, see \cite[Thm.\ 4.1.2]{TDiff06} and the references therein. 
\begin{satz}\label{satz:feffstein2}
	Let $1<p,q<\infty$ and $(f_{\zb k})_{\zb k}\subset L_p(\tor^d,\ell_{\theta})$ and $i\in [d]$. Then we have
	$$\|M_if_{\zb k}|L_p(\tor^d,\ell_{\theta})\|\lesssim \|f_{\zb k}|L_p(\tor^d,\ell_{\theta})\|.$$
\end{satz}
\begin{defi}[Mixed Peetre maximal function]\label{def:peetremax}
	Let $a>0$ and $\zb b>0$ then we define for $f\in C(\T)$ 
	$$P_{\zb b,a}f(\zb x):=\sup_{\zb y\in \R}\frac{|f(\zb x+\zb y)|}{(1+b_1|y_1|)^{a}\ldots (1+b_d|y_d|)^{a}}.$$
\end{defi}
The mixed Peetre maximal function can be pointwise estimated by an iteration of \eqref{HLdir}, see 
\cite[Lem.\ 3.1.1]{TDiff06}, which results in the following
maximal inequalities for the mixed Peetre maximal function. 
\begin{satz}\label{satz:peetremaximalineqnonvec}
	Let $0<p\leq\infty$ and $f$ be a trigonometric polynomial
	with
	$$f=\sum_{\substack{|k_i|\leq b_i\\i=1,\ldots,d}}\hat{f}(\zb k)e^{ i \zb k \cdot \zb x}$$
	and $a>\frac{1}{p}$.
	Then there is a constant $C>0$ (independent of $f$ and $\zb b$) such that
	$$\|P_{\zb b,a}f\|_p\leq C \|f\|_p$$
	holds.
\end{satz}
\begin{proof}
	We refer to \cite[Thm.\ 4.1.3]{TDiff06}.
\end{proof}
\begin{satz}\label{satz:peetremaximalineq}
	Let $0<p<\infty$, $0<\theta\leq \infty$ and $(f_{\zb j})_{\zb j\in \N^d_0}$ be a sequence of trigonometric polynomials
	with
	$$f_{\zb j}=\sum_{\substack{|k_i|\leq b^{\zb j}_i\\i=1,\ldots,d}}\hat{f}(\zb k)e^{ i \zb k \cdot \zb x}$$
	and $a>\max\{\frac{1}{p},\frac{1}{\theta}\}$. 	Then there is a constant $C>0$ (independent of $f$ and ${\zb
		b}_{\zb j}$) such that
	$$\|P_{{\zb b}_{\zb j},a}f_{\zb j}|L_p(\ell_{\theta})\|\leq C \|f_{\zb j}|L_p(\ell_{\theta})\|$$
	holds.
\end{satz}
\begin{proof}
	We refer to \cite[Thm.\ 4.1.3]{TDiff06}.
\end{proof}
\noindent Let us now state the main result of this subsection.
\begin{satz}\label{satz:bandlimrep1}
	Let $0<p<\infty$, $0<\theta\leq \infty$, $\zb r\in \R$ with $\zb r>\sigma_{p,\theta}$ and $(f_{\zb j})_{\zb j \in
		\N_0^d}$ such that $f_{\zb j}\in \mathcal{T}^0_{\zb j}$ and $\big\|2^{\zb r \cdot \zb j}f_{\zb
		j}\big|L_p(\ell_{\theta})\big\|<\infty$.
	Then
	\begin{enumerate}
		\item $\sum_{\zb j\in\N_0^d}f_{\zb j}$ converges unconditionally in $\srptf$ if $\theta<\infty$ and in every
		$S^{\bf {\tilde{r}}}_{p,\nu}F(\T)$ with $0<\nu\leq \infty$ and $\zb {\tilde{r}}< \zb r.$
		\item There is a constant $C>0$ (independent of $f$) such that
		$$ \Big\|\sum_{\zb j\in\N_0^d}f_{\zb j}\Big|\srptf\Big\|\leq C \big\|2^{\zb r \cdot \zb j}f_{\zb j}\big|L_p(\ell_{\theta})\big\|$$
		holds.
	\end{enumerate}
\end{satz}
\begin{proof}
	{\em Step 1.} We assume the unconditional convergence of $\sum_{\zb \ell\in\N_0^d}f_{\zb \ell}$ in $\srptf$ (or in case  $\theta =\infty$ at least in $S^{\bf {\tilde{r}}}_{p,\nu}F(\T)$) and prove the inequality $$\Big\|\sum_{\zb \ell\in\N_0^d}f_{\zb \ell}\Big|S^r_{p,\theta}F(\T)\Big\|\lesssim \big\|2^{\zb r \cdot \zb j}f_{\zb j}\big|L_p(\ell_{\theta}(\N_0^d))\big\|.$$ We mimic Step 1 of the proof of \cite[Theorem 3.4.1]{TDiff06}. This is rather technical in the multivariate situation. For that reason we give a proof for the univariate situation first. Later we explain the necessary modifications for the multivariate situation.
	We prove
	$$\|f|F^r_{p,\theta}(\tor)\|\lesssim \big\|2^{ r j}f_j\big|L_p(\ell_{\theta}(\N_0))\big\|$$
	by using methods from difference characterization of Triebel-Lizorkin spaces. We start by switching to the difference norm in $F^r_{p,\theta}(\tor)$ with $m>r$
	\be
	\Big\|\sum_{\ell\in\N_0}f_{\ell}\Big|\srptfu\Big\|\asymp\Big\|\sum_{\ell\in\N_0}f_{\ell}\Big\|_p+\Big\|\Big[\sum_{j=0}^{\infty}2^{\theta j r}\Big(2^j\int_{-2^{-j}}^{2^{-j}}\Big|\Delta^m_h \Big[\sum_{\ell\in\N_0}f_{\ell}\Big]\Big|dh\Big)^{\theta}\Big]^{\frac{1}{\theta}}\Big\|_p\label{eq:h2}.
	\ee
	First we estimate the $L_p$-norm of $f$ and obtain trivially using either H\"older's inequality (in case $\theta\geq 1$) or the embedding $\ell_{\theta}\hookrightarrow \ell_{1}$ (in case $0<\theta<1$) the estimate 
	$$\Big\|\sum_{\ell\in\N_0}f_{\ell}\Big\|_p\lesssim \Big\|\Big(\sum_{j\in\N_0}2^{\theta r j}|f_j|^{\theta}\Big)^{\frac{1}{\theta}}\Big\|_p. $$
	Let  $a>0$ be a positive real number such that $a>\max\{\frac{1}{p},\frac{1}{\theta}\}$ is fulfilled. Additionally choose in case $\min\{p,\theta\}\leq 1$
	\be 0<\lambda< \min\{p,\theta\} \label{eq:lambda}\ee such that 
	\be r>(1-\lambda)a>\sigma_{p,\theta}.\label{eq:lambda2}\ee
	This is possible since
	\beqq
	(1-\lambda)a&>&(1-\lambda)\max\Big\{\frac{1}{p},\frac{1}{\theta}\Big\}\\
	&\geq&(1-\min\{p,\theta,1\})\max\Big\{\frac{1}{p},\frac{1}{\theta}\Big\}=\sigma_{p,\theta}.
	\eeqq
	In case $\min\{p,\theta\}>1$ we simply choose $\lambda=1$.
	Fix $j\in\N_0$ and use the identity
	$$\sum_{\ell\in\N_0}f_{\ell}=\sum_{\ell\in\zz}f_{j+\ell} $$
	with
	$f_{j+\ell}=0$
	for $j+\ell<0$. The unconditional convergence of $\sum_{\ell\in\zz}f_{j+\ell}$ in $\srptf$ implies (by Lemma \ref{lem:embeddings}) an unconditional convergence also in $L_1(\T)$. Therefore we can estimate the integral means as follows
	\be
	2^j\int_{-2^{-j}}^{2^{-j}}\Big|\Delta^m_h \Big[\sum_{\ell\in\N_0}f_{j+\ell}\Big]\Big|dh\leq \sum_{\ell\in \zz} 2^j\int_{-2^{-j}}^{2^{-j}}|\Delta^m_h f_{j+\ell}(x)|dh.\label{eq:est1n}
	\ee
	We split the sum over $\ell$
	\beq
	\sum_{\ell\in \zz} 2^j\int_{-2^{-j}}^{2^{-j}}|\Delta^m_h f_{j+\ell}(x)|dh=\sum_{\ell\geq 0} 2^j\int_{-2^{-j}}^{2^{-j}}|\Delta^m_h f_{j+\ell}(x)|dh+\sum_{\ell<0} 2^j\int_{-2^{-j}}^{2^{-j}}|\Delta^m_h f_{j+\ell}(x)|dh
	\label{eq:decomp1}\eeq
	and prove 
	\be
	2^j\int_{-2^{-j}}^{2^{-j}}|\Delta^m_h f_{j+\ell}(x)|dh\lesssim \begin{cases}
		2^{\ell m}P_{2^{j+\ell},a}f_{j+\ell}&:\quad\ell\geq 0,\\
		2^{(1-\lambda)\ell a}[P_{2^{j+\ell},a}f_{j+\ell}]^{1-\lambda}M|f_{j+\ell}|^{\lambda}&:\quad\ell<0.
	\end{cases}\label{eq:esth1}
	\ee
	First we prove the case $\ell>0$. Applying Lemma \ref{lem:diffversuspeetre} immediately gives
	\beqq
	2^j\int_{-2^{-j}}^{2^{-j}}|\Delta^m_h f_{j+\ell}(x)|dh&\lesssim & 2^{\ell m}P_{2^{j+\ell},a}(x).
	\eeqq In case $\ell<0$ with $\lambda<1$ we estimate as follows
	\beqq
	2^j\int_{-2^{-j}}^{2^{-j}}|\Delta^m_h f_{j+\ell}(x)|dh&\lesssim&2^j\int_{-2^{-j}}^{2^{-j}}|\Delta^m_h f_{j+\ell}(x)|^{\lambda}|\Delta^m_h f_{j+\ell}(x)|^{1-\lambda}dh.
	\eeqq
	Applying Lemma \ref{lem:diffversuspeetre} to the second factor yields
	\beqq
	2^j\int_{-2^{-j}}^{2^{-j}}|\Delta^m_h f_{j+\ell}(x)|dh&\lesssim&2^{\ell (1-\lambda)a}[P_{2^{j+\ell},a}f_{j+\ell}(x)]^{1-\lambda}2^j\int_{-2^{-j}}^{2^{-j}}|\Delta^m_h f_{j+\ell}(x)|^{\lambda}dh\nonumber\\
	&\lesssim&  2^{\ell (1-\lambda)a}[P_{2^{j+\ell},a}f_{j+\ell}(x)]^{1-\lambda}M|f_{j+\ell}|^{\lambda}(x).
	\eeqq
	Attention in case $\min\{p,\theta\}>1$ with $\lambda=1$ the estimate in case $\ell<0$ simplifies to the Hardy-Littlewood maximal function of $|f_{j+\ell}|$.
	Inserting the decomposition in \eqref{eq:decomp1} together with the estimates obtained in \eqref{eq:esth1}
	into the last term on the right hand side of \eqref{eq:h2} then we obtain by $\mu$-triangle inequality in $L_p(\ell_{\theta}(\N))$ with $\mu:=\min\{p,\theta,1\}$
	\beq
	&&\Big\|\Big[\sum_{j=0}^{\infty}2^{\theta j r}\Big(2^j\int_{-2^{-j}}^{2^{-j}}\Big|\Delta^m_h \Big[\sum_{\ell\in\N_0}f_{j+\ell}\Big]\Big|dh\Big)^{\theta}\Big]^{\frac{1}{\theta}}\Big\|_p\nonumber\\
	&&\lesssim \Big[\sum_{\ell\geq 0}2^{\mu\ell (m-r)}\|2^{(j+\ell)r}P_{2^{j+\ell},a}f_{j+\ell}(x)|L_p(\ell_{\theta}(\N_0))\|^{\mu}\nonumber
	+\sum_{\ell<0}2^{\mu \ell [a(1-\lambda)-r]}\|2^{(j+\ell)r}[P_{2^{j+\ell},a}f_{2^{j+\ell}}(x)]^{1-\lambda}(x)\\&&\quad\times M|f_{j+\ell}|^{\lambda}(x)|L_p(\ell_{\theta}(\N_0))\|^{\mu}\Big]^{\frac{1}{\mu}}.\nonumber\\\label{eq:h7}
	\eeq
	To estimate the first summand we apply Theorem \ref{satz:peetremaximalineq}, which gives
	$$
	\|2^{(j+\ell)r}P_{2^{j+\ell},a}f_{j+\ell}(x)|L_p(\ell_{\theta}(\N_0))\|\lesssim \|2^{(j+\ell)r}f_{j+\ell}(x)|L_p(\ell_{\theta}(\N_0))\|.
	$$
	An index shift yields
	\be
	\|2^{(j+\ell)r}P_{2^{j+\ell},a}f_{j+\ell}(x)|L_p(\ell_{\theta}(\N_0))\|\lesssim \|2^{jr}f_{j}(x)|L_p(\ell_{\theta}(\N_0))\|.\label{eq:h13}
	\ee
	In case $\min\{p,\theta\}\leq 1$ with $\lambda<1$ we apply to the norm expression in \eqref{eq:h7} H\"older's inequality with $\frac{1}{1-\lambda}$,  $\frac{1}{\lambda}$ twice and obtain
	\beq
	\|2^{(j+\ell)r}[P_{2^{j+\ell},a}f_{j+\ell}(x)]^{1-\lambda}M|f_{j+\ell}|^{\lambda}(x)|L_p(\ell_{\theta}(\N_0))\|&\leq&
	\|2^{(j+\ell)r}P_{2^{j+\ell},a}f_{j+\ell}(x)|L_p(\ell_{\theta}(\N_0))\|^{1-\lambda}\nonumber\\
	&&\times\|2^{(j+\ell)r}(M|f_{j+\ell}|^{\lambda}(x))^{\frac{1}{\lambda}}|L_p(\ell_{\theta}(\N_0))\|^{\lambda}.\nonumber\\
	\label{eq:h9}
	\eeq
	We skip this in case $\lambda=1$.
	Considering the factors in \eqref{eq:h9} separately we obtain by applying Theorem \ref{satz:peetremaximalineq}
	\be
	\|2^{(j+\ell)r}P_{2^{j+\ell},a}f_{j+\ell}(x)|L_p(\ell_{\theta}(\N_0))\|\lesssim \|2^{(j+\ell)r}f_{j+\ell}(x)|L_p(\ell_{\theta}(\N_0))\|. \label{eq:h10}
	\ee
	For the second factor we rewrite the $L_p(\ell_{\theta}(\N_0))$-norm as a $L_{\frac{p}{\lambda}}(\ell_{\frac{\theta}{\lambda}}(\N_0))$-norm. This allows for applying Theorem \ref{thm:feffermanstein}. 
	\beq
	\|2^{(j+\ell)r}(M|f_{j+\ell}|^{\lambda}(x))^{\frac{1}{\lambda}}|L_p(\ell_{\theta}(\N_0))\|&=& \|2^{(j+\ell)r\lambda}M|f_{j+\ell}|^{\lambda}(x)|L_{\frac{p}{\lambda}}(\ell_{\frac{\theta}{\lambda}}(\N_0))\|^{\frac{1}{\lambda}}\nonumber\\
	&\lesssim&\|2^{(j+\ell)r\lambda}|f_{j+\ell}(x)|^{\lambda}|L_{\frac{p}{\lambda}}(\ell_{\frac{\theta}{\lambda}}(\N_0))\|^{\frac{1}{\lambda}}\nonumber\\
	&=&\|2^{(j+\ell)r}f_{j+\ell}(x)|L_p(\ell_{\theta}(\N_0))\|\label{eq:h11}.
	\eeq
	Inserting the estimates from \eqref{eq:h10} and \eqref{eq:h11} into \eqref{eq:h9} implies
	\beqq
	\|2^{(j+\ell)r}[P_{2^{j+\ell},a}f_{j+\ell}(x)]^{1-\lambda}M|f_{j+\ell}|^{\lambda}(x)|L_p(\ell_{\theta}(\N_0))\|\leq\|2^{(j+\ell)r}f_{j+\ell}(x)|L_p(\ell_{\theta}(\N_0))\|.
	\eeqq
	A similar index shift as above yields
	\be
	\|2^{(j+\ell)r}[P_{2^{j+\ell},a}f_{j+\ell}(x)]^{1-\lambda}M|f_{j+\ell}|^{\lambda}(x)|L_p(\ell_{\theta}(\N_0))\|\leq\|2^{jr}f_{j}(x)|L_p(\ell_{\theta}(\N_0))\|\label{eq:h12}.
	\ee
	We continue estimating \eqref{eq:h7} and insert \eqref{eq:h13} and \eqref{eq:h12} to obtain
	\be
	\Big\|\Big[\sum_{j=0}^{\infty}2^{\theta j r}\Big(2^j\int_{-2^{-j}}^{2^{-j}}\Big|\Delta^m_h \Big[\sum_{\ell\in\N_0}f_{j+\ell}\Big]\Big|dh\Big)^{\theta}\Big]^{\frac{1}{\theta}}\Big\|_p\lesssim\|2^{jr}f_{j}|L_p(\ell_{\theta}(\N_0))\|\Big[\sum_{\ell\geq 0}2^{\mu\ell (m-r)}
	+\sum_{\ell<0}2^{\mu \ell [a(1-\lambda)-r]}\Big]^{\frac{1}{\mu}}\label{eq:hhh14}.
	\ee
	Finally, the choice of the parameters $m, a, \lambda$ in \eqref{eq:lambda2} yields that the series in \eqref{eq:hhh14} converge to a constant. Altogether we obtain the desired bound
	\be
	\Big\|\sum_{\ell\in\N_0}f_{\ell}\Big|\srptfu\Big\|\lesssim \|2^{jr}f_{j}|L_p(\ell_{\theta}(\N_0))\|.
	\ee
	{\em Step 2.} We explain the modifications in the multivariate situation. This time we start computing the norm of $\sum_{\zb \ell \in \N_0^d}f_{\zb \ell}\in\srptf$ in terms of differences, cf. Theorem \ref{satz:differences},
	$$\Big\|\sum_{\zb \ell \in \N_0^d}f_{\zb \ell}\Big|\srptf\Big\|\asymp \sum_{e\subset [d]}\|f|\srptf\|_{e,\zb m}.$$
	For each $e\subset[d]$ we have to show that
	$$\Big\|\sum_{\zb \ell \in \N_0^d}f_{\zb \ell}\Big|\srptf\Big\|_{e,\zb m}\lesssim \|2^{\zb  r \zb j}f_{\zb j}|L_p(\ell_{\theta}(\N_0^d))\|$$
	holds. A full proof consists in applying the arguments from above to every single direction contained in $e$. Here the directionwise Hardy-Littlewood maximal function and
	corresponding maximal inequality come into play, see Definition\ \ref{dirHL} and Thms.\ \ref{satz:feffstein2}, 
	\ref{satz:peetremaximalineq}.
	Since this requires an extensive case study in $e$ and $\zb \ell$ we refer to the proof given in detail in 
	\cite[Thm. 3.4.1, Step 1]{TDiff06} where we have to replace the decomposition of $f$ used there by the representation $\sum_{\zb \ell\in\Z}f_{\zb j+\zb \ell}$.\newline
	{\em Step 3.} We prove (i) in case $\theta<\infty$. To begin with, we define the set of sequences with finite index sets given by
	$$\mathfrak{E}:=\Big\{\mathcal{E}=(\mathcal{E}_n)_{n\in\N}:\;\mathcal{E}_n\subset \N_0^d,\;|\mathcal{E}_n|=n,\;\mathcal{E}_n\subset \mathcal{E}_{n+1} \mbox{ for all }n\in \N \mbox{, and }\bigcup_{n=1}^{\infty}\mathcal{E}_n=\N_0^d \Big\}.$$
	Every sequence in $\mathfrak{E}$ defines an order of summation.
	Furthermore for $\mathcal{E}\in\mathfrak{E}$ we define
	$F_{\mathcal{E}_n}:=\sum_{\zb j\in \mathcal{E}_n}f_{\zb j}$. We take a second sequence $A\in \mathfrak{E}$ and consider $F_{\mathcal{E}_n}-F_{\mathcal{A}_m}$. This difference can be written as a sum with finitely many $f_{\zb j}$. This fulfills the assumptions necessary in Step 1 and yields
	$$\|F_{\mathcal{E}_n}-F_{\mathcal{A}_m}|\srptf\|\lesssim \Big\|\Big(\sum_{\zb j \in (\mathcal{E}_n\cup \mathcal{A}_m)\backslash(\mathcal{E}_n\cap \mathcal{A}_m)}2^{\zb r \cdot \zb j \theta}|f_{\zb j}|^{\theta}\Big)^{\frac{1}{\theta}}\Big\|_p.$$ 
	Obviously,
	$$\Big(\sum_{\zb j \in (\mathcal{E}_n\cup \mathcal{A}_m)\backslash(\mathcal{E}_n\cap \mathcal{A}_m)}2^{\zb r \cdot \zb j \theta}|f_{\zb j}|^{\theta}\Big)^{\frac{1}{\theta}}\leq \Big(\sum_{\zb j \in \N_0^d}2^{\zb r \cdot \zb j \theta}|f_{\zb j}|^{\theta}\Big)^{\frac{1}{\theta}}\in L_p(\T)$$
	holds almost everywhere. Therefore Lebesgue's dominated convergence theorem yields that we find for every $\varepsilon>0$ a $n_0\in\N$ such that
	$$\|F_{\mathcal{E}_n}-F_{\mathcal{A}_m}|\srptf\|\leq \varepsilon$$
	holds for all $m,n>n_0$. Finally this implies unconditional convergence in \srptf. In case $\theta=\infty$ we stress on the embeddings
	$$S^{\zb s}_{p,1}F(\T) \hookrightarrow S^{\zb {\tilde{r}}}_{p,\nu}F(\T)$$ 
	and
	$$ \|2^{\zb s \cdot \zb j}f_j|L_p(\ell_{1})\| \lesssim \|2^{\zb r \cdot \zb j}f_j|L_p(\ell_{\infty})\|,$$ 
	where $\zb r>\zb s>\sigma_{p,\nu}$, $\zb s >\zb {\tilde{r}}$ and $0<\nu\leq\infty$. Applying the arguments from above to $S^{\zb s}_{p,1}F(\T)$ yields the result for  $S^{\zb {\tilde{r}}}_{p,\nu}F(\T)$.
\end{proof}

We will also need the following diagonal embedding relation which is the periodic counterpart of
\cite[Prop.\ 2.4.1]{ST87}, see also the diagonal embedding in Lemma \ref{lem:embeddings}, (vi) above.

\begin{lem}\label{lem:qjdiagembedding}
	Let $0<p<q<\infty$ and $0<\theta,\nu\leq \infty$. Then
	\beqq \Big\|\Big(\sum_{\zb j \in \N_0^d}|f_{\zb j}|^{\nu}\Big)^{\frac{1}{\nu}}\Big\|_q\lesssim
	\Big\|\Big(\sum_{\zb j \in \N_0^d}2^{\theta |\zb
		j|_1(1/p-1/q)}|f_{\zb j}|^{\theta}\Big)^{\frac{1}{\theta}}\Big\|_p
	\eeqq
	holds for all $(f_{\zb j})_{\zb j \in
		\N_0^d}$ such that $f_{\zb j}\in \mathcal{T}^0_{\zb j}$\,.
\end{lem}

Let us finally state the counterpart of Theorem \ref{satz:bandlimrep1} for the $B$-case. 
\begin{satz}\label{satz:bandlimrep11}
	Let $0<p\leq \infty$, $0<\theta\leq \infty$, $\zb r\in \R$ with $\zb r>\sigma_{p}$ and $(f_{\zb j})_{\zb j \in \N_0^d}$ such that $f_{\zb j}\in \mathcal{T}^0_{\zb j}$ and $\big\|2^{\zb r \cdot \zb j}f_{\zb j}\big|\ell_{\theta}(L_p)\big\|<\infty$.
	Then
	\begin{enumerate}
		\item $\sum_{j\in\N_0^d}f_{\zb j}$ converges unconditionally in $\srptb$ if $\max\{p,\theta\}<\infty$ and in every
		$S^{\bf {\tilde{r}}}_{p,\nu}B(\T)$ with $0<\nu\leq \infty$ and $\zb {\tilde{r}}< \zb r.$
		\item it holds
		\beqq \Big\|\sum_{\zb j\in\N_0^d}f_{\zb j}\Big|\srptb\Big\|\lesssim \big\|2^{\zb r \cdot \zb j}f_{\zb
			j}\big|\ell_{\theta}(L_p(\T)\big\|.\eeqq
	\end{enumerate}
\end{satz}
\begin{proof}
	We follow the proof of Theorem \ref{satz:bandlimrep1} line by line and point out the necessary modifications for the $B$-case. To convince the reader we explain this modifications for (ii) in the univariate case.  Again, we prove
	$$\Big\|\sum_{\ell\in\N_0}f_{\ell}\Big|B^r_{p,\theta}(\tor)\Big\|\lesssim \big\|2^{ r j}f_j\big|L_p(\ell_{\theta}(\N_0))\big\|$$
	by using methods from difference characterization. We start by switching to the difference norm in $\srptbu$ with $m>r$
	\be
	\Big\|\sum_{\ell\in\N_0}f_{\ell}\Big|\srptbu\Big\|\asymp\Big\|\sum_{\ell\in\N_0}f_{\ell}\Big\|_p+\Big(\sum_{j=0}^{\infty}2^{\theta j r}\Big\|2^j\int_{-2^{-j}}^{2^{-j}}\Big|\Delta^m_h \Big[\sum_{\ell\in\N_0}f_{\ell}\Big]\Big|dh\Big\|^{\theta}_p\Big)^{\frac{1}{\theta}}\label{eq:h2h}.
	\ee
	First we estimate the $L_p$-norm of $f$ and obtain trivially using either H\"older's inequality (in case $\min\{p,\theta\}>1$ or $p<\min\{1,\theta\}$) or the embedding $\ell_{\theta}\hookrightarrow \ell_{1}$ (otherwise) the estimate 
	$$\Big\|\sum_{\ell\in\N_0}f_{\ell}\Big\|_p\lesssim \Big(\sum_{j\in\N_0}2^{\theta r j}\|f_j\|_p^{\theta}\Big)^{\frac{1}{\theta}}. $$
	Let  $a>0$ be a positive real number such that $a>\frac{1}{p}$ is fulfilled. Additionally, in case $p>1$ we choose $\lambda=1$. Whereas  in case $p\leq 1$ we choose
	$$ 0<\lambda< p $$ such that 
	$$ r>(1-\lambda)a>(1-\lambda)\frac{1}{p}\geq(1-p)\frac{1}{p}=\sigma_{p}.$$
	For the second term in \eqref{eq:h2h} the estimates in \eqref{eq:est1n}, \eqref{eq:decomp1} and \eqref{eq:esth1} yield
	\beq
	&&\Big(\sum_{j=0}^{\infty}2^{\theta j r}\Big\|2^j\int_{-2^{-j}}^{2^{-j}}\Big|\Delta^m_h \Big[\sum_{\ell\in \Z}f_{j+\ell}\Big]\Big|dh\Big\|^{\theta}_p\Big)^{\frac{1}{\theta}}\nonumber\\
	&&\lesssim \Big[\sum_{\ell\geq 0}2^{\mu\ell (m-r)}\|2^{(j+\ell)r}P_{2^{j+\ell},a}f_{j+\ell}(x)|\ell_\theta(L_p(\tor))\|^{\mu}\label{eq:h7222}\\
	&&\quad+\sum_{\ell<0}2^{\mu \ell [a(1-\lambda)-r]}\|2^{(j+\ell)r}[P_{2^{j+\ell},a}f_{2^{j+\ell}}(x)]^{1-\lambda}\times M|f_{j+\ell}|^{\lambda}(x)|\ell_\theta(L_p(\tor))\|^{\mu}\Big]^{\frac{1}{\mu}}\nonumber
	\eeq
	with $\mu=\min\{p,\theta,1\}$.
	The $L_p(\tor)$-norm is now the inner norm in the sequence spaces. For that reason it suffices to use simpler (non-vector valued) maximal inequalities.  
	We apply Theorem \ref{satz:peetremaximalineqnonvec} to the first summand, which gives
	\beqq
	\|2^{(j+\ell)r}P_{2^{j+\ell},a}f_{j+\ell}(x)|\ell_{\theta}(L_{p}(\tor))\|\lesssim \|2^{(j+\ell)r}f_{j+\ell}(x)|\ell_{\theta}(L_p(\tor))\|.
	\eeqq
	An index shift yields
	\be
	\|2^{(j+\ell)r}P_{2^{j+\ell},a}f_{j+\ell}(x)|\ell_{\theta}(L_{p}(\tor))\|\lesssim \|2^{jr}f_{j}(x)|\ell_{\theta}(L_{p}(\tor))\|.\label{eq:h1333}
	\ee
	In case $p\leq 1$ we apply H\"older's inequality to the second summand in \eqref{eq:h7222}  with $\frac{1}{1-\lambda}$,  $\frac{1}{\lambda}$ and obtain
	\beqq
	\|2^{(j+\ell)r}[P_{2^{j+\ell},a}f_{j+\ell}(x)]^{1-\lambda}M|f_{j+\ell}|^{\lambda}(x)|\ell_{\theta}(L_{p}(\tor))\|^{\theta}&\leq&\Big(\sum_{j\in\N_0}2^{\theta(j+\ell)r}
	\|P_{2^{j+\ell},a}f_{j+\ell}(x)|L_p(\tor)\|^{(1-\lambda)\theta}\nonumber\\
	&&\times\|(M|f_{j+\ell}|^{\lambda}(x))^{\frac{1}{\lambda}}|L_p(\tor)\|^{\lambda\theta}\Big)^{\frac{1}{\theta}}.\nonumber\\
	\eeqq
	This can be skipped in case $p>1$.
	Applying the maximal inequalities stated in Theorem \ref{satz:peetremaximalineqnonvec}  and Theorem \ref{satz:hardylittlewood} (together with a trick similar to \eqref{eq:h11})  yields
	\beqq
	\|2^{(j+\ell)r}[P_{2^{j+\ell},a}f_{j+\ell}(x)]^{1-\lambda}M|f_{j+\ell}|^{\lambda}(x)|\ell_{\theta}(L_{p}(\tor))\|&\lesssim&\Big(\sum_{j\in\N_0}2^{\theta(j+\ell)r}
	\|f_{j+\ell}(x)|L_p(\tor)\|^{\theta}\Big)^{\frac{1}{\theta}}\nonumber\\
	&\lesssim& \|2^{jr}f_{j}(x)|\ell_{\theta}(L_{p}(\tor))\|.\nonumber
	\label{eq:h999}
	\eeqq
	Hence, the estimates from \eqref{eq:h1333} and \eqref{eq:h999} imply
	\beq
	&&\Big(\sum_{j=0}^{\infty}2^{\theta j r}\Big\|2^j\int_{-2^{-j}}^{2^{-j}}\Big|\Delta^m_h \Big[\sum_{\ell\in \Z}f_{j+\ell}\Big]\Big|dh\Big\|^{\theta}_p\Big)^{\frac{1}{\theta}}\\
	&&\lesssim \Big[\sum_{\ell\geq 0}2^{\mu\ell (m-r)}\nonumber
	+\sum_{\ell<0}2^{\mu \ell [a(1-\lambda)-r]}\Big]^{\frac{1}{\mu}}\|2^{jr}f_{j}(x)|\ell_{\theta}(L_{p}(\tor))\|.\nonumber
	\eeq
	The choice of $\lambda, a$ and $m$ relatively to $r$ ensures the convergence of the series to an absolute constant.
	This concludes the proof in the univariate case. For the multivariate situation see the comments in Step 2 of Theorem \ref{satz:bandlimrep1}. 
\end{proof}
\section{Sampling representations}
In this section we provide theorems that allow for replacing the Fourier analytic building blocks $\delta_{\zb j}[f]$ used to define the spaces $\srptf$ and $\srptb$ by building blocks $q^L_{\zb j}[f]$, cf. Definition \ref{def:srptf}. 
Here we will prove the following main results.
\begin{satz}\label{maincharF}
	Let $0<p<\infty$, $0<\theta\leq \infty$, $L>\max\{\frac{1}{p},\frac{1}{\theta}\}$ ($L=1$ requires $\theta<\infty$) and $\zb r >\max\{\frac{1}{p},\frac{1}{\theta}\}$ then the (quasi-)norms 
	$$\|f|\srptf\|\asymp \|2^{\zb r \cdot \zb j}q^L_{\zb j}(f)|L_p(\ell_{\theta})\|$$ 
	are equivalent for all $f\in\srptf$.
\end{satz}
\begin{proof}
	The result is a consequence of 
	Theorem \ref{satz:bandlimpres2} together with Theorem \ref{satz:bandlimrep1}. For the case $L=1$ we refer to Theorem \ref{satz:bandlimpresdirichletf}.
\end{proof}

\begin{rem} We strongly conjecture the optimality of the condition on $L$ in the above theorems, 
	see also Remark \ref{rem:propfspace},(ii) above.
\end{rem}

\noindent For the $B$-case weaker conditions on $\boldr$ and $L$ are sufficient. 
\begin{satz}\label{maincharB}
	Let $0<p,\theta\leq \infty$, $L>\frac{1}{p}$ ($L=1$ requires $p<\infty$) and $\zb r >\frac{1}{p}$ then the (quasi-)norms 
	$$\|f|\srptb\|\asymp \|2^{\zb r \cdot \zb j}q^L_{\zb j}(f)|\ell_\theta(L_{p})\|$$ 
	are equivalent for all $f\in\srptb$.
\end{satz}
\begin{proof}
	The proof is a consequence of Theorem \ref{satz:bandlimrep11} together with Theorem \ref{satz:bandlimpresbesov}. For the case $L=1$ we refer to Theorem \ref{satz:bandlimpresdirichletb}.
\end{proof}

\begin{rem}
	In case of $\srptb$ with $p\geq 1$ and $\boldr >1/p$ similar characterizations were proved by Dinh D\~ung \cite{Di91,Di92,Di01} 
	using the following variant of the de la Vall\'{e}e-Poussin kernel 
	\be V_j(x)=\frac{\sin(2^{j-1}x)\sin(32^{j-1}x)}{2^{2j}3\sin^2(\frac{x}{2})}\label{eq:poussin}\,,
	\ee
	which yields to an interpolation operator on $3\cdot 2^j$ equidistant nodes. We can reproduce and extend 
	this result to the Triebel-Lizorkin scale as well as to $p>1/2$ with straight-forward modifications of the 
	arguments used in Theorems \ref{satz:bandlimpresbesov} below. Note, that our proof only uses a reproduction and a 
	decay property of the kernel. Also the de la Vall{\'ee} Poussin sampling operator $R_m$ used by Temlyakov in \cite[I.6]{Tem93} 
	is admissible here. 
\end{rem}

\subsection{The case of quadratically decaying kernels with $L\geq 2$}
Let us first deal with kernels providing at least a quadratic decay according to Lemma \ref{lem:coredecreasing}.
We introduce the characteristic function $\chi_{j,u}$ of the dyadic interval $[2\pi u/2^j, 2\pi (u+1)/2^j]$ indexed by $j\in \N_0$ and $u\in \zz$.
For $\zb j\in\N_0^d$ and $\zb u\in\Z$ we denote with
$$ \chi_{\zb j,\zb u}(\zb x):= \prod_{i=1}^d\chi_{j_i,u_i}(x_i)$$
the characteristic function of the respective parallelepiped.
The following lemma represents the ``hyperbolic'' version  of \cite[Lem.\ 3, 7]{Ke10}. The lemma is
originally due to Kyriazis \cite[Lem.\ 7.1]{Ky03}. For the convenience of the reader we will give a proof.

\begin{lem}\label{lem:kyriazis}
	Let $0<\lambda\leq 1$ and $L>\frac{1}{\lambda}$. For any sequence $(\lambda_{\zb u})_{ \zb u\in A_{\zb j}(\zb x)}$ of complex numbers and every $\zb j \in \N_0^d$ we have
	\be
	\sum_{\zb u \in \Z}|\lambda_{\zb j, \zb u}|\prod_{i=1}^d (1+2^{j}|x_i-x^{j_i}_{u_i}|)^{-L}\leq C \Big[M\Big|\sum_{\zb u\in \Z}\lambda_{\zb j, \zb u}\chi_{\zb j,\zb u}\Big|^\lambda(\zb x)\Big]^{\frac{1}{\lambda}}
	\ee
	with a constant C independent of $\zb j$, $(\lambda_{\zb j, \zb u})_{\zb u}$ and $\zb x$.
\end{lem}
\begin{proof} We follow the proof given in \cite[Lem.\ 7]{Ke10}. Let $\delta = L-1/\lambda>0$. We define the following decomposition of $\Z$. Let $\zb x\in \R$ then
	$$
	\Omega_{\zb k}(\zb x):=\Omega_{k_1}\cdot...\cdot \Omega_{k_d},\quad \zb k \in \N_0^d\,,
	$$
	with 
	\begin{equation}\nonumber
		\begin{split}
			\Omega_{k_i}&:=\{u_i \in A_{j_i}(x_i)~:~\pi 2^{k_i-1}<2^{j_i}|x_i-x^{j_i}_{u_i}| \leq
			\pi 2^{k_i}\}\\
			\Omega_0&:=\big\{u_i \in A_{j_i}(x_i)~:~2^{j_i}|x_i-x^{j_i}_{u_i}| \leq
			\pi\big\}\,.
		\end{split}
	\end{equation}
	We estimate 
	\begin{equation}\label{est1}
		\begin{split}
			&\sum\limits_{\zb u \in \Z} |\lambda_{\zb j,\zb u}|\prod\limits_{i=1}^d
			(1+2^{j_i}|x_i-x^{j_i}_{u_i}|)^{-L}\\
			&~~~~~=\sum\limits_{\zb k \in \N_0^d}\sum\limits_{\zb u \in \Omega_{\zb k}(\zb x)}
			|\lambda_{\zb j,\zb u}|\prod\limits_{i=1}^d
			(1+2^{j_i}|x_i-x^{j_i}_{u_i}|)^{-L}\\
			&~~~~~\lesssim \sum\limits_{\zb k \in \N_0^d}\sum\limits_{\zb u \in \Omega_{\zb k}(\zb x)}
			|\lambda_{\zb j,\zb u}|2^{-\delta|\zb k|_1-|\zb k|_1/\lambda} \lesssim \sup\limits_{\substack{ k_i\leq j_i\\ i\in[d]}} \Big(\sum\limits_{\zb u \in
				\Omega_{\zb k}(\zb x)}|\lambda_{\zb j,\zb u}|\Big)2^{-|\zb k|_1/\lambda}\\
			&~~~~~\lesssim \Big(\sup\limits_{\zb k \in \N_0^d}2^{-|\zb k|_1}\sum\limits_{\zb u \in 
				\Omega_{\zb k}(\zb x)}|\lambda_{\zb j,\zb u}|^\lambda\Big)^{1/\lambda}\,.
		\end{split}
	\end{equation}
	Note, that 
	$$
	\int\limits_{\bigcup\limits_{\zb u \in \Omega_{\zb k}(\zb x)}Q_{\zb j,\zb u}} \sum\limits_{\zb \omega \in \Omega_{\zb k}(\zb x)}
	|\lambda_{\zb j,\zb \omega}|^\lambda\chi_{\zb j,\zb \omega}(\zb y)\, d\zb y = 2^{-|\zb j|_1}\sum\limits_{\zb u \in \Omega_{\zb k}(\zb x)}
	|\lambda_{\zb j,\zb u}|^\lambda
	$$
	and hence
	\begin{equation}\label{est2}
		2^{-|\zb k|_1}\sum\limits_{\zb u \in 
			\Omega_{\zb k}(\zb x)}|\lambda_{\zb j,\zb u}|^\lambda = 2^{-|\zb k|_1+|\zb j|_1}\int\limits_{\bigcup\limits_{\zb u \in
				\Omega_{\zb k}(\zb x)}Q_{\zb j,\zb u}} \sum\limits_{\zb \omega \in \Omega_{\zb k}(\zb x)}
		|\lambda_{\zb j,\zb \omega}|^\lambda\chi_{\zb j,\zb \omega}(\zb y)\, d \zb y\,.
	\end{equation}
	We further observe that for $Q(\zb x):=\bigcup\limits_{\zb u \in
		\Omega_{\zb k}(\zb x)}Q_{\zb j,\zb u}$ we have 
	$$
	|Q(\zb x)| \asymp 2^{|\zb k|_1-|\zb j|_1}.
	$$
	and $\zb x\in Q(\zb x)$\,. Recalling the definition of the Hardy-Littlewood maximal functions in Definition \ref{def:hl} we obtain
	\begin{equation}\nonumber
		\begin{split}      
			\int_{Q(\zb x)} \sum\limits_{\zb u \in \Omega_{\zb k}(\zb x)}
			|\lambda_{\zb j,\zb u}|^\lambda\chi_{\zb j,\zb u}(\zb y)\,d\zb y &\lesssim  2^{|\zb k|_1-|\zb j|_1}
			\frac{1}{|Q(\zb x)|}\int_{Q(\zb x)} \sum\limits_{\zb u \in \Omega_{\zb k}(\zb x)} |\lambda_{\zb j,\zb u}|^\lambda\chi_{\zb j,\zb u}(\zb y)\,d \zb y\\
			&\leq 2^{|\zb k|_1-|\zb j|_1}M\Big|\sum\limits_{\zb u\in \Z}
			\lambda_{\zb j,\zb u}\chi_{\zb j,\zb u}\Big|^\lambda(\zb x)\,.
		\end{split}
	\end{equation}
	Putting this into \eqref{est2} we obtain
	$$
	2^{-|\zb k|_1}\sum\limits_{\zb u \in 
		\Omega_{\zb k}(\zb x)}|\lambda_{\zb j,\zb u}|^\lambda \lesssim M\Big|\sum\limits_{\zb u\in \Z}
	\lambda_{\zb j,\zb u}\chi_{\zb j,\zb u}\Big|^\lambda(\zb x)\,.
	$$
	Finally, we plug this estimate into \eqref{est1} and obtain the desired result. 
\end{proof}

\begin{prop}\label{lem:samplingvspeetre}
	Let $\zb \ell,\zb j\in \N_0^d$, $0<\lambda\leq 1$, $L\in \N$ with $L>\frac{1}{\lambda}$ and $a>0$. Let further $f\in C(\T)$. 
	\begin{enumerate}
		\item  Then $$|I^L_{\zb j}[f](\zb x)|\lesssim 2^{a|\zb \ell|_1}[M|P_{\zb 2^{\zb j+\zb \ell},a}f|^{\lambda}(\zb x)]^{\frac{1}{\lambda}}$$ 
		\item and also $$|q^L_{\zb j}[f](\zb x)|\lesssim 2^{a|\zb \ell|_1}[M|P_{\zb 2^{\zb j+\zb \ell},a}f|^{\lambda}(\zb x)]^{\frac{1}{\lambda}}$$ 
	\end{enumerate}
	holds with a constant independent of $\zb \ell,\zb j,\zb x$ and $f$.
\end{prop}
\begin{proof}
	We start proving (i). Recall the notation from \eqref{Aj(x)}. Periodicity of $f$ and $K^{L}_{\pi,j}$ yields
	\beqq
	I^L_{\zb j}f(\zb x)|&=& \sum_{\zb u \in A_{\zb j}}f(\zb x^{\zb j}_{\zb u}) K_{\pi^d,\zb j}^L(\zb x-\zb x^{\zb j}_{\zb u})\\
	&=& \sum_{\zb u \in A_{\zb j}(\zb x)}f(\zb x^{\zb j}_{\zb u}) K_{\pi^d,\zb j}^L(\zb x-\zb x^{\zb j}_{\zb u}).
	\eeqq
	Lemma \ref{lem:coredecreasing} with $|x_i-x_{u_i}^{j_i}|\leq \pi$ gives
	\beqq
	|I^L_{\zb j}f(\zb
	x)|&\leq&\sum_{\zb u \in A_{\zb j}(\zb x)}\left|f(\zb x^{\zb j}_{\zb u})\prod_{i=1}^d K_{\pi,j_i}^L(x_i-x^{j_i}_{u_i})\right|\\
	&\lesssim&\sum_{\zb u\in \Z}|\lambda_{\zb j,\zb u}|\prod_{i=1}^d (1+2^{j_i}|x_i-x^{j_i}_{u_i}|)^{-L}\,,
	\eeqq
	where we used the notation  
	$$ \lambda_{\zb j,\zb u} := \sup_{\substack{\zb y:|y_i-x^{j_i}_{u_i}|< \frac{2\pi}{2^{j_i}}\\i\in[d]}}|f(\zb y)|\,.$$
	Applying Lemma \ref{lem:kyriazis} gives
	\be
	|I^L_{\zb j}f(\zb x)|\lesssim \Big[M\Big|\sum_{\zb u\in \Z}\lambda_{\zb j,\zb u}\chi_{\zb j,\zb k}\Big|^{\lambda}(\zb x)\Big]^{\frac{1}{\lambda}}.\label{eq:kyriazisest}
	\ee
	Taking $\zb z\in \supp \chi_{\zb j,\zb u^*}$ gives for any $a>0$
	\beqq
	\Big|\sum_{\zb u\in \Z}\lambda_{\zb j,\zb u}\chi_{\zb j,\zb u}(\zb z)\Big|=|\lambda_{\zb j,\zb u^*}|
	&=&\sup_{\substack{\zb y:|y_i-x^{j_i}_{u^*_i}|< \frac{2\pi}{2^{j_i}}\\i\in[d]}}|f(\zb y)|\\
	&\lesssim& \sup_{\substack{\zb y:|y_i- z_i|< \frac{4\pi}{2^{j_i}}\\i\in[d]}}\frac{|f(\zb y)|}{\prod_{i=1}^d(1+2^{j_i}|y_i-z_i|)^a}\,.\\
	\eeqq
	Finally, Lemma \ref{lem:peetreenv} yields
	\be
	\Big|\sum_{\zb u\in \Z}\lambda_{\zb j,\zb u}\chi_{\zb j,\zb u}(\zb z)\Big| \lesssim 2^{|\zb \ell|_1 a} P_{\zb 2^{\zb j},a}f(\zb z).\label{eq:peetreest}
	\ee
	Inserting \eqref{eq:peetreest} into \eqref{eq:kyriazisest} finishes the proof of (i).
	The bound in (ii) is a trivial consequence of applying triangle inequality to  \eqref{eq:qjasijinterpret} and (i)
	\beq
	|q^L_{\zb j}[f](\zb x)|&\leq&\sum_{\zb b\in \{-1,0\}^d}|I^L_{{\zb j+\zb b}}[f](\zb x)|.\label{eq:ijtriangle}
	\eeq
\end{proof}

\begin{satz}\label{satz:bandlimpres2}
	Let $0<p,\theta\leq \infty$ ($p<\infty$), $L>\max\{\frac{1}{p},\frac{1}{\theta},1\}$ and $\zb r >\max\{\frac{1}{p},\frac{1}{\theta}\}.$ 
	\begin{enumerate}
		\item Then every $f\in \srptf$ admits the representation
		\be f=\sum_{\zb j\in \N_0^d}q^L_{\zb j}[f], \label{eq:qkseries}\ee
		
		with unconditional convergence in $\srptf$ in case $0<\theta<\infty$ and 
		with unconditional convergence in $S^{\zb {\tilde{r}}}_{p,\nu}F(\T)$ for every $\zb r > \zb {\tilde{r}}$ and $0<\nu\leq\infty$ in case $\theta=\infty.$ 
		
		\item There is a constant $C>0$ independent of $f$ such that
		\be\|2^{\zb r \cdot \zb j}q^L_{\zb j}(f)|L_p(\ell_{\theta})\|\leq C \|f|\sabpqf\|\label{eq:repest2}\ee
		holds for all $f\in\srptf$.
	\end{enumerate}
\end{satz}
\begin{proof}{\em Step 1.} We prove \eqref{eq:repest2}. To begin with we choose $a>0$ such that $\zb r>a>\max\{\frac{1}{p},\frac{1}{q}\}$ is fulfilled. Let $f\in\srptf$.
	We start for $\zb j\in \N_0^d$ with the Fourier decomposition
	\be f(\zb x) = \sum_{\zb \ell \in \Z}\delta_{\zb j+\zb \ell}[f](\zb x),\label{eq:deltaseries}\ee
	cf. \eqref{eq:fourierdecomp}, where $\delta_{\zb j}[f]:=0$ for $\zb j\in \Z\backslash
	\N_0^d$ . This series converges unconditionally in $C(\T)$, due to the embedding $\srptf\hookrightarrow \C(\T)$.
	That yields the point-wise estimate
	\beqq
	|q^L_{\zb j}[f](\zb x)|
	\leq \sum_{\zb \ell\in\Z}|q^L_{\zb j}[\delta_{\zb j+\zb \ell}[f]](\zb x)|.
	\eeqq
	For the sake of simplicity we assume that the constants $A,B,C$ in Definition \ref{def:varphi} are chosen in such a way that 
	$\delta_{\zb j}[f]\in\mathcal{T}^{L}_{\zb j}$ is fulfilled for all $\zb j\in \N_0^d$. 
	Then Proposition \ref{lem:qkfreprod} implies
	$$|q^L_{\zb j}[f](\zb x)|
	\leq \sum_{\zb \ell\geq 0}|q^L_{\zb j}[\delta_{\zb j+\zb \ell}[f]](\zb x)|.$$
	Applying Proposition \ref{lem:samplingvspeetre}, (ii) we obtain
	\beqq
	|q^L_{\zb j}[f](\zb x)|&\lesssim& \sum_{\zb \ell\geq 0}2^{a|\zb \ell|_1}\Big[M|P_{\zb 2^{\zb j+\zb \ell},a}\delta_{\zb j+\zb \ell}[f]|^{\lambda}(\zb x)\Big]^{\frac{1}{\lambda}}.
	\eeqq
	Multiplying with the weight $2^{\zb r \cdot \zb j}$ we find the point-wise estimate 
	\be
	2^{\zb r \cdot \zb j}|q^L_{\zb j}[f](\zb x)|\lesssim \sum_{\zb \ell\geq 0}2^{ (a\zb 1-\zb r)\cdot \zb \ell}2^{\zb r \cdot (\zb j+\zb \ell)}\Big[M|P_{\zb 2^{\zb j+\zb \ell},a}\delta_{\zb j+\zb \ell}[f]|^{\lambda}(\zb x)\Big]^{\frac{1}{\lambda}}.\label{eq:estm99}
	\ee
	where $\lambda$ is chosen as $L>\frac{1}{\lambda}>\frac{1}{\min\{p,\theta\}}$ ($\lambda=1$ in case $\min\{p,\theta\}>1$). The parameter $a$ will be fixed later. 
	Now we take the $L_p(\ell_{\theta})$ (quasi)-norm on both sides. Due to $u$-triangle inequality in $L_p(\ell_{\theta})$ with $u=\min\{p,\theta,1\}$ we obtain
	\be
	\|2^{\zb r \cdot \zb j}q^L_{\zb j}[f]|L_p(\ell_{\theta})\|\lesssim \Big(\sum_{\zb \ell\geq 0}2^{(a\zb 1-\zb r)\cdot \zb \ell u}\Big\|2^{\zb r \cdot (\zb j+\zb \ell)}\Big[M|P_{\zb 2^{\zb j+\zb \ell},a}\delta_{\zb j+\zb \ell}[f]|^{\lambda}\Big]^{\frac{1}{\lambda}}\Big|L_p(\ell_{\theta})\Big\|^u\Big)^{\frac{1}{u}}.\label{eq:estfgh}\\
	\ee
	Since $\lambda < \min\{p,\theta\}$ in case $\min\{p,\theta\}\leq 1$ a trick similar to \ref{eq:h11} yields
	\beqq
	\Big\|2^{\zb r \cdot (\zb j+\zb \ell)}\Big[M|P_{\zb 2^{\zb j+\zb \ell},a}\delta_{\zb j+\zb \ell}[f]|^{\lambda}\Big]^{\frac{1}{\lambda}}\Big|L_p(\ell_{\theta})\Big\|&=& \Big\|2^{\lambda\zb r \cdot (\zb j+\zb \ell)}M|P_{\zb 2^{\zb j+\zb \ell},a}\delta_{\zb j+\zb \ell}[f]|^{\lambda}\Big|L_{\frac{p}{\lambda}}(\ell_{\frac{\theta}{\lambda}})\Big\|^{\frac{1}{\lambda}}.
	\eeqq
	This allows us to apply Fefferman-Stein maximal inequality (Theorem \ref{thm:feffermanstein})
	\beqq
	\Big\|2^{\zb r \cdot (\zb j+\zb \ell)}\Big[M|P_{\zb 2^{\zb j+\zb \ell},a}\delta_{\zb j+\zb \ell}[f]|^{\lambda}\Big]^{\frac{1}{\lambda}}\Big|L_p(\ell_{\theta})\Big\|&\lesssim&\|2^{\zb r \cdot (\zb j+\zb \ell)}P_{\zb 2^{\zb j+\zb \ell},a}\delta_{\zb j+\zb \ell}[f]|L_p(\ell_{\theta})\|.
	\eeqq
	Next we choose $a$ such that $\zb r> a>\max\{\frac{1}{p},\frac{1}{\theta}\}$ holds. Then applying Peetre maximal inequality (Theorem \ref{satz:peetremaximalineq}) gives
	\beqq
	\|2^{\zb r \cdot (\zb j+\zb \ell)}P_{\zb 2^{\zb \ell+\zb j},a}\delta_{\zb \ell+\zb j}[f]|L_p(\ell_{\theta})\|\lesssim \|2^{\zb r \cdot (\zb j+\zb \ell)}\delta_{\zb \ell+\zb j}[f]|L_p(\ell_{\theta})\|.
	\eeqq
	Obviously, we have
	\beqq\|2^{\zb r \cdot (\zb j+\zb \ell)}\delta_{\zb \ell+\zb j}[f]|L_p(\ell_{\theta})\|\leq \|2^{\zb r \cdot \zb j}\delta_{\zb j}[f]|L_p(\ell_{\theta})\|.\eeqq
	Inserting this into \eqref{eq:estfgh} yields
	\beqq
	\|2^{\zb r \cdot \zb j}q^L_{\zb j}[f]|L_p(\ell_{\theta})\|&\lesssim& \|2^{\zb r \cdot \zb j}\delta_{\zb j}[f]|L_p(\ell_{\theta})\|\Big(\sum_{\zb \ell\geq 0}2^{(a\zb 1-\zb r)\cdot \zb \ell u}\Big)^{\frac{1}{u}}\\
	&\lesssim& \|2^{\zb r \cdot \zb j}\delta_{\zb j}[f]|L_p(\ell_{\theta})\|,
	\eeqq
	where the choice of $a$ ensures the convergence of the series to an absolute constant.
	{\em Step 2.} We prove (i). The equation \eqref{eq:repest2} implies 
	$$\|2^{\zb r \cdot \zb j}q^L_{\zb j}[f]|L_p(\ell_{\theta})\|<\infty.$$
	Then Theorem \ref{satz:bandlimrep1} yields unconditional convergence of the series
	$\sum_{\zb j\in \N_0^d}q^L_{\zb j}[f]$.
	We show in case $0<\theta<\infty$ $$\Big\|f-\sum_{|\zb j|_1<n}q^L_{\zb j}[f]\Big|\srptf\Big\|\longrightarrow 0 \quad(n\to \infty). $$
	As a consequence of Definition \ref{def:srptf} trigonometric polynomials are dense in $\srptf$ if $\theta<\infty $. For that reason we find for every $\varepsilon>0$ a trigonometric polynomial $t$ such that $$\|f-t|\srptf\|<\varepsilon.$$
	The $u$-triangle inequality gives
	\beqq
	\Big\|f-\sum_{|\zb j|_1<n}q^L_{\zb j}[f]\Big|\sabpqf\Big\|^u&\leq& \|f-t|\srptf\|^u+\Big\|t-\sum_{|\zb j|_1<n}q^L_{\zb j}[f]\Big|\sabpqf\Big\|^u.
	\eeqq
	For $n$ sufficiently large we obtain by Lemma \ref{lem:trigreprod_hyp} 
	\beqq
	t-\sum_{|\zb j|_1<n}q^L_{\zb j}[f]=\sum_{|\zb j|_1<n}q^L_{\zb j}(t-f).
	\eeqq
	Applying Theorem \ref{satz:bandlimrep1} we have
	\beqq
	\Big\|\sum_{|\zb j|_1<n}q^L_{\zb j}(t-f)\Big|\srptf\Big\|\lesssim \Big\|\Big(\sum_{|\zb j|_1<n} 2^{\theta\zb r \cdot \zb j}|q^L_{\zb j}(t-f)|^{\theta}\Big)^{\frac{1}{\theta}}\Big\|_{p}.
	\eeqq
	Finally, Step 1 yields
	\beqq
	\Big\|\Big(\sum_{|\zb j|_1<n} 2^{\theta\zb r \cdot \zb j}|q^L_{\zb j}(t-f)|^{\theta}\Big)^{\frac{1}{\theta}}\Big\|_p\lesssim \|t-f|\srptf\|
	\eeqq
	and hence, there is a constant $C>0$ independent of $n,f$ and $t$ such that
	\beqq
	\Big\|f-\sum_{|\zb j|_1<n}q^L_{\zb j}[f]\Big|\srptf\Big\|\leq C 2\varepsilon.
	\eeqq
	The case $\theta=\infty$ is based on the embedding
	$$S^{\zb r}_{p,\infty} F(\T)\hookrightarrow S^{\zb s}_{p,p}F(\T)\hookrightarrow S^{\zb {\tilde{r}}}_{p,\nu}F(\T)$$
	with $\zb r>\zb s >\frac{1}{p}$, $\zb s>\zb {\tilde{r}}$ and $0<\nu<\infty$ where the density argument from above is applied to $ S^{\zb s}_{p,p}F(\T)$.
\end{proof}

\begin{rem}
	The recent result in \cite[Rem.\ 7.3]{SeUl15_1}, see also \cite{SeUl15_2}, indicates that a corresponding characterization in case of small smoothness, i.e. $\frac{1}{p}<r\leq \frac{1}{\theta}$ may fail.
\end{rem}

\begin{satz}\label{satz:bandlimpresbesov}
	Let $0<p,\theta\leq \infty$, $ L>\max\{\frac{1}{p},1\}$, $\zb r >\frac{1}{p}.$ 
	\begin{enumerate}
		\item Then every $f\in \srptb$  can be represented by
		$$ f=\sum_{\zb j\in \N_0^d}q^L_{\zb j}(f),$$
		
		with unconditional convergence in $\srptb$ in case $\max\{p,\theta\}<\infty,$
		and with unconditional convergence in $S^{\zb {\tilde{r}}}_{p,\nu}B(\T)$ for every $\zb r > \zb {\tilde{r}}$ and $0<\nu\leq \infty$ in case $\max\{p,\theta\}=\infty.$ 
		\item There is a constant $C>0$ independent of $f$ such that
		\beqq\|2^{\zb r \cdot \zb j}q^L_{\zb j}[f]|\ell_{\theta}(L_p(\T))\|\leq C \|f|\srptb\|\eeqq
		holds for all $f\in\srptb$.
	\end{enumerate}
	
\end{satz}
\begin{proof}
	Concerning representation and unconditional convergence we follow the proof of Theorem \ref{satz:bandlimpres2} line by line with the obvious modifications for the $B$-case. The inequality in (ii) can be proved by the following arguments.
	We take the $\ell_{\theta}(L_{p}(\T))$ (quasi)-norm on both sides of the estimate in \eqref{eq:estm99}. Due to $u$-triangle inequality in $\ell_{\theta}(L_{p}(\T))$ with $u=\min\{p,\theta,1\}$ we obtain
	\be
	\|2^{\zb r \cdot \zb j}q^L_{\zb j}[f]|\ell_{\theta}(L_{p}(\T))\|\lesssim \Big(\sum_{\zb \ell\geq0}2^{(a\zb 1-\zb r)\cdot \zb \ell u}\Big\|2^{\zb r \cdot (\zb j+\zb \ell)}\Big[M|P_{\zb 2^{\zb j+\zb \ell},a}\delta_{\zb j+\zb \ell}[f]|^{\lambda}\Big]^{\frac{1}{\lambda}}\Big|\ell_\theta(L_p(\T))\Big\|^u\Big)^{\frac{1}{u}}\label{eq:tzu}
	\ee
	with $\zb r>a>\frac{1}{p}$ and $0<\lambda<p$ ($\lambda=1$ if $p>1$). In case $p\leq 1$ a trick similar to \ref{eq:h11} yields
	\beqq
	\Big\|2^{\zb r \cdot (\zb j+\zb \ell)}\Big[M|P_{\zb 2^{\zb j+\zb \ell},a}\delta_{\zb j+\zb \ell}[f]|^{\lambda}\Big]^{\frac{1}{\lambda}}\Big|\ell_{\theta}(L_p(\T)\Big\|&=&\Big(\sum_{\zb j\in\N_0^d} 2^{ \zb r \cdot (\zb j+\zb \ell) \theta} \Big\|M|P_{\zb 2^{\zb j+\zb \ell},a}\delta_{\zb j+\zb \ell}[f]|^{\lambda}\Big\|_{\frac{p}{\lambda}}^{\frac{\theta}{\lambda}}\Big)^{\frac{1}{\theta}}.
	\eeqq
	This allows us to apply Hardy-Littlewood maximal inequality (Theorem \ref{satz:hardylittlewood}). We obtain
	\beqq
	\Big\|2^{\zb r \cdot (\zb j+\zb \ell)}\Big[M|P_{\zb 2^{\zb j+\zb \ell},a}\delta_{\zb j+\zb \ell}[f]|^{\lambda}\Big]^{\frac{1}{\lambda}}\Big|\ell_{\theta}(L_p(\T)\Big\|&\lesssim&\Big(\sum_{\zb j\in\N_0^d} 2^{ \zb r \cdot (\zb j+\zb \ell) \theta}\Big\|P_{\zb 2^{\zb j+\zb \ell},a}\delta_{\zb j+\zb \ell}[f]|\Big\|_{p}^{\theta}\Big)^{\frac{1}{\theta}}.
	\eeqq
	Inserting this into \eqref{eq:tzu} and applying (non-vector valued) Peetre maximal inequality (Theorem \ref{satz:peetremaximalineqnonvec}) gives
	\beqq
	\|2^{\zb r \cdot \zb j}q^L_{\zb j}[f]|\ell_{\theta}(L_{p}(\T))\|&\lesssim& \Big(\sum_{\zb \ell\geq 0}2^{(a\zb 1-\zb r)\cdot \zb \ell u}\|2^{\zb r \cdot (\zb j+\zb \ell)}\delta_{\zb \ell+\zb j}[f]|\ell_{\theta}(L_{p}(\T))\|^u\Big)^{\frac{1}{u}}\\
	&\leq& \Big(\sum_{\zb \ell\geq 0}2^{(a\zb 1-\zb r)\cdot \zb \ell u}\Big)^{\frac{1}{u}}\|2^{\zb r \cdot \zb j}\delta_{\zb j}[f]|\ell_{\theta}(L_{p}(\T))\|,
	\eeqq
	where the term inside the $\ell_{\theta}(L_{p}(\T))$ norm does not depend any longer on $\ell$. Therefore the sum over $\zb \ell$ converges to a constant depending only on $a$, $\zb r$ and the dimension $d$. Finally, we obtain
	\beqq\|2^{\zb r \cdot \zb j}q^L_{\zb j}[f]|\ell_{\theta}(L_p(\T))\|\lesssim \|f|\srptb\|.\eeqq
\end{proof}

\subsection{The case of the Dirichlet kernel where $L=1$}
In this subsection we study sampling representations based on the Dirichlet kernel $K^1_{\pi,j}$. 
Its slow decay causes some difficulties. We define an auxiliary kernel
$$\tilde{K}^{2}(x):=\sqrt{2\pi}\mathcal{F}^{-1}\big(4\chi_{{[-\frac{5}{8},\frac{5}{8}]}} \ast \chi_{[-\frac{1}{8},\frac{1}{8}]}\big)(x) = 16\frac{\sin\big(\frac{5}{8}x\big)\sin\big(\frac{1}{8}x\big)}{x^2}\in L_1(\R)$$
and its periodization
$$\tilde{K}^2_{\pi,j}(x):= \sum_{k=-\infty}^{\infty}\tilde{K}^2(2^j(x+2\pi k)).$$
Similar to Lemma \ref{lem:coredecreasing} we can show for $|x|<\pi$  the following decay property
\be |\tilde{K}^2_{\pi,j}(x)|\lesssim \frac{1}{(1+2^j|x|)^2}.\label{eq:core2decreasing}\ee
Note, that the corresponding operator $\tilde{I}^2_j$ defined via \eqref{I_j} is a sampling but not an interpolation operator. 
However, Lemma \ref{lem:fourierformula} still holds true. 
According to Subsection \ref{sec:22} we define the multivariate sampling operator $\tilde{I}^2_{\zb j}f$ based on the tensorized kernel $\tilde{K}^2_{\pi^d,\zb j}$. 

The following formula is a counterpart of a similar formula used by Temlyakov in \cite[Lem.\ I.6.2]{Tem93}\,. Taking \eqref{Dirichlet} into account 
we denote
$$
\mathcal{D}^1_{\zb j} = \mathcal{D}^1_{j_1} \otimes \cdots \otimes \mathcal{D}^1_{j_d}\quad,\quad \zb j \in \N_0^d\,.
$$
\begin{lem}\label{lem:i1rep} Let $f\in C(\T)$. Then
	\be I^1_{\zb j} f = (2\pi)^{-d}\mathcal{D}^1_{\zb j}\ast \tilde{I}_{\zb j}^2 f \label{eq:temlyakovdirichtletrep}\ee
	for all $\zb j\in \N_0^d$.
\end{lem}
\begin{proof}
	We prove the identity by comparing the Fourier series for arbitrary continuous functions $f$. \eqref{eq:fseries} implies
	\be\label{eq0000}\widehat{I^1_{\zb j}f}(\zb \ell)= \Big(\prod_{i\in[d]}\chi_{[-2^{j_i-1},2^{j_i}-1-1]}(\ell_i)\Big) \sum_{\zb u\in A_{\zb j}} f(\zb x_{\zb u}^{\zb j})e^{-i\zb x_{\zb u}^{\zb j} \cdot \zb \ell}.\ee
	Additionally, the same computation as used in Lemma \ref{lem:fourierformula} shows
	$$\widehat{\tilde{I}_{\zb j}^2[f]}(\zb \ell)=\frac{1}{(2\pi)^{d/2}}\Big(\prod_{i=1}^d \mathcal{F}\tilde{K}^2\Big(\frac{\ell_i}{2^{j_i}}\Big)\Big)\sum_{\zb u\in A_{\zb j}} f(\zb x_{\zb u}^{\zb j})e^{-i\zb x_{\zb u}^{\zb j} \cdot \zb \ell}.$$
	Clearly,
	\be(2\pi)^{-d}\widehat{ \mathcal{D}^1_{\zb j}\ast \tilde{I}_{\zb j}^2 f}(\zb \ell)=
	\widehat{\mathcal{D}^1_{\zb j}}(\zb \ell )\widehat{\tilde{I}_{\zb j}^2 f}(\zb \ell) = 
	\widehat{\mathcal{D}^1_{\zb j}}(\zb \ell )
	\sum_{\zb u\in A_{\zb j}} f(\zb x_{\zb u}^{\zb j})e^{-i\zb x_{\zb u}^{\zb j} \cdot \zb \ell}\label{eq0001}\ee
	since
	$$\mathcal{F}\tilde{K}^2\Big(\frac{\ell_i}{2^{j_i}}\Big)=\sqrt{2\pi}$$
	for $\ell_i \in [-2^{j_i},2^{j_i}), i\in[d]$. Comparing \eqref{eq0000} and \eqref{eq0001} yields the claim. 
\end{proof}
\begin{lem}
	\label{lem:itjasij}
	Let $\zb \ell,\zb j\in \N_0^d$, $a>0$ and $1/2 < \lambda\leq 1$. Furthermore, let $f\in C(\T)$. Then 
	$$|\tilde{I}^2_{\zb j}[f](\zb x)|\lesssim 2^{a|\zb \ell|_1}[M|P_{\zb 2^{\zb j+\zb \ell},a}f|^{\lambda}(\zb x)]^{\frac{1}{\lambda}}$$ 
	holds with a constant independent of $\zb \ell,\zb j,\zb x$ and $f$.
\end{lem}
\begin{proof} We refer to the proof of Proposition \ref{lem:samplingvspeetre}. Recognizing, that the only property of $\tilde{I}^2_{\zb j}$ we need  is the decay of the underlying kernel $\tilde{K}^2_{\pi^d,\zb j}$ provided in \eqref{eq:core2decreasing}.
\end{proof}

\begin{rem}\label{rem512} {\em (i)} The estimates in Lemmas \ref{lem:samplingvspeetre}, \ref{lem:itjasij} are pointwise and very useful for 
	$L_p(\tor^d,\ell_\theta)$ estimates. In case one is interested in (scalar) $L_p$ estimates, similar as in \cite[Lem.\ I.6.2]{Tem93}, then 
	Lemmas \ref{lem:samplingvspeetre} and \ref{lem:itjasij} together with the maximal inequalities Theorems \ref{satz:hardylittlewood}, 
	\ref{satz:peetremaximalineqnonvec} imply for $0<p\leq \infty$, $L>\max\{1/p,1\}$ and any $a>1/p$
	\be\label{L_pbound}
	\|I^L_{\zb j}f\|_p \lesssim_{L,a} 2^{|\zb \ell|_1a}\|f\|_p  \quad,\quad f\in \mathcal{T}_{\zb j + \zb \ell}^0	
	\ee
	(similar for $\tilde{I}^2_{\zb j}$).\\
	
	{\em (ii)} There is a different technique based 
	on periodic versions of Plancherel-Polya inequalities (Marcinkiewicz-Zygmund inequalities) 
	for $0<p\leq \infty$, see \cite[Thms.\ 6,10]{SchmSi2000}. A straight-forward modification of the argument in \cite[Lem.\ 13,(ii)]{SchmSi2000}
	gives for $0<p\leq \infty$ and $L>\max\{1/p,1\}$
	\be\label{planch-polya}
	\|I^L_{\zb j}f\|_p \lesssim_{p} 2^{|\zb \ell|_1/p}\|f\|_p  \quad,\quad f\in \mathcal{T}_{\zb j + \zb \ell}^0
	\ee
	(similar for $\tilde{I}^2_{\zb j}$). In case $L=2$ (de la Vall{\'ee} Poussin) this yields an extension of \cite[Lem.\ I.6.2]{Tem93} to the range $1/2<p\leq \infty$.\\

	{\em (iii)} 
	By Lemma \ref{lem:i1rep} and the uniform boundedness of the 
	multivariate Fourier partial sum operator in $L_p(\tor^d)$, $1<p<\infty$, we obtain from \eqref{L_pbound} and \eqref{planch-polya} corresponding estimates also 
	for $\|I^1_{\zb j}f\|_p$\,.
\end{rem}

\begin{satz}\label{satz:bandlimpresdirichletf}
	Let $1<p,\theta< \infty$  and $\zb r >\max\{\frac{1}{p},\frac{1}{\theta}\}.$ 
	\begin{enumerate}
		\item Then every $f\in \srptf$ admits the representation
		$$ f=\sum_{\zb j\in \N_0^d}q^1_{\zb j}[f], $$
		with unconditional convergence in $\srptf$.
		\item There is a constant $C>0$ independent of $f$ such that
		$$\|2^{\zb r \cdot \zb j}q^1_{\zb j}(f)|L_p(\ell_{\theta})\|\leq C \|f|\sabpqf\|$$
		holds for all $f\in\srptf$.
	\end{enumerate}
\end{satz}
\begin{proof}
	The proof of (i) is similar to Theorem \ref{satz:bandlimpres2}, (i). We prove (ii) here. 
	Inserting the decomposition \eqref{decunity}, applying triangle inequality and afterwards Proposition \ref{lem:qkfreprod} gives
	\beqq
	\|2^{\zb r \cdot \zb j}q^1_{\zb j}[f]|L_p(\ell_{\theta}(\N_0^d))\|\lesssim \sum_{\zb \ell \geq 0}2^{-\zb \ell\cdot \zb r }\|2^{\zb r \cdot (\zb j+\zb \ell)}q^1_{\zb j}[\delta_{\zb j+\zb \ell}[f]]|L_p(\ell_{\theta}(\N_0^d))\|.
	\eeqq
	The relation in \eqref{eq:ijtriangle} shows
	\beqq
	\|2^{\zb r \cdot \zb j}q^1_{\zb j}[f]|L_p(\ell_{\theta}(\N_0^d))\|\lesssim \sum_{\zb b\in\{-1,0\}^d}\sum_{\zb \ell \geq 0}2^{-\zb \ell\cdot \zb r}\|2^{\zb r \cdot (\zb j+\zb \ell)}I^1_{\zb j+\zb b}[\delta_{\zb j+\zb \ell}[f]]|L_p(\ell_{\theta}(\N_0^d))\|.
	\eeqq
	Hence, Lemma \ref{lem:i1rep} yields
	\be
	\|2^{\zb r \cdot \zb j}q^1_{\zb j}[f]|L_p(\ell_{\theta}(\N_0^d))\|\lesssim 
	\sum_{\zb b\in\{-1,0\}^d}\sum_{\zb \ell \geq 0}2^{-\zb \ell\cdot \zb r}\|2^{\zb r 
		\cdot (\zb j+\zb \ell)}\mathcal{D}^1_{\zb j+\zb b}\ast \tilde{I}_{\zb j+\zb b}^2 [\delta_{\zb j+\zb \ell}[f]]|L_p(\ell_{\theta}(\N_0^d))\|.\label{Liz}
	\ee
	Lizorkin presented in \cite[p. 241, Thm.\ 5]{Li67} a theorem on Fourier multipliers for the $L_p(\ell_\theta)$ situation. The result in 
	\cite[Thm.\ 3.4.2]{ST87} transfers this to the periodic setting. Referring to a comment in 
	\cite[2.5.4]{Tr83}  the Fourier partial sum with respect to 
	a parallelepiped fulfills the requirements of this theorem and we get rid of $\mathcal{D}^1_{\zb j+\zb b}$ in \eqref{Liz}. This gives
	\beqq
	\|2^{\zb r \cdot \zb j}q^1_{\zb j}[f]|L_p(\ell_{\theta}(\N_0^d))\|&\lesssim& \sum_{\zb b\in\{-1,0\}^d}\sum_{\zb \ell \geq 0}2^{-\zb \ell\cdot \zb r}\|2^{\zb r 
		\cdot (\zb j+\zb \ell)} \tilde{I}_{\zb j+ \zb b}^2 \delta_{\zb j+\zb \ell}[f]|L_p(\ell_{\theta}(\N_0^d))\|.
	\eeqq
	Lemma \ref{lem:itjasij} with $\lambda = 1$ yields
	\beqq
	\|2^{\zb r \cdot \zb j}q^1_{\zb j}[f]|L_p(\ell_{\theta}(\N_0^d))\|&\lesssim& \sum_{\zb \ell \geq 0}2^{\zb \ell\cdot (a\zb 1-\zb r)}\|2^{\zb r \cdot (\zb j+\zb \ell)}M|P_{\zb 2^{\zb j+\zb \ell},a}f_{ j+\zb \ell}|(\zb x)|L_p(\ell_{\theta}(\N_0^d))\|.
	\eeqq
	We finish the proof by following the estimates in the proof of Theorem \ref{satz:bandlimpres2} beginning from \eqref{eq:estfgh}.
\end{proof}
\begin{satz}\label{satz:bandlimpresdirichletb}
	Let $1<p<\infty,0<  \theta\leq  \infty$  and $\zb r >\frac{1}{p}.$ 
	\begin{enumerate}
		\item Then every $f\in \srptb$  can be represented by
		$$ f=\sum_{\zb j\in \N_0^d}q^1_{\zb j}(f),$$
		with unconditional convergence in $\srptb$ in case $\theta<\infty,$
		and with unconditional convergence in $S^{\zb {\tilde{r}}}_{p,\nu}B(\T)$ for every $\zb r > \zb {\tilde{r}}$ and $0 <\nu\leq \infty$ in case $\theta=\infty.$ 
		\item There is a constant $C>0$ independent of $f$ such that
		\beqq\|2^{\zb r \cdot \zb j}q^1_{\zb j}[f]|\ell_{\theta}(L_p(\T))\|\leq C \|f|\srptb\|\eeqq
		holds for all $f\in\srptb$.
	\end{enumerate}
\end{satz}
\begin{proof}
	To prove (i) we follow the proof of Theorem \ref{satz:bandlimpresbesov}, (i). The assertion 
	(ii) can be obtained following the proof of Theorem 
	\ref{satz:bandlimpresdirichletf} where we replace $\|\cdot|L_p(\ell_\theta(\N_0))\|$ by 
	$\|\cdot|\ell_{\theta}(L_p(\T))\|$. Now we use the estimates in Remark \eqref{L_pbound}, \eqref{planch-polya} from Remark \ref{rem512}\,.
\end{proof}
\begin{rem}
	Similar (but not nested) Dirichlet kernels were studied in \cite{BDSU15} connected with sampling representations in case $p=\theta=2$.
\end{rem}
\section{Interpolation on Smolyak grids}
In this section we analyze a direction-wise modified version of Smolyak's algorithm, cf. \eqref{f0}, given by

\be T^{L,\zb \eta}_mf:=\sum_{\frac{1}{\eta_1}  \zb \eta\cdot \zb j\leq m}q^L_{\zb j}[f].\label{eq:anisosmolop}\ee
The parameter $\zb \eta >0$ allows to control the level of refinement in single directions. 
A comparatively large value of $\zb \eta$ in the $s$-th component ends up in a small refinement in the $s$-th direction. 
The interpolation operator $T_m^{L,\zb \eta}f$ maps a continuous function to a trigonometric polynomial with frequencies in an anisotropic hyperbolic cross \beqq AH^{d,\zb \eta}_m:=\bigcup_{\{\zb j:\;\frac{1}{\eta_1} \zb \eta \cdot \zb j\leq m\}}\mathcal{P}^0_{j}.\eeqq According to Lemma \ref{lem:triginterpol} the operator $T_m^{L,\zb \eta}$ interpolates functions on an anisotropic sparse grid \be AG^{d,\zb \eta}_m:=\bigcup_{\frac{1}{\eta_1}  \zb \eta\cdot \zb j\leq m}\Big\{\zb x_{\zb u}^{\zb j}:\;\zb u\in\Z,\;-2^{j_i-1}\leq u_i\leq 2^{j_i-1}-1,\;i\in [d]\Big\}\label{eq:anisogrid}.\ee
\begin{lem}\label{lem:ansiopnumberoffuncvalues}
	Let $\zb \eta\in\R$ with 
	$$ 0<\eta_1=\ldots=\eta_{\mu} <\eta_{\mu+1}\leq \ldots\leq \eta_{d}<\infty.$$ Then	
	\be|AG^{d,{\zb \eta}}_m|\asymp  m^{\mu-1}2^{m} \label{eq:anisogridnumber} \ee
	holds for all $m\geq1$.
\end{lem}
\begin{proof}
	Due to \eqref{eq:anisogrid} an upper bound for the cardinality of $AG^{d,{\zb \eta}}_m$ is provided by
	$$|AG^{d,{\zb \eta}}_m|\leq \sum_{\frac{1}{\eta_1}  \zb \eta \cdot \zb j \leq m}2^{|\zb j|_1}.$$
	Therefore Lemma \ref{lem:ansiopnumberoffuncvaluessum} in the appendix provides the upper bound in
	\eqref{eq:anisogridnumber}. A trivial lower bound of $2^m$ is provided by simply counting the sampling nodes of $q_{\zb
		j}[f]$ of the level $j=(m,0,\ldots,0)$. A sharp bound can be obtained  by using reproduction properties of $T_m^{L,\zb
		\eta}$ for trigonometric polynomials (cf. Lemma \ref{lem:trigreprod_hyp}) with frequencies in $AH^{d,\zb \eta}_m$. The
	dimension of $AH^{d,\zb \eta}_m$ is given by $\sum_{\frac{1}{\eta_1} \zb \eta \cdot \zb j \leq m}2^{|\zb j|_1}$.
\end{proof}
\begin{rem}
	Comparing this estimate to uniformly refined sparse grids  ($\zb \eta=\zb 1$) we recognize that the underlying dimension of the space plays no role for the asymptotic bound. The dimension dependence is replaced by the $\mu$ largest refinement directions. 
\end{rem}
\begin{satz}\label{satz:lqsampaniso}
	Let $0<p<q<\infty$ and $0< \theta\leq \infty$. Additionally let $L>\frac{1}{q}$ and the smoothness vector $ \zb r>\frac{1}{p}$ with \eqref{eq:smoothnessvector}. Then 
	\beqq
	\|f-T^{L,\zb \eta}_mf\|_q&\lesssim& 2^{-m(r_1-\frac{1}{p}+\frac{1}{q})} \|f|\srptf\|\\
	\eeqq
	holds for all $m> 0$.  The operator generating vector $\zb \eta\in\R$ 
	is chosen as $\zb \eta=\zb r-\frac{1}{p}+\frac{1}{q}.$
\end{satz}
\begin{proof} We start expanding $f$ into the series \eqref{eq:qkseries}. This allows us to estimate
	\beqq
	\|f-T_m^{L,\zb \eta}f\|_{q}&\leq& \Big\|\sum_{\frac{1}{\eta_1}  \zb \eta \cdot \zb j> m}|q_{\zb j}[f]|\Big\|_q
	\leq 2^{-(r_1-\frac{1}{p}+\frac{1}{q})m}\Big\|\sum_{\zb j \in \N_0^d}2^{(\zb r -\frac{1}{p}+\frac{1}{q})\cdot \zb j}|q_{\zb
		j}[f]|\Big\|_q.
	\eeqq
	We choose some parameters. Since $L>\frac{1}{q}+1$ we find $\tilde{q}\in\re$ with $p<\tilde{q}<q$ such that
	$L>\frac{1}{\tilde{q}}+1$ is fulfilled. Let ${\bf \tilde{r}}:=\zb r-\frac{1}{p}+\frac{1}{\tilde{q}}$. Applying Lemma
	\ref{lem:qjdiagembedding} yields
	\beqq
	\|f-T_m^{L,\zb \eta}f\|_{q}
	&\lesssim& 2^{-(r_1-\frac{1}{p}+\frac{1}{q})m}\Big\|\sup_{\zb j \in \N_0^d}2^{{\bf \tilde{r}}\cdot \zb{j}}|q_{\zb j}[f]|\Big\|_{\tilde{q}}
	\eeqq
	Theorem \ref{satz:bandlimpres2} yields
	\beqq \|f-T^{L,\zb \eta}_mf\|_{q}\lesssim 2^{-(r_1-\frac{1}{p}+\frac{1}{q})m}\|f|S^{\bf \tilde{r}}_{\tilde{q},\infty}F(\T)\|.\eeqq
	Finally, using the diagonal embedding stated in Lemma \ref{lem:embeddings}, (vi) gives
	\beqq \|f-T_m^{L,\zb \eta}f\|_{q}\lesssim 2^{-(r_1-\frac{1}{p}+\frac{1}{q})m}\|f|S^{\zb r}_{p,\theta}F(\T)\|,\eeqq
	which finishes the proof.
\end{proof}
For $\theta=2$ we can reproduce and generalize a result due to Temlyakov \cite{Tem93complexity}.
\begin{satz}\label{satz:linftysamplingmu}
	Let $0< p< \infty$, $0<\theta\leq \infty$.
	Additionally, let $L\geq 1$ and the smoothness vector $ \zb r>\frac{1}{p}$ with \eqref{eq:smoothnessvector}. Then 
	\beqq \|f-T^{L,\zb \eta}_mf\|_{\infty}&\lesssim &
	m^{(\mu-1)(1-\frac{1}{p})_+}2^{-m(r_1-\frac{1}{p})} \|f|\srptf\|
	\eeqq
	holds for all $m\geq 0$. The operator generating vector $\zb \eta\in\R$ is chosen as $\zb \eta={\zb \nu}-\frac{1}{p}$, where ${\zb \nu}\in\R$ with
	$$r_s=\nu_s,\quad s=1,\ldots,\mu\quad\mbox{and}\quad r_1<\nu_s< r_s,\quad s=\mu+1,\ldots,d.$$
\end{satz}
\begin{proof}
	{\em Step 1.} We prove 	\be
	\|f-T_m^{L,\zb \eta}f\|_{\infty}\lesssim \|f|S^{\zb r-\frac{1}{p}+\frac{1}{\tilde{p}}}_{\tilde{p},p}B(\T)\|\begin{cases}
		2^{-m(r_1-\frac{1}{p})}&:\quad 0<p\leq 1,\\
		m^{(\mu-1)(1-\frac{1}{p})}2^{-m(r_1-\frac{1}{p})}&:\quad p>1,
	\end{cases}\label{eq:h14}\ee
	where $\tilde{p}$ is chosen such that $\max\{p,1\}<\tilde{p}<\infty$ is fulfilled.
	Expanding into \eqref{eq:qkseries} and using triangle inequality yields
	\beqq
	\|f-T_m^{L,\zb \eta}f\|_{\infty}&=&\Big\|\sum_{\frac{1}{\eta_1}  \zb \eta \cdot \zb j> m}q^L_{\zb j}[f]\Big\|_{\infty}
	\leq\sum_{\frac{1}{\eta_1}  \zb \eta \cdot \zb j> m}\|q^L_{\zb j}[f]\|_{\infty}.
	\eeqq
	We have to distinguish the cases $0<p\leq 1$ and the case $p>1$. We start with $0<p\leq 1$. The elementary embedding $\ell_p(\N_0^d)\hookrightarrow\ell_1(\N_0^d)$ yields
	\beqq
	\|f-T_m^{L,\zb \eta}f\|_{\infty}&\leq&2^{-m(r_1-\frac{1}{p})}\sum_{\frac{1}{\eta_1}  \zb \eta \cdot \zb j> m}2^{(\zb r-\frac{1}{p})\cdot \zb j}\|q^L_{\zb j}[f]\|_{\infty}\leq 2^{-m(r_1-\frac{1}{p})}\Big(\sum_{\frac{1}{\eta_1}  \zb \eta \cdot \zb j> m}2^{p(\zb r-\frac{1}{p})\cdot \zb j}\|q^L_{\zb j}[f]\|_{\infty}^p\Big)^{\frac{1}{p}}.\\
	\eeqq	
	In case $p>1$ we apply H\"older's inequality with $1=\frac{1}{p}+\frac{1}{p'}$ and obtain
	\beqq
	\|f-T_m^{L,\zb \eta}f\|_{\infty}&\leq& \Big(\sum_{\frac{1}{\eta_1}  \zb \eta \cdot \zb j> m}2^{-p'(\zb r-\frac{1}{p})\cdot \zb j}\Big)^{\frac{1}{p'}}\Big(\sum_{\frac{1}{\eta_1}  \zb \eta \cdot \zb j> m}2^{p(\zb r-\frac{1}{p})\cdot\zb j}\|q^L_{\zb j}[f]\|_{\infty}^p\Big)^{\frac{1}{p}}.\\
	\eeqq
	Lemma \ref{lem:anisosmolsum} yields
	\beqq
	\|f-T_m^{L,\zb \eta}f\|_{\infty}\leq m^{(\mu-1)(1-p)}2^{-m(r_1-\frac{1}{p})}\Big(\sum_{\frac{1}{\eta_1}  \zb \eta \cdot \zb j> m}2^{p(\zb r-\frac{1}{p})\cdot \zb j}\|q^L_{\zb j}[f]\|_{\infty}^p\Big)^{\frac{1}{p}}.
	\eeqq
	Nikolskij's inequality (special case of Lemma \ref{lem:qjdiagembedding}) gives
	\beqq
	\|f-T_m^{L,\zb \eta}f\|_{\infty}\leq m^{(\mu-1)(1-p)}2^{-m(r_1-\frac{1}{p})}\Big(\sum_{\frac{1}{\eta_1}  \zb \eta \cdot \zb j> m}2^{p(\zb r-\frac{1}{p}+\frac{1}{\tilde{p}})\cdot \zb j}\|q^L_{\zb j}[f]\|_{\tilde{p}}^p\Big)^{\frac{1}{p}}.
	\eeqq
	In both cases Theorem \ref{satz:bandlimpresbesov} yields \eqref{eq:h14}.
	\\{\em Step 2.} The Jawerth-Franke type embedding implies
	$$\srptf\hookrightarrow S^{\zb r-\frac{1}{p}+\frac{1}{\tilde{p}}}_{\tilde{p},p}B(\T)$$ (cf. Lemma \ref{lem:jf}). Applying this we obtain
	\beqq 
	\|f-T_m^{L,\zb \eta}f\|_{\infty}\lesssim	m^{(\mu-1)(1-\frac{1}{p})_+}2^{-m(r_1-\frac{1}{p})} \|f|\srptf\|
	,
	\eeqq
	which proves the claim.
\end{proof}
\begin{rem}
	It is remarkable that Theorem \ref{satz:lqsampaniso} allows to use the Smolyak algorithm based on the 
	classical (nested) trigonometric interpolation (Dirichlet kernel) in case $1<q\leq \infty$ although $p<q$ may be less than one. 
	A similar observation has been made recently in \cite[Rem.\ 6.12]{BDSU15}.
\end{rem}

In the remainder of this section we deal with Besov spaces $\srptb$. 
A similar result as stated here was obtained by Dinh D\~ung in \cite{Di11}, see also \cite{Di01}. 
We contribute the case $\min\{p,\theta\}<1$ for the Fourier analytical approach and allow the Dirichlet kernel ($L=1$) for $q>1$.
\begin{satz}\label{satz:lqsampanisonikolskij}
	Let $0<p<q<\infty$ and $0<\theta\leq \infty$. Additionally let $L>\frac{1}{q}$ and the smoothness vector $ \zb r>\frac{1}{p}$ with \eqref{eq:smoothnessvector}. Then 
	\beqq
	\|f-T^{L,\zb \eta}_mf\|_q&\lesssim& 2^{-m(r_1-\frac{1}{p}+\frac{1}{q})}m^{(\mu-1)(\frac{1}{q}-\frac{1}{\theta})_+} \|f|\srptb\|\\
	\eeqq
	holds for all $m> 0$.  The operator generating vector $\zb \eta\in\R$ is chosen as $\zb \eta={\zb \nu}-\frac{1}{p}+\frac{1}{q}$, where ${\zb \nu}\in\R$ with
	$$r_s=\nu_s,\quad s=1,\ldots,\mu\quad\mbox{and}\quad r_1<\nu_s< r_s,\quad s=\mu+1,\ldots,d.$$
\end{satz}
\begin{proof}  First we prove the case $q>1$ with $\theta<\infty$. We find $\tilde{q}<q$ such that $L>\frac{1}{\tilde{q}}>\frac{1}{q}$ holds.
	The Jawerth-Franke embedding $S^{\frac{1}{\tilde{q}}-\frac{1}{q}}_{\tilde{q},q}B(\T)\subset L_q(\T)$ (cf. Lemma \ref{lem:jf}) yields
	$$\|f-T^{L,\zb \eta}f\|_{q}\leq \|f-T^{L,\zb \eta}f|S^{\frac{1}{\tilde{q}}-\frac{1}{q}}_{\tilde{q},q}B(\T)\|$$
	Expanding $f$ into the series \eqref{eq:qkseries} and applying Theorem \ref{satz:bandlimrep11} gives
	\be\|f-T^{L,\zb \eta}f\|_{q}\leq \Big(\sum_{\frac{1}{\eta_1}  \zb \eta \cdot \zb j> m}2^{q \zb j\cdot \zb 1 (\frac{1}{\tilde{q}}-\frac{1}{q})}\|q_{\zb j}[f]\|^q_{\tilde{q}}\Big)^{\frac{1}{q}}.\label{eq:continue1}\ee
	In case $\infty >\theta>q$ this can be estimated by using H\"older's inequality
	$$\|f-T^{L,\zb \eta}f\|_{q}\leq \Big(\sum_{\frac{1}{\eta_1}  \zb \eta \cdot \zb j> m}2^{-\frac{q\theta}{\theta-q} \zb j\cdot (\zb r-\frac{1}{p}+\frac{1}{q})}\Big)^{\frac{1}{q}-\frac{1}{\theta}} \Big(\sum_{\frac{1}{\eta_1}  \zb \eta \cdot \zb j> m} 2^{\theta\zb j\cdot (\zb r-(\frac{1}{p}-\frac{1}{\tilde{q}}))}\|q_{\zb j}[f]\|^{\theta}_{\tilde{q}}\Big)^{\frac{1}{\theta}}.$$
	The estimate for the sum in Lemma \ref{lem:anisosmolsum} gives
	$$\|f-T^{L,\zb \eta}f\|_{q}\leq 2^{-m(r_1-\frac{1}{p}+\frac{1}{q})}m^{(\mu-1)(\frac{1}{q}-\frac{1}{\theta})} \Big(\sum_{\frac{1}{\eta_1}  \zb \eta \cdot \zb j> m} 2^{\theta\zb j\cdot (\zb r-(\frac{1}{p}-\frac{1}{\tilde{q}}))}\|q_{\zb j}[f]\|^{\theta}_{\tilde{q}}\Big)^{\frac{1}{\theta}}.$$
	In case $\theta\leq q$ we use the embedding $\ell_{\theta}\hookrightarrow \ell_{q}$ and obtain
	$$\|f-T^{L,\zb \eta}f\|_{q}\leq 2^{-m(r_1-\frac{1}{p}+\frac{1}{q})}\Big(\sum_{\frac{1}{\eta_1}  \zb \eta \cdot \zb j> m} 2^{\theta\zb j\cdot (\zb r-(\frac{1}{p}-\frac{1}{\tilde{q}}))}\|q_{\zb j}[f]\|^{\theta}_{\tilde{q}}\Big)^{\frac{1}{\theta}}.$$
	Theorem \ref{satz:bandlimpresbesov} allows to estimate
	$$\|f-T^{L,\zb \eta}f\|_{q}\leq 2^{-m(r_1-\frac{1}{p}+\frac{1}{q})}m^{(\mu-1)(\frac{1}{q}-\frac{1}{\theta})_+}\|f|S^{\zb r - (\frac{1}{p}-\frac{1}{\tilde{q}})}_{\tilde{q},\theta}B(\T)\|.$$
	Finally, the diagonal embedding stated in Lemma \ref{lem:embeddings}, (vi) yields
	$$\|f-T^{L,\zb \eta}f\|_{q}\leq 2^{-m(r_1-\frac{1}{p}+\frac{1}{q})}m^{(d-1)(\frac{1}{q}-\frac{1}{\theta})_+}\|f|\srptb\|.$$
	The case $q\leq 1$ is simpler. We expand $f$ into the series \eqref{eq:qkseries}. Then $q$-triangle inequality yields
	$$\|f-T^{L,\zb \eta}_m f\|_{q}\leq \Big(\sum_{\frac{1}{\eta_1}  \zb \eta \cdot \zb j> m}\|q_{\zb j}[f]\|^q_{q}\Big)^{\frac{1}{q}}.$$ 
	The same case study as in the lines after \eqref{eq:continue1} with $\tilde{q}=q$ finishes the proof. As usual in case $\theta=\infty$ we have to replace the corresponding sum by $\sup$.
\end{proof}
\section{Linear sampling recovery}\label{sec:sharplin}
In this section we consider the optimality of convergence rates for linear sampling algorithms in case of Triebel-Lizorkin and H\"older-Nikolskij spaces with mixed smoothness, we abbreviate by \boldF. As a benchmark quantity we study linear sampling widths, cf. \eqref{samplingnumber} in the introduction,
$$\varrho_n^{\text{lin}}(\boldF,L_q(\T)):=\inf_{\substack{(\zb \xi_i)_{i=1}^n\subset \T\\(\psi_i)_{i=1}^n\subset L_q(\T)}}\sup_{\|f|\boldF\|\leq 1}\Big\|f-\sum_{i=1}^n f(\zb \xi_i)\psi_i\Big\|_q.$$
This quantity can be interpreted as the minimal worst case error for the approximation of functions from the unit ball of $\boldF$ by linear algorithms using $n$ function evaluations and where the error is measured in $L_q(\T)$. In case of $\boldF=\srptf$ with $\theta=2$ and $1<p<\infty$ 
we have the coincidence $\srpw=\srptf$. This case is of special interest in this section because it denotes the probably most famous representative of the $F$-scale. Choosing $m$ in \eqref{eq:anisosmolop} such that $n\gtrsim m^{\mu-1}2^m $ an upper bound for $\varrho_n^{\text{lin}}(\boldF,L_q(\T))$ is provided by
$$\varrho_n^{\text{lin}}(\boldF,L_q(\T))\lesssim \sup_{\|f|\boldF\|\leq 1}\|f-T^{L,\zb \eta}_{m}f\|_q.$$
Approximation with general linear information in case of mixed order Sobolev spaces $\svrpw$ and H\"older-Nikolskij spaces $\srpib$ has been intensively studied in the past. We recall the concept of linear $n$-widths:
$$\lambda_n(\boldF,L_q(\T)):=\inf_{\substack{A:\boldF\to L_q(\T)\\\rank A\leq n}}\sup_{\|f|\boldF\|\leq 1}\|f-A(f)\|_q.$$
In comparison to $\varrho_n^{\text{lin}}(\boldF,L_q(\T))$ this quantity allows to benchmark linear operators using $n$ pieces of linear information. Function evaluations are also linear information. Therefore, we have the relation
\beqq \lambda_n(\boldF,L_q(\T))\leq \varrho_n^{\text{lin}}(\boldF,L_q(\T)).\eeqq
That means linear $n$-widths can serve as lower bounds for linear sampling $n$-widths.

\begin{cor}\label{satz:lqsampanisosampnumb}
	Let $0<p<q<\infty$ and $0< \theta\leq \infty$. Additionally, the smoothness vector $ \zb r>\frac{1}{p}$ is supposed to satisfy \eqref{eq:smoothnessvector}.
	Then
	\beqq
	\varrho_n^{\text{lin}}(\srptf,L_q(\T))&\lesssim& (n^{-1}\log^{(\mu-1)}n)^{r_1-\frac{1}{p}+\frac{1}{q}}
	\eeqq
	holds for all $n> 0$.
\end{cor}
\begin{proof}
	The proof follows by Theorem \ref{satz:lqsampaniso} with the estimate from Lemma \ref{lem:ansiopnumberoffuncvalues} for the number of function evaluations used by $T^{L,\zb \eta}_m$.
\end{proof}
\begin{cor}\label{cor:sharp}
	Let $\zb r>\frac{1}{p}$ fulfilling \eqref{eq:smoothnessvector}.
	Furthermore, let $1<p<q\leq 2$, $1\leq \theta\leq \infty$ or $2\leq p<q<\infty$, $2\leq \theta\leq \infty$. Then
	$$\lambda_n(\srptf,L_q(\T))\asymp \varrho_n^{\text{lin}}(\srptf,L_q(\T))\asymp (n^{-1}\log^{(\mu-1)}n)^{(r_1-\frac{1}{p}+\frac{1}{q})}$$
	for all $n\in\N$.
\end{cor}
\begin{proof}
	The upper bound for $\varrho_n^{\text{lin}}$ follows by Corollary \ref{satz:lqsampanisosampnumb}. The lower bounds for $\lambda_n$ are provided in Theorem \ref{satz:linwidthsf}. 
\end{proof}
\begin{rem}
	The result stated above is not completely new. In case $2\leq p,\theta<q$, ($\theta=q$) and $1<p,q<2$ with $\theta<q$ the upper bounds can be obtained with the help of Besov space results proven by Dinh D{\~{u}}ng  in \cite{Di11,Di15} using the embedding relation $\srptf \hookrightarrow S^{\zb r}_{p,\max\{p,\theta\}}B(\T)$. Nevertheless, the cases $1<p<q<2$ with $\theta>q$ and $2\leq p<q<\theta$ are new. Compared to Besov spaces in that range of parameters we do not observe an additional logarithmic factor in the convergence rate. This parameter range includes the situation of Sobolev spaces in case $1<p<q<2$.
\end{rem}
The following result is based on an observation by Novak/Triebel \cite{NoTr05} for the univariate situation.
\begin{satz}\label{satz:sampineqapprox}Let $1<p<2<q<\infty$ and $\zb r>\begin{cases}
	\frac{1}{p}&:\; \frac{1}{p}+\frac{1}{q}\geq 1,\; 1\leq \theta\leq \infty,\\
	\max\{\frac{1}{p},1-\frac{1}{q}\}&:\;\frac{1}{p}+\frac{1}{q}\leq 1,\;1\leq \theta\leq \infty,
	\end{cases}$ \\with \eqref{eq:smoothnessvector}.
	Then 
	$$ \lambda_n(\srptf,L_q(\T)) = o(\varrho_n^{\text{lin}}(\srptf,L_q(\T))),$$
	or more precisely
	\beqq\lambda_n(\srptf,L_q(\T))  \lnsim   n^{-(r_1-\frac{1}{p}+\frac{1}{q})}\lesssim  \varrho_n^{\text{lin}}(\srptf,L_q(\T))\eeqq
	holds for all $n>0$.
\end{satz}
\begin{proof} The bounds for $\lambda_n(\srptf,L_q(\T))$ come from the embedding $\srptf\hookrightarrow S^{\zb r}_{p,\infty}B(\T)$ that yields
	\beqq
	\lambda_n(\srptf,L_q(\T))\leq \lambda_n(S^{\zb r}_{p,\infty}B(\T),L_q(\T)).
	\eeqq
	and the results from \cite{Ga96}, see also \cite[Thm.\ 4.46]{DTU16}. 
	The proof for the (non-sharp) lower bound of $\varrho_n^{\text{lin}}(\srptf,L_q(\T))$ follows from the univariate situation considered in \cite[Theorem 23]{NoTr05}. 
\end{proof}
\begin{rem}
	The fact that the exponents of the main rate and the exponent of the logarithm in the upper bound obtained in Corollary \ref{satz:lqsampanisosampnumb} coincide and additionally the main rate is sharp seems to be a strong indication for the conjecture
	
	$$\varrho_n^{\text{lin}}(\srptf,L_q(\T))\asymp (n^{-1}\log^{\mu-1}n)^{r_1-\frac{1}{p}+\frac{1}{q}}$$
	in case $1<p<q<\infty$, $1\leq \theta\leq \infty$ and $\zb r>\frac{1}{p}$ with \eqref{eq:smoothnessvector}.
\end{rem}
Sharp lower bounds for $\lambda_n(\srpib,L_q(\T))$  obtained in \cite{MaRu16} yield the following observation for H\"older-Nikolskij spaces.
\begin{cor}\label{cor:sharpboundnikolskij}
	Let $1<p<q\leq 2$ or $2\leq p<q<\infty$ and $ \zb r>\frac{1}{p}$ is supposed to satisfy \eqref{eq:smoothnessvector}.
	Then
	\beqq
	\varrho_n^{\text{lin}}(\srpib,L_q(\T))&\asymp&\lambda_n(\srpib,L_q(\T))\asymp(n^{-1}\log^{(\mu-1)}n)^{r_1-\frac{1}{p}+\frac{1}{q}}(\log n)^{\frac{\mu-1}{q}}
	\eeqq
	holds for all $n> 0$.
\end{cor}
\begin{proof}The upper bound was originally obtained by Dinh D\~ung in \cite{Di91}. The lower bound for linear widths 
	is due to Galeev \cite{Ga96}. In our context the upper bound for $\varrho_n^{\text{lin}}$ follows by Theorem \ref{satz:lqsampanisonikolskij} with the estimate from 
	Lemma \ref{lem:ansiopnumberoffuncvalues} for the number of function evaluations used by $T^{L,\zb \eta}_m$. 
	The lower bound for $\lambda_n$ in the second case was proven recently by Malykhin and Ryutin \cite{MaRu16},  
	see also \cite{Ga96} and \cite[Thm.\ 4.46]{DTU16}.
\end{proof}
\begin{cor}\label{satz:sampineqapproxbn}Let $1<p<2<q<\infty$ and $\zb r>\begin{cases}
	\frac{1}{p}&:\; \frac{1}{p}+\frac{1}{q}\geq 1,\\
	\max\{\frac{1}{p},1-\frac{1}{q}\}&:\;\frac{1}{p}+\frac{1}{q}< 1,
	\end{cases}$ \\fulfilling \eqref{eq:smoothnessvector}.
	Then 
	$$\lambda_n(\srpib,L_q(\T))=o(\varrho_n^{\text{lin}}(\srpib,L_q(\T))),$$
	or more precisely
	\beqq\lambda_n(\srpib,L_q(\T))  \lnsim   n^{-(r_1-\frac{1}{p}+\frac{1}{q})}\lesssim  \varrho_n^{\text{lin}}(\srpib,L_q(\T))\eeqq
	holds for all $n>0$.
\end{cor}
\begin{proof} The bounds for $\lambda_n(\srpib,L_q(\T))$ come from \cite{Ga96}.
	The proof for the (non-sharp) lower bounds for $\varrho_n^{\text{lin}}(\srpib,L_q(\T))$ follow from the univariate situation considered in \cite[Theorem 23]{NoTr05}. 
\end{proof}
\begin{cor}
	Let $0< p,\theta< \infty$ ($\theta=\infty$) and the smoothness vector $ \zb r>\frac{1}{p}$  which is supposed to satisfy \eqref{eq:smoothnessvector} be given. Then 
	$$\varrho_n^{\text{lin}}(\srptf,L_{\infty}(\T))\lesssim (n^{-1}\log^{\mu-1}n)^{r_1-\frac{1}{p}}	\log^{(\mu-1)(1-\frac{1}{p})_+}n$$
	holds for all $n>0$.	
\end{cor}
\begin{proof}
	The upper bound follows by Theorem \ref{satz:linftysamplingmu} with the estimate from Lemma \ref{lem:ansiopnumberoffuncvalues} for the  number of used function evaluations.
\end{proof}
Based on a recent observation of Nguyen in \cite[Theorem 2.15]{VKN16} we can state the following theorem:
\begin{cor}\label{cor:sampineqapprox}Let $1<p< 2$ and $\zb r>1$ fulfilling \eqref{eq:smoothnessvector}.
	Then 
	$$\lambda_n(\srpw,L_{\infty}(\T))=o(\varrho_n^{\text{lin}}(\srpw,L_{\infty}(\T))),$$
	or more precisely
	\beqq\lambda_n(\srpw,L_{\infty}(\T))\asymp n^{-(r_1-\frac{1}{2})}(\log^{\mu-1} n)^{r_1}  \lnsim   n^{-(r_1-\frac{1}{p})}\lesssim  \varrho_n^{\text{lin}}(\srpw,L_{\infty}(\T))\eeqq
	holds for all $n>0$.
\end{cor}
\begin{proof} The bound for $\lambda_n(\srpw,L_{\infty}(\T))$ comes from Theorem \ref{satz:VKN}. The proof for the (non-sharp) lower bound for $\varrho_n^{\text{lin}}(\srpw,L_{\infty}(\T))$ follows from the univariate situation considered in \cite[Theorem 23]{NoTr05}. 
\end{proof}
\section{Sampling recovery and Gelfand $n$-widths}\label{sec:sharpgelf}
The considerations above cover linear algorithms in the classical sense. 
Last but not least we consider an extension of this concept, so-called approximation using standard information, cf. \cite{NoWo08,NoWo10}. This means we consider algorithms that are defined as a composition of a linear information map and a possibly non-linear reconstruction operator. To avoid further technicalities we restrict to Banach spaces $\boldF$ that are either Sobolev spaces $\svrpw$ or H\"older-Nikolskij spaces $\srpib$ in this subsection.
The non-linear sampling widths were defined in \eqref{sampling}. The following relation clearly holds true
\beqq\varrho_n(\boldF,L_q(\T)\leq \varrho_n^{\text{lin}}(\boldF,L_q(\T))\,.\eeqq
Therefore (possibly non-sharp) upper bounds for sampling widths are always provided by linear sampling widths. To consider questions on optimality of these bounds we consider Gelfand $n$-widths
\begin{equation}\label{gelf}
	c_n(\boldF,L_q(\T)):=\inf_{\substack{B:\;\boldF \to \C^n\\\mbox{linear} }}\sup_{\substack{\|f|\boldF\|\leq 1\\f\in \ker B}}\|f\|_q.
\end{equation}
Here $B$ denotes a general linear mapping $B:\boldF \to \C^n$. 
This means $c_n$ measures the minimal (over all information mappings) worst case distance of elements in the unit ball of $\boldF$ which can not be distinguished by the information mapping $B$.
This immediately gives
$$c_n(\boldF,L_q(\T))\lesssim \varrho_n(\boldF,L_q(\T)).$$
Note that \eqref{gelf} is actually the definition of the $n$th ``Gelfand numbers'', which we call ``Gelfand $n$-width'' here. 
For a thorough discussion on the relation between Gelfand numbers and suitable worst-case errors we refer to the recent paper 
\cite[Rem.\ 2.3]{DiUl17}. Since Gelfand widths for embeddings $\id:\svrpw \to L_q(\T)$ are not studied directly we use a duality relation to Kolmogorov $n$-widths, cf. \eqref{def:kolmogorov}.
\begin{lem}\label{lem:gelfduality} The following duality relation holds true
	$$d_n(T:X\to Y) = c_n(T':Y'\to X'),$$
	where $T'$ denotes the adjoint operator of $T$ and $X'$, $Y'$ the topological dual spaces of $X$ and $Y$.
\end{lem}
\begin{proof}
	We refer to \cite[Theorem 6.2]{Pi85}.
\end{proof}
\begin{cor}\label{prop:gelfand}
	Let $1<p,q<\infty$ and $\zb r>\begin{cases}
	\frac{1}{2}&:\quad 1<p<q\leq2,\\
	1-\frac{1}{q}&:\quad p < 2 < q,\\
	(\frac{1}{p}-\frac{1}{q})_+&:\quad\mbox{otherwise,}
	\end{cases}$ \\with
	\begin{equation} r_1=\ldots=r_{\mu} <r_{\mu+1}\leq \ldots\leq r_{d}<\infty\label{eq:smoothvec}.\end{equation}
	Then
	$$c_n(\svrpw,L_q(\T))\asymp (n^{-1}\log^{\mu-1} n)^{r_1-(\min\{\frac{1}{p},\frac{1}{2}\}-\frac{1}{q})_+}$$
	for all $n\in\N$.
\end{cor}
\begin{proof}
	The proof follows by the duality relation stated in Lemma \ref{lem:gelfduality} and a lifting argument. The topological dual spaces of $\srpw$ and $L_q(\T)$ are the spaces $S^{ -\zb r}_{p'}W(\T)$ and $L_{q'}(\T)$ with $1=\frac{1}{p}+\frac{1}{p'}=\frac{1}{q}+\frac{1}{q'}$. Lemma \ref{lem:gelfduality}
	yields
	$$c_n(\svrpw,L_q(\T))=d_n(L_{q'}(\T),S^{-\zb r}_{p'}W(\T)).$$
	Finally we show the identity
	$$d_n(L_{q'}(\T),S^{-\zb r}_{p'}W(\T))\asymp d_n(S^{\zb r}_{q'}W(\T),L_{p'}(\T)).$$
	For that reason we consider the lifting operator $I_{\zb r}$ in $D'(\T)$ given by
	\beqq I_{\zb r}:f=\sum_{\zb k\in\Z}\widehat{f}(\zb k)e^{i\zb k \zb x} \mapsto \sum_{\zb k\in\Z}\widehat{f}(\zb k)\Big(\prod_{i=1}^d(1+|k_i|^{2})^{-\frac{r_i}{2}}\Big)e^{i\zb k \zb x}.\eeqq
	It is easy to check that this is an isometry that maps $f\in S^{\alpha}_pW$ to $I_{\zb r}f\in S^{\zb \alpha+\zb r}_pW$, $\alpha\in \re$ with $(I_{\zb r})^{-1}=I_{-\zb r}$. Therefore we may use the commutative diagram,
	\[
	\begin{diagram}
	\node{L_{q'}(\T)}
	\arrow{s,r}{I_{\zb r}}
	\arrow{e,t}{\id_1}
	\node{S^{-\zb r}_{p'}W(\T)}
	\\
	\node{S^{\zb r}_{q'}W(\T)}
	\arrow{e,t}{\id_2}
	\node{L_{p'}(\T)}
	\arrow{n,r}{I_{-\zb r}}
	\end{diagram}
	\]
	which allows to describe the operators $\id_1$, $\id_2$ by
	$$\id_1=I_{-\zb r} \circ \id_2\circ I_{\zb r}\quad\mbox{and}\quad\id_2=I_{\zb r} \circ\id _1 \circ I_{-\zb r}.$$
	Kolmogorov widths are $s$-numbers and fulfill a multiplicativity property that yields
	$$d_n(\id_1)=d_n(I_{-\zb r} \circ \id_2\circ I_{\zb r})\leq \|I_{-\zb r}\|d_n(\id_2)\|I_{\zb r}\|\asymp d_n(\id_2)$$
	and
	$$d_n(\id_2)=d_n(L_{\zb r} \circ \id_1\circ L_{- \zb{r}})\leq \|L_{\zb r}\|d_n(\id_1)\|I_{-\zb r}\|\asymp d_n(\id_1).$$
	Inserting the result from Theorem \ref{satz:kolmogorov} finishes the proof.
\end{proof}
Surprisingly, a new result in \cite{MaRu16} allows us to prove the following results for Gelfand $n$-widths of H\"older spaces $\srpib$.
\begin{satz}\label{satz:gelfandnikolskij}
	Let $1<p<q<\infty$ and $\zb r$ with
	\beqq (1/p - 1/q)_+<r_1=\ldots=r_{\mu} <r_{\mu+1}\leq \ldots\leq r_{d}<\infty\eeqq
	then
	$$c_n(S^{\zb r}_{p,\infty}B(\T),L_q(\T))\asymp \begin{cases}
	\left(\frac{\log^{\mu-1} n}{n}\right)^{r-\frac{1}{2}+\frac{1}{q}}(\log n)^{\frac{\mu-1}{q}}&:\quad \frac{1}{p}+\frac{1}{q}<1,\;p\leq 2,\;r_1>1-\frac{1}{q},\\
	\left(\frac{\log^{\mu-1} n}{n}\right)^{r-\frac{1}{p}+\frac{1}{q}}(\log n)^{\frac{\mu-1}{q}}&:\quad 2\leq p<q.
	\end{cases}$$
\end{satz}
\begin{proof}
	The upper bounds follow from the results for linear widths in \cite{Ga96}. The lower bounds are new. Malykhin and Ryutin proved in \cite{MaRu16} the following bound on Kolmogorov $n$-widths for finite dimensional normed spaces $\ell^M_p(\ell^N_q)$
	\be d_{\lfloor\frac{NM}{2}\rfloor}(\ell^M_{\infty}(\ell^N_{1}),\ell_{1}^M(\ell_{2}^N))\asymp M.\label{eq:maru16}\ee
	In the first case the technique for the lower bounds on linear widths presented in \cite{Ga96} works well also for Gelfand $n$-widths. The discretization stated there yields
	$$ c_n(\srpib,L_q(\T))\gtrsim 2^{u(-r+\frac{1}{2}-\frac{1}{q})}c_n(\ell_{\infty}^{u^{\mu-1}}(\ell_2^{2^u}),\ell_{q}^{u^{\mu-1}2^{u}}).$$
	The duality relation in Lemma \ref{lem:gelfduality} gives 
	$$ c_n(\srpib,L_q(\T))\gtrsim 2^{u(-r+\frac{1}{2}-\frac{1}{q})}d_n(\ell_{q'}^{u^{\mu-1}2^{u}},\ell_1^{u^{\mu-1}}(\ell_2^{2^u})).$$
	Applying H\"olders inequality in finite dimensional spaces $\ell^M_p(\ell^N_q)$ yields the following estimate
	$$ c_n(\srpib,L_q(\T))\gtrsim 2^{u(-r+\frac{1}{2}-\frac{1}{q})}u^{-\frac{\mu-1}{q'}}d_n(\ell_{\infty}^{u^{\mu-1}}(\ell_{1}^{2^{u}}),\ell_1^{u^{\mu-1}}(\ell_2^{2^u})).$$
	Choosing $n\asymp u^{\mu-1}2^{u}$ then the relation in \eqref{eq:maru16} implies
	$$ c_n(\srpib,L_q(\T))\gtrsim 2^{u(-r+\frac{1}{2}-\frac{1}{q})}u^{\frac{\mu-1}{q}}\asymp \left(\frac{\log^{\mu-1} n}{n}\right)^{r-\frac{1}{2}+\frac{1}{q}}(\log n)^{\frac{\mu-1}{q}}.$$
	The second case is obtained by the embedding
	$$S^{\zb r-(\frac{1}{2}-\frac{1}{p})}_{2,\infty}B(\T)\hookrightarrow S^{\zb r}_{p,\infty}B(\T)$$
	together with the result from the first case.
\end{proof}

\begin{cor}\label{cor:nonlinsampeqnonlinappr}
	Let $2\leq p<q<\infty$ and $\zb r>\frac{1}{p}$ fulfilling \eqref{eq:smoothnessvector}. Then
	\begin{enumerate}
		\item	 
		\beqq \varrho_n(\svrpw,L_q(\T))&\asymp& c_n(\svrpw,L_q(\T)) \asymp \varrho_n^{\text{lin}}(\svrpw,L_q(\T))\\&\asymp&
		\lambda_n(\svrpw,L_q(\T))\asymp (n^{-1}\log^{\mu-1}n)^{r_1-\frac{1}{p}+\frac{1}{q}},\eeqq
		\item	\beqq \varrho_n(\srpib,L_q(\T))&\asymp& c_n(\srpib,L_q(\T)) \asymp \varrho_n^{\text{lin}}(\srpib,L_q(\T))\\&\asymp&
		\lambda_n(\srpib,L_q(\T))\asymp (n^{-1}\log^{\mu-1}n)^{r_1-\frac{1}{p}+\frac{1}{q}}(\log^{\frac{\mu-1}{q}}n),\eeqq
	\end{enumerate}
	holds for all $n\in \N$.
\end{cor}
\begin{proof}
	The proof follows by Theorems \ref{satz:lqsampaniso}, \ref{satz:lqsampanisonikolskij}, \ref{satz:gelfandnikolskij} and Corollary \ref{prop:gelfand}.
\end{proof}
\begin{rem}
	In the parameter range $2<p<q<\infty$ permitting non-linear reconstruction operators does not yield better results. Optimal rates can be achieved by completely linear sampling algorithms.
\end{rem}
We obtain the following counterpart of Theorem \ref{satz:sampineqapprox} for non-linear sampling.
\begin{cor}\label{prop:sampineqproxnonlin}Let $1<p<2<q<\infty$ and $\zb r>\max\{\frac{1}{p},1-\frac{1}{q}\}$ fulfilling \eqref{eq:smoothnessvector}. Additionally let $\boldF$ denote either $\srpw$ or $\srpib$.
	Then 
	$$c_n(\boldF,L_q(\T))=o(\varrho_n(\boldF,L_q(\T))),$$
	or more precisely
	$$c_n(\boldF,L_q(\T))\lnsim n^{-(r_1-\frac{1}{p}+\frac{1}{q})}\lesssim \varrho_n(\boldF,L_q(\T))$$
	holds for all $n\in\N$.
\end{cor}
\begin{proof}
	The proof can be obtained by following the construction of the lower bound for the univariate situation in \cite{NoTr05}, where we recognize that the stronger inequality
	$$\varrho_n(\boldF,L_q(\T))\geq \inf_{(\xi_k)_{k=1}^n\subset \T} \sup_{\substack{\|f|\boldF\|\leq 1\\f(\xi_k)=0,\;k=1,\ldots,n}} \|f\|_q$$
	holds. The estimates for $c_n(\svrpw,L_q(\T))$ were obtained in Corollary \ref{prop:gelfand}. For $\srpib$ we refer to Theorem \ref{satz:gelfandnikolskij}.
\end{proof}
\begin{rem}
	As a consequence of the lower bound in Corollary \ref{prop:sampineqproxnonlin} for $\varrho_n(\boldF,L_q(\T))$, we obtain that in the parameter range $1<p<2<q<\infty$ even linear approximation behaves significantly better than sampling recovery with a possibly non-linear reconstructing operator.
\end{rem}

\setcounter{section}{0}
\renewcommand{\thesection}{\Alph{section}}
\renewcommand{\theequation}{\Alph{section}.\arabic{equation}}
\renewcommand{\thesatz}{\Alph{section}.\arabic{satz}}
\renewcommand{\thesubsection}{\Alph{section}.\arabic{subsection}}
\section{Appendix: Tools from Fourier analysis}
The following Lemma collects trivial properties of the Fourier transform and Fourier coefficients, cf. \eqref{eq:fouriertransform}.

\begin{lem}\label{lem:wienalg}
	Let $f\in L_1(\tor)$ with $\sum_{\ell\in\zz}|\widehat{f}(\ell)|<\infty$. Then
	$$f(\cdot)=\sum_{\ell\in\zz}\widehat{f}(\ell)e^{i\ell \cdot}$$
	in C(\tor).
\end{lem}
\begin{lem}[Poisson summation]\label{lem:poissonsum}
	Let $f\in L_1(\re)$. Then its periodization $\sum_{k\in \zz}f(\cdot+2\pi k)$ converges absolutely in the norm of $L_1([-\pi,\pi])$. Furthermore  its formal Fourier series is given by
	$$\sum_{k\in \zz}f(\cdot+2\pi k)=\frac{1}{\sqrt{2\pi}}\sum_{\ell\in\zz}\mathcal{F}f(\ell)e^{i\ell\cdot}
	$$
\end{lem}
\begin{proof}
	We refer to \cite[p. 252]{StWe71}.
\end{proof}

\begin{defi}[Hardy-Littlewood maximal function]\label{def:hl}
	Let $f\in L_1^{loc}(\tor^d)$. Then we define the Hardy-Littlewood maximal function as
	$$Mf(\zb x):=\sup_{Q\ni \zb x}\frac{1}{|Q|}\int_{Q}|f(\zb x)|d\zb x,$$
	where the $Q$ are cuboids centered in $\zb x$.
\end{defi}

\begin{satz}[Hardy-Littlewood maximal inequality]\label{satz:hardylittlewood}
	Let $1<p\leq \infty$ and $f\in L_p(\T)$. Then we have
	$$\|Mf|L_p(\T)\|\lesssim \|f|L_p(\T)\|.$$
\end{satz}
For the following periodic version of the Fefferman-Stein maximal inequality we refer to \cite[Prop. 3.2.4]{ST87}.
\begin{satz}[Fefferman-Stein maximal inequality]\label{thm:feffermanstein}
	Let $1<p<\infty$, $1<\theta\leq \infty$ and $(f_{\zb k})_{k}\subset L_p(\tor^d,\ell_{\theta})$. Then we have
	$$\|Mf_{k}|L_p(\tor^d,\ell_{\theta})\|\lesssim \|f_{k}|L_p(\tor^d,\ell_{\theta})\|.$$
\end{satz}
The following pointwise bound is proved in \cite[Lemma 3.3.1]{TDiff06}.
\begin{lem}\label{lem:diffversuspeetre}
	Let $a,b>0$ and
	$$f=\sum_{|k|<b}\hat{f}(k)e^{ ikx}:\quad x\in \tor$$
	be a univariate trigonometric polynomial with frequencies in $[-b,b]$. Then there exists a constant $C>0$ such that
	$$
	|\Delta_h^m(f,x)|\leq C\min\{1,|bh|^m\}\max\{1,|bh|^a\}P_{b,a}f(x)
	$$
	holds.
\end{lem}

\begin{lem}\label{lem:peetreenv}
	Let $a>0$, $b>0$ and $f\in C(\tor)$. 
	\begin{enumerate}
		\item If $|x-x_0|<\frac{1}{b}$ then
		$$|f(x_0)| \leq 2^a P_{b,a}f(x)$$
		holds.
		\item Furthermore let $b'>b>0$. Then
		$$P_{b,a}f(x) \leq \Big(\frac{b'}{b}\Big)^a P_{b',a}f(x).$$
	\end{enumerate}
\end{lem}
\begin{proof}
	The proof is an easy exercise playing with the definition of the Peetre maximal function. 
\end{proof}
\section{Appendix: Some multi-indexed geometric sums}

\begin{lem}\label{lem:anisosmolsum}
	Let $\zb r,\zb \eta\in\R$ with $0<r_1=\eta_1=\ldots=r_{\mu}=\eta_\mu<r_{\mu+1}\leq\ldots\leq r_{d}$ and $r_1<\eta_s<r_s$ for $s=\mu+1,\ldots,d$.  Then
	\beqq
	\sum_{\frac{1}{\eta_1}  \zb \eta \cdot \zb j > m}2^{-\zb r \cdot \zb j } \lesssim m^{\mu-1}2^{- r_1m} 
	\eeqq
	holds for all $m\geq 1$.
\end{lem}
\begin{proof} We refer to \cite[p. 9, Lemma B]{Tem86}.
\end{proof}

\begin{lem}\label{lem:ansiopnumberoffuncvaluessum}
	Let $\zb r\in\R$ with 
	\beqq  0<r_1=\ldots=r_{\mu} <r_{\mu+1}\leq \ldots\leq r_{d}<\infty\eeqq
	and $\mu\leq d$.
	Then	
	\beqq \sum_{\frac{1}{r_1}  \zb r \cdot \zb j \leq m}2^{|\zb j|_1}\asymp m^{\mu-1}2^{m} \eeqq
	holds for all $m\geq1$.
\end{lem}
\begin{proof}
	We refer to \cite[p. 10, Lemma D]{Tem86}.
\end{proof}
\section{Appendix: Known results on linear and Kolmogorov-widths}
\begin{satz} \label{satz:linwidths}
	Let $1 < p < \infty$, $1 \leq q<\infty$ and $\boldr>(1/p - 1/q)_+$ with \eqref{eq:smoothnessvector}\,.
	Then we have 
	\beqq \label{lambda_nW}
	\lambda_n(\svrpw,L_q(\T))
	\ \asymp \
	\begin{cases}
		\left(\frac{(\log n)^{\mu-1}}{n}\right)^{r_1 - (1/p - 1/q)_+} &:\quad  
		\quad q \le 2, \ \mbox{or} \quad p \ge 2,
		\\[1ex]
		\left(\frac{(\log n)^{\mu-1}}{n}\right)^{r_1 - 1/p + 1/2} &:\quad
		\quad 1/p + 1/q \ge 1, \ q > 2, \ \zb r > 1/p,
		\\[1ex]
		\left(\frac{(\log n)^{\mu-1}}{n}\right)^{r_1 - 1/2 + 1/q} &:\quad
		\quad 1/p + 1/q \le 1, \ p < 2, \ \zb r > 1 - 1/q.
	\end{cases}
	\eeqq
\end{satz}
\begin{proof} The case $1<q<\infty$ was proven by Galeev\cite{Ga87,Ga96}, see also \cite{EdTr89,EdTr92}. The case $q=1$ by Romanyuk \cite{Ro08}. Additionally we refer to \cite[Theorem 4.39]{DTU16} and the comments therein.
\end{proof}
\begin{satz} \label{satz:linwidthsf}
	Let $\zb r$ as in \eqref{eq:smoothvec}.
	Let additionally $1<p<q\leq 2$  and $1\leq \theta\leq \infty$ or $2\leq p<q<\infty$ and $\theta\geq 2$.
	Then we have 
	\beqq
	\lambda_n(S^{\zb r}_{p,\theta}F(\T),L_q(\T))
	\asymp 
	\left(\frac{(\log n)^{\mu-1}}{n}\right)^{r_1 - 1/p - 1/q}\quad,\\
	\eeqq
	for all $n\in\N$.
	
\end{satz}
\begin{proof} The upper bound can be obtained for instance by sampling recovery, cf. Theorem \ref{satz:lqsampanisosampnumb}. We focus on lower bounds. In case $\theta\geq 2$ the embedding $\srpw\hookrightarrow\srptf$ yields
	\beqq
	\lambda_n(\srptf,L_q(\T))\geq \lambda_n(\srpw,L_q(\T))
	\eeqq
	The results stated in Theorem \ref{satz:linwidths} provide the correct order. In case $\theta<p$ the embedding $\srptb\hookrightarrow\srptf$ yields
	\beqq
	\lambda_n(\srptf,L_q(\T))\geq \lambda_n(\srptb,L_q(\T)).
	\eeqq
	This gives the right order in cases $1<p<q\leq 2$ and $2\leq q<p$, cf. \cite{Ro01}.
	Finally for $\theta\geq p$ we stress on the embedding
	$$S^{\zb r}_{p,p}B(\T)\hookrightarrow\srptf$$ with
	\beqq
	\lambda_n(\srptf,L_q(\T))\geq \lambda_n(S^{\zb r}_{p,p}B(\T),L_q(\T)).
	\eeqq
	This provides the lower bound in case $1<p<q\leq 2$. We refer again to \cite{Ro01}.\end{proof}
The following is known for Kolmogorov widths in case of Sobolev spaces $\svrpw$ defined by
\be d_n(\svrpw,L_q(\T))=\inf_{\substack{A\subset L_q(\T)\\\dim A \leq n}} \sup_{\|f|\svrpw\|\leq 1} \inf_{g\in A} \|f-g\|_q.\label{def:kolmogorov}\ee
\begin{satz}\label{satz:kolmogorov}Let $1<p,q<\infty$ and $\zb r>\begin{cases}
	(\frac{1}{p}-\frac{1}{q})_+&:\quad 1\leq p\leq q \leq 2 \mbox{ or } 1\leq q \leq p <\infty,\\
	\max\{\frac{1}{2},\frac{1}{p}\}&:\quad otherwise,
	\end{cases}$ as in \eqref{eq:smoothvec}. Then
	$$d_n(\svrpw,L_q(\T))\asymp (m^{-1}\log^{\mu-1}m)^{r_1-(\frac{1}{p}-\max\{\frac{1}{2},\frac{1}{q}\})_+}.$$
\end{satz}
\begin{proof}
	The proof with every single case has a history of more than 20 years. For an overview we refer to \cite[Section 4.3]{DTU16}.
\end{proof}
\begin{satz}\label{satz:VKN}
	Let $1<p\leq 2$ and $\zb r>1$ satisfying \eqref{eq:smoothvec}.
	Then
	$$\lambda_n(\srpw,L_{\infty}(\T))\asymp n^{-(r_1-\frac{1}{2})}\log^{(\mu-1)r_1}n.$$
\end{satz}
\begin{proof}
	We refer to \cite[Theorem 2.14]{VKN16}.
\end{proof}
\section*{Acknowledgment}
The authors acknowledge the fruitful discussions with D.B. Bazarkhanov, Dinh D\~ung, A. Hinrichs, E. Novak 
and M. Ullrich on this topic. G.B and T.U. especially thank W. Sickel and V.N. Temlyakov (for asking about the Dirichlet kernel in the 
sampling representations) and the CRM Barcelona and S. Tikhonov for inviting them to the 
Intensive Research Program ``Constructive Approximation and Harmonic Analysis'' and the opportunity to speak about this work. 
 G.B. and T.U. gratefully acknowledge support by the German Research
Foundation (DFG) and the Emmy-Noether programme, Ul-403/1-1.

\end{document}